\documentclass[12pt]{amsart}
\usepackage{comment}
\usepackage[utf8]{inputenc} 
\usepackage[T1]{fontenc} 
\DeclareMathAlphabet{\mathpzc}{OT1}{pzc}{m}{it}

\usepackage{amsfonts}
\usepackage{amssymb}
\usepackage{amsthm}
\usepackage{amsmath}
\usepackage{amscd}
\usepackage[shortlabels]{enumitem}
\usepackage{mathrsfs}
\usepackage{tikz}
\usetikzlibrary{calc,arrows,decorations.pathreplacing}
\usepackage{nicefrac, xfrac}
\usepackage{mathtools}
\usepackage[bottom]{footmisc}
\usepackage[mathcal]{euscript}
\usepackage{float}

\usepackage[pagebackref, colorlinks = true,
linkcolor = red,
urlcolor  = blue,
citecolor = blue,
anchorcolor = blue]{hyperref}
\hypersetup{pdftitle={\@title},pdfauthor={\@author}}

\theoremstyle{plain}

\newtheorem{theorem}{Theorem}[section]
\newtheorem{corollary}[theorem]{Corollary}
\newtheorem{lemma}[theorem]{Lemma}

\newtheorem{proposition}[theorem]{Proposition}

\theoremstyle{definition}
\newtheorem{definition}[theorem]{Definition}
\newtheorem{observation}[theorem]{Observation}
\newtheorem{remark}[theorem]{Remark}

\newtheorem{thmx}{Theorem}

\renewcommand{\phi}{\varphi}
\renewcommand{\epsilon}{\varepsilon}

\newcommand{\sub}{\ensuremath{ \subseteq}} 
\newcommand{\bus}{\ensuremath{ \supseteq}} 
\newcommand{\set}[1]{\ensuremath{ \left\{#1\right\} }} 
\newcommand{\norm}[1]{\ensuremath{ \left\lVert#1\right\rVert }} 
\newcommand{\absval}[1]{\ensuremath{ \left\lvert#1\right\rvert }} 
\newcommand{\spot}{\ensuremath{ \makebox[1ex]{\textbf{$\cdot$}} }}

\newcommand{\ol}[1]{\overline{#1}}
\newcommand{\ul}[1]{\underline{#1}}
 
\newcommand{\bs}{\ensuremath{\backslash}}

\newcommand{\lt}{\ensuremath{\left}}
\newcommand{\rt}{\ensuremath{\right}}

\newcommand{\Om}{\ensuremath{\Omega}}
\newcommand{\Lm}{\ensuremath{\Lambda}}
\newcommand{\Gm}{\ensuremath{\Gamma}}

\newcommand{\Sg}{\ensuremath{\Sigma}}

\newcommand{\Dl}{\ensuremath{\Delta}}

\newcommand{\Ta}{\ensuremath{\Theta}}

\newcommand{\om}{\ensuremath{\omega}}
\newcommand{\lm}{\ensuremath{\lambda}}
\newcommand{\gm}{\ensuremath{\gamma}}  
\newcommand{\al}{\ensuremath{\alpha}} 
\newcommand{\bt}{\ensuremath{\beta}}  
\newcommand{\sg}{\ensuremath{\sigma}}
\newcommand{\ep}{\ensuremath{\epsilon}}
\newcommand{\dl}{\ensuremath{\delta}}
\newcommand{\kp}{\ensuremath{\kappa}}
\newcommand{\zt}{\ensuremath{\zeta}}
\newcommand{\ta}{\ensuremath{\theta}}

\newcommand{\vta}{\ensuremath{\vartheta}}
\newcommand{\vkp}{\ensuremath{\varkappa}}
\newcommand{\vrho}{\ensuremath{\varrho}}

\newcommand{\pzh}{\ensuremath{\mathpzc{h}}}

\DeclareSymbolFont{bbold}{U}{bbold}{m}{n}
\DeclareSymbolFontAlphabet{\mathbbold}{bbold}

\newcommand{\ind}{\ensuremath{\mathbbold{1}}}

\DeclareMathOperator{\HD}{HD}
\DeclareMathOperator{\diam}{diam}
\DeclareMathOperator{\supp}{supp}

\newcommand{\cA}{\ensuremath{\mathcal{A}}}
\newcommand{\cB}{\ensuremath{\mathcal{B}}}
\newcommand{\cC}{\ensuremath{\mathcal{C}}}
\newcommand{\cD}{\ensuremath{\mathcal{D}}}
\newcommand{\cE}{\ensuremath{\mathcal{E}}}
\newcommand{\cF}{\ensuremath{\mathcal{F}}}

\newcommand{\cI}{\ensuremath{\mathcal{I}}}

\newcommand{\cL}{\ensuremath{\mathcal{L}}}

\newcommand{\cP}{\ensuremath{\mathcal{P}}}

\newcommand{\cV}{\ensuremath{\mathcal{V}}}

\newcommand{\cX}{\ensuremath{\mathcal{X}}}

\newcommand{\cZ}{\ensuremath{\mathcal{Z}}}

\newcommand{\sA}{\ensuremath{\mathscr{A}}}
\newcommand{\sB}{\ensuremath{\mathscr{B}}}
\newcommand{\sC}{\ensuremath{\mathscr{C}}}

\newcommand{\sF}{\ensuremath{\mathscr{F}}}

\newcommand{\NN}{\ensuremath{\mathbb N}}

\newcommand{\RR}{\ensuremath{\mathbb R}}

\newcommand{\ZZ}{\ensuremath{\mathbb Z}} 

\usepackage{mathtools}

\newcommand{\floor}[1]{\ensuremath{\left\lfloor#1 \right\rfloor}}

\newcommand{\maeom}{$m$ a.e. $\om\in\Om$}
\newcommand{\flag}[1]{\textbf{***[#1]***}}

\def\lra{\longrightarrow}
\def\var{\text{{\rm var}}}
\def\BV{\text{{\rm BV}}}
\def\Leb{\text{{\rm Leb}}}

\def\lt{\left}
\def\rt{\right}

\providecommand{\phantomsection}{}
\AtBeginDocument{\let\textlabel\label}
\makeatletter
\newcommand{\mylabel}[2]{\raisebox{.7\normalbaselineskip}{\phantomsection}(#1)%
	\def\@currentlabel{#1}\textlabel{#2}}
\makeatother

\renewcommand\~{\tilde}
\newcommand{\hcL}{\ensuremath{\hat\cL}} 
\newcommand{\cEP}{\ensuremath{\cE P}} 
\allowdisplaybreaks

\numberwithin{equation}{section}
\title[]{Thermodynamic Formalism for Random Interval Maps with Holes}
\date{\today}

\author{Jason Atnip}
\address{School of Mathematics and Statistics, University of New South Wales, Sydney, NSW 2052, Australia}
\email{\href{j.atnip@unsw.edu.au}{j.atnip@unsw.edu.au} }

\author{Gary Froyland}
\address{School of Mathematics and Statistics, University of New South Wales, Sydney, NSW 2052, Australia}
\email{\href{g.froyland@unsw.edu.au}{g.froyland@unsw.edu.au} }

\author{Cecilia Gonz\'alez-Tokman}
\address{School of Mathematics and Physics, The University of Queensland, St Lucia, QLD 4072, Australia}
\email{\href{cecilia.gt@uq.edu.au}{cecilia.gt@uq.edu.au} }

\author{Sandro Vaienti}
\address{Aix Marseille Université, Université de Toulon, CNRS, CPT, 13009 Marseille, France}
\email{\href{vaienti@cpt.univ-mrs.fr}{vaienti@cpt.univ-mrs.fr} }

\begin{document}
	\begin{abstract}
	We develop a quenched thermodynamic formalism for open random dynamical systems generated by finitely branched, piecewise-monotone mappings of the interval.
	The openness refers to the presence of holes in the interval, which terminate trajectories once they enter;  the holes may also be random.
	Our random driving is generated by an invertible, ergodic, measure-preserving transformation $\sigma$ on a probability space $(\Omega,\mathscr{F},m)$.
	For each $\omega\in\Omega$ we associate a piecewise-monotone, surjective map $T_\omega:I\to I$, and a hole $H_\omega\subset [0,1]$;  the map $T_\omega$, the random potential $\varphi_\omega$, and the hole $H_\omega$ generate the corresponding open transfer operator $\mathcal{L}_\omega$.
	For a contracting potential, under a condition on the open random dynamics in the spirit of Liverani--Maume-Deschamps \cite{LMD}, we prove there exists a unique random probability measure $\nu_\omega$ supported on the survivor set ${X}_{\omega,\infty}$ satisfying $\nu_{\sigma(\omega)}(\mathcal{L}_\omega f)=\lambda_\omega\nu_\omega(f)$.
	Correspondingly, we also prove the existence of a unique (up to scaling and modulo $\nu$) random family of functions $q_\omega$ that satisfy $\mathcal{L}_\omega q_\omega=\lambda_\omega q_{\sigma(\omega)}$.
	Together, these provide an ergodic random invariant measure $\mu=\nu q$ supported on the global survivor set $\mathcal{X}_{\infty}$, while $q$ combined with the random closed conformal measure yields a unique random absolutely continuous conditional invariant measure (RACCIM) $\eta$ supported on $[0,1]$.
	Further, we prove quasi-compactness of the transfer operator cocycle generated by $\mathcal{L}_\omega$ and exponential decay of correlations for $\mu$.
	The escape rates of the random closed conformal measure and the RACCIM $\eta$ coincide, and are given by the difference of the expected pressures for the closed and open random systems. Finally, we prove that the Hausdorff dimension of the surviving set $X_{\omega,\infty}$ is equal to the unique zero of the expected pressure function for almost every fiber $\omega\in\Omega$.
	We provide several examples of our general theory.
	In particular, we apply our results to random $\beta$-transformations and random Lasota-Yorke maps, but our results also apply to the random non-uniformly expanding maps that are treated in \cite{AFGTV21}, such as intermittent maps and maps with contracting branches.
	
\end{abstract}

	\maketitle
	\tableofcontents

	\section{Introduction}
	Deterministic transitive dynamics $T:[0,1]\to[0,1]$ with enough expansivity enjoy a ``thermodynamic formalism'':  the transfer operator with a sufficiently regular potential $\varphi$ has a unique absolutely continuous invariant measure (ACIM) $\mu$, absolutely continuous with respect to the conformal measure $\nu$.
	Furthermore, $\mu$ arises as an equilibrium state, i.e.\ a maximiser of the sum of the integral of the potential $\phi$ and the metric entropy $h(\mu)$.
	Classical results in this direction include \cite{Sinai_72,bowen75,Ruelle_Thermodynamicformalism_1978} for shifts of finite type and smooth dynamics, and \cite{denker1990,LSV,buzzi-sarig03} for piecewise smooth dynamics.
	
	Continuing with the deterministic setting, if one introduces a hole $H\subset [0,1]$, which terminates trajectories when they enter, the situation becomes considerably more complicated.
	In the simplest case where the potential is the usual geometric potential $\varphi=-\log|T'|$, because of the lack of mass conservation, one expects at best an absolutely continuous \emph{conditional} invariant measure (ACCIM) $\mu$, conditioned according to survival from the infinite past. 
	Absolutely continuity is again with respect to a conformal measure $\nu$, which is supported on the survivor set $X_\infty$, the set of points whose infinite forward trajectories remain in $[0,1]$.
	Early work on the existence of the ACCIM and exponential convergence of non-equilibrium densities, includes \cite{vandenBedemChernov97,Colletetal00}.
	The paper \cite{LMD} handles general potentials that are contracting \cite{LSV} for the closed system, demonstrating exponential decay for $\mu$.
	There has been further work on the Lorentz gas and billiards \cite{DemersWrightYoung10,Demers13}, intermittent maps \cite{DemersFernandez14,DemersTodd17intermittent}, and multimodal maps \cite{DemersTodd17multimodal, DemersTodd20multimodal}.
	In the setting of diffeomorphisms with SRB measures, following the introduction of a hole, relations between escape rates and the pressures have been studied in \cite{DemersWrightYoung12}. 
	
	Looking to the fractal dimension of the surviving set $X_\infty$, the machinery of thermodynamic formalism was first employed by Bowen \cite{bowen_hausdorff_1979} to find the Hausdorff dimension of the limit sets of quasi-Fuchsian groups in terms of the pressure function, and then pioneered in the setting of open dynamical systems in \cite{urbanski_hausdorff_1986}. 
	
	In the random setting, repeated iteration of single deterministic map is replaced with composition of maps $T_\omega:X\to X$ drawn from a collection $\{T_\omega\}_{\omega\in\Omega}$.
	A driving map $\sigma:\Omega\to\Omega$ on a probability space $(\Omega,\sF,m)$ creates a map cocycle $T_\omega^n:=T_{\sigma^{n-1}\omega}\circ\cdots\circ T_{\sigma\omega}\circ T_\omega$.
	Thermodynamic formalism for random dynamical systems has been largely restricted to $T_\omega$ that are subshifts \cite{Bogenschutz_RuelleTransferOperator_1995a,gundlach96}, distance expanding \cite{mayer_distance_2011,simmons_relative_2013}, or continuous \cite{kifer92,Bogenschutz92,khanin-kifer96,stadlbauer_quenched_2020} (\cite{stadlbauer_quenched_2020} considers the sequential case where there is a single orbit of $\omega$).
	Countable Markov shifts have been treated in a series of papers beginning with \cite{denker-etal08} (see \cite{AFGTV21} for further references).
	The authors recently developed a complete, quenched thermodynamic formalism for random, countably-branched, piecewise monotonic interval maps \cite{AFGTV21}, enabling the treatment of discontinuous, non-Markov $T_\omega$.
	
	The situation of random open dynamics is relatively untouched.
	For a single piecewise expanding map $T:[0,1]\to[0,1]$ with holes $H_\omega$ randomly chosen in an i.i.d.\ fashion, \cite{BahsounVaienti13} consider escape rates for the annealed (averaged) transfer operator in the small hole limit (the Lebesgue measure of the $H_\omega$ goes to zero).
	In a similar setting, now assuming $T$ to be Markov and considering non-vanishing holes, \cite{Bahsounetal15} show existence of equilibrium states, again for the annealed transfer operator.
	In \cite{atnip_critically_2020}, the authors consider random, full-branched interval maps with negative Schwarzian derivative. The maps are allowed to have critical points, but the partition of monotonicity and holes, made up of finitely many open intervals, are fixed and non-random. 
	In this setting the existence of a unique invariant random probability measure is proven as well as a formula for the Hausdorff dimension of the surviving set. 
	In our current setting, we do not allow the existence of critical points, however our maps may have non-full branches, and our partitions of monotonicity as well as our holes are allowed to vary randomly from fiber to fiber.
	
	Sequential systems with holes have been considered in \cite{GeigerOtt21}, where a cocycle $T_\om$ of open maps is generated by a single $\omega$ orbit.
	The maps (which include the hole) must be chosen in a small neighborhood of a fixed map (with hole), in contrast to our setting where our cocycle may include very different maps. 
	Moreover in \cite{GeigerOtt21}, Lebesgue is used as a reference measure and the specific potential $-\log|\det DT|$ is used.
	The theory is developed for uniformly expanding maps in higher dimensions and the main goal is to establish the ``conditional memory loss'', a concept analogous to exponential decay of correlations for closed dynamics.


	In the present paper, we make a considerable advance over prior work by establishing a full, quenched thermodynamic formalism for piecewise monotonic random dynamics with general potentials and general driving---the random driving $\sigma$ can be any invertible ergodic process on $\Omega$.
	We begin with the random closed dynamics dealt with in \cite{AFGTV21}:  piecewise monotonic interval maps satisfying a random covering condition;  we have no Markovian assumptions, our maps may have discontinuities and may lack full branches.
	The number of branches of our maps need not be uniformly bounded above in $\omega$ and our potentials $\phi_\omega$ need not be uniformly bounded below or above in $\omega$.
	To this setting we introduce random holes $H_\omega$ and formulate sufficient conditions that guarantee a random conformal measure $\nu_\omega$ and corresponding equivariant measure $\mu_\omega$ supported on the random survivor set $X_{\om,\infty}$, and a random ACCIM $\eta_\omega$ supported on $H_\om^c$.
	These augment the notion of a random contracting potential \cite{AFGTV21} with accumulation rates of contiguous ``bad'' intervals (with zero conformal measure), and extend similar constructions \cite{LMD} to the random situation.
	
	To establish the existence of the family of measures $(\nu_\omega)_{\om\in\Om}$, we follow the limiting functional approach of \cite{LSV, LMD} by defining a random functional $\Lm_\om$ which is a limit of ratios of transfer operators and then showing that $\Lm_\om$ may be identified with the open conformal measure $\nu_\om$. This technique improves on the approach of \cite{AFGTV21}, which uses the Schauder-Tichonov Fixed Point Theorem to prove the existence of $\nu_{\om,c}$, by eliminating the extra steps necessary to show that the family $(\nu_\omega)_{\omega\in\Omega}$ is measurable with respect to $m$. 
	We establish exponential decay of correlations for $\mu$ and show that the escape rate of the closed conformal measure coincides with that of the RACCIM, and equals the difference of the expected pressures of the closed and open random systems. 
	We define the expected pressure function $t\mapsto\cE P(\phi_t)$ for the potential $\phi_{\om,t}=-t\log|T_\om'|$ and any $t\geq 0$, and then show that this function has a unique zero $h\in[0,1]$. Furthermore, we show that the Hausdorff dimension of the survivor set $X_{\om,\infty}$ is equal to $h$ for $m$-a.e. $\om\in\Om$.

	We apply our general theory to a large class of random $\bt$-transformations with random holes as well as a general random Lasota-Yorke maps with random holes.
	In particular, for the first time we allow both the maps and the holes to be random.
	In fact, our theory applies to all of the finitely-branched examples discussed in \cite{AFGTV21} (this includes maps which are non-uniformly expanding or have contracting branches which appear infrequently enough that we still maintain on-average expansion) when suitable conditions are put on the holes $H_\om$. This includes the case where $H_\om$ is composed of finitely many intervals and the number of connected components of $H_\om$ is $\log$-integrable with respect to $m$. 
	
	An outline of the paper is as follows. In Section~\ref{sec: prelim} we present formal definitions, properties, and assumptions concerning the closed and open random dynamics. The notion of random absolutely continuous conditionally invariant probability measures is introduced in Section~\ref{sec: RCIM}. In Section~\ref{sec: tr op and Lm} we introduce the random functional $\Lm_\om$ which we eventually show can be identified with the conformal measure on the open system. In this section we also present our main assumptions on the open dynamics as well as state our main results. Section~\ref{sec:cones} contains background material on Birkhoff cone techniques and the construction of our random cones. In Section~\ref{sec: LY ineq} we develop several random Lasota-Yorke type inequalities in terms of the variation and the random functional $\Lm_\om$. Section~\ref{sec:good} sees the construction of a large measure set of ``good'' fibers $\Om_G\sub\Om$ for which we obtain cone invariance at a uniform time step, and in Section~\ref{sec:bad} we show that the remaining ``bad'' fibers occur infrequently and behave sufficiently well. In Section~\ref{sec: props of Lm} we collect further properties of the random functional $\Lm$, which are then used in Section~\ref{sec: fin diam} to construct a large measure set of fibers $\Om_F\sub\Om$ for which we obtain cone contraction with a finite diameter image in a random time step. Using Hilbert metric contraction arguments, Section~\ref{sec: conf and inv meas} collects together the fruits of Sections~\ref{sec:good}-\ref{sec: fin diam} to prove our main technical lemma (Lemma~\ref{lem: exp conv in C+ cone}), which is then used to (i) obtain the existence of a random density $q$, (ii) prove the existence of a unique non-atomic random conformal measure $\nu$, and (iii) a random $T$-invariant measure $\mu$ which is absolutely continuous with respect to $\nu$. In Section~\ref{sec: dec of cor} we use the results of Section~\ref{sec: conf and inv meas} to show that the convergence to the random density $q$ happens exponentially quickly, which then implies that the random $T$-invariant measure $\mu$ satisfies an exponential decay of correlations. We also establish the existence of a unique random conditionally invariant measure $\eta$ which is absolutely continuous with respect to the closed conformal measure $\nu_c$, and establish the uniqueness of the measures $\nu$ and $\mu$ as well as the density $q$. The escape rate of the measure $\nu_c$ is given in Section~\ref{sec: exp press} in terms of the closed and open expected pressures. The expected pressure function is further developed and used to prove a Bowen's formula type result for the Hausdorff dimension of the survivor set $X_{\om,\infty}$ for $m$-a.e. $\om\in\Om$ in Section~\ref{sec: bowen}. Finally, in Section~\ref{sec: examples} we present detailed calculations for a large classes of random $\bt$-transformations with random holes and random Lasota-Yorke maps with random holes for which our results apply. 
		
	\section{Preliminaries of Random Interval Maps with Holes}\label{sec: prelim}
	Given a probability space $(\Om,\sF,m)$, we begin be considering an invertible, ergodic map $\sg:\Om\to\Om$ which preserves the measure $m$, i.e. 
	\begin{align*}
		m\circ\sg^{-1}=m.
	\end{align*} 
	We will refer to the tuple $(\Om,\sF,m,\sg)$ as the base dynamical system.
	Take $I$ to be a compact interval in $\RR$ and for each $\om\in\Om$ we consider the map
	 $T_\om:I\to I$ such that there exists a finite partition $\cZ_\om$ of $I$ such that 
	\begin{flalign}
		&T_\om:I\to I \text{ is surjective},
		\label{cond T1}\tag{T1}
		&\\
		&T_\om(Z) \text{ is an interval for each } Z\in\cZ_\om,
		\label{cond T2}\tag{T2}
		\\
		&T_\om\rvert_Z \text{ is continuous and strictly monotone for each } Z\in\cZ_\om.
		\label{cond T3}\tag{T3}
	\end{flalign}
	In addition, we will assume that 
	\begin{flalign}
		&\log\#\cZ_\om\in L^1(m).&
		\label{cond LIP}\tag{LIP}
	\end{flalign}
	The maps $T_\om$ induce the skew product map $T:\Om\times I\to \Om\times I$ given by 
	\begin{align*}
		T(\om,x)=(\sg(\om),T_\om(x)).
	\end{align*}
	For each $n\in\NN$ we consider the fiber dynamics of the maps $T_\om^n:I\to I$ given by the compositions
	\begin{align*}
		T_\om^n(x)=T_{\sg^{n-1}(\om)}\circ\dots\circ T_\om(x).
	\end{align*}
	We let $\cZ_{\om}^{(n)}$, for $n\geq 2$, denote the monotonicity partition of $T_\om^n$ on $I$ which is given by 
	\begin{align*}
		\cZ_\om^{(n)}:=\bigvee_{j=0}^{n-1}T_\om^{-j}\lt(\cZ_{\sg^j(\om)}\rt).
	\end{align*}
	Given $Z\in\cZ_\om^{(n)}$, we denote by
	$$	
	T_{\om,Z}^{-n}:T_\om^n(Z)\lra Z
	$$ 
	the inverse branch of $T_\om^n$ which takes $T_\om^n(x)$ to $x$ for each $x\in Z$.
	We will assume that the partitions $\cZ_\om$ are generating, i.e. 
	\begin{align}
		\bigvee_{n=1}^\infty \cZ_\om^{(n)}=\sB,
		\label{cond GP}\tag{GP}
	\end{align}
	where $\sB=\sB(I)$ denotes the Borel $\sg$-algebra of $I$. 
	Let $\cB(I)$ denote the set of all bounded real-valued functions on $I$ and for each $f\in\cB(I)$ and each $A\sub I$ let
	\begin{align*}
		\var_A(f):=\sup\set{\sum_{j=0}^{k-1}\absval{f(x_{j+1})-f(x_j)}: x_0<x_1<\dots x_k, \, x_j\in A \text{ for all } k\in\NN},
	\end{align*}  
	denote the variation of $f$ over $A$. We let 
	\begin{align*}
		\BV(I):=\set{f\in\cB(I): \var_I(f)<\infty}
	\end{align*}
	denote the set of functions of bounded variation on $I$. Let $\norm{f}_\infty$ and 
	$$
	\norm{f}_\BV:=\var(f)+\norm{f}_\infty
	$$ 
	be norms on the respective Banach spaces $\cB(I)$ and $\BV(I)$. Given a function $f:\Om\times I\to\RR$, by $f_\om: I\to I$ we mean 
	\begin{align*}
		f_\om(\spot):=f(\om,\spot).
	\end{align*}
	\begin{definition}\label{def: random bounded}
		We say that a function $f:\Om\times I\to\RR$ is \textit{random bounded} if 
		\begin{enumerate}[(i)]
			\item $f_\om\in\cB(I)$ for each $\om\in\Om$, 
			\item for each $x\in I$ the function $\Om\ni\om\mapsto f_\om(x)$ is measurable, 
			\item the function $\Om\ni\om\mapsto\norm{f_\om}_\infty$ is measurable. 
		\end{enumerate}
		Let $\cB_\Om(I)$ denote the collection of all random bounded functions on $\Om\times I$.
	\end{definition}
	\begin{definition}\label{def: random BV}
		We say that a function $f\in\cB_\Om(I)$ is of \textit{random bounded variation} if $f_\om\in\BV(I)$ for each $\om\in\Om$.
		We let $\BV_\Om(I)$ denote the set of all random bounded variation functions.
	\end{definition}
	For functions $f:\Om\times I\to\RR$ and $F:\Om\to\RR$ we let 
	\begin{align*}
		S_{n,T}(f_\om):=\sum_{j=0}^{n-1}f_{\sg^j(\om)}\circ T_\om^j
		\quad\text{ and }\quad
		S_{n,\sg}(F):=\sum_{j=0}^{n-1}F\circ\sg^j
	\end{align*}
	denote the Birkhoff sums of $f$ and $F$ with respect to $T$ and $\sg$ respectively. 
	We will consider a potential of the form 
	$\phi_c:\Om\times I\to\RR$ and for each $n\geq 1$ we consider the weight $g_c^{(n)}:\Om\times I\to\RR$ whose disintegrations are given by
	\begin{align}\label{def: formula for g_om^n}
		g_{\om,c}^{(n)}:=\exp(S_{n,T}(\phi_{\om,c}))
	\end{align} 
	for each $\om\in\Om$ \footnote{Throughout the text we will use the subscript $c$ notation to denote a quantity related to the closed dynamical system.}. We will often denote $g_{\om,c}^{(1)}$ by $g_{\om,c}$.
	We define the (Perron-Frobenius) transfer operator, $\cL_{\om,c}:\cB(I)\to \cB(I)$, with respect to the potential $\phi_c:\Om\times I\to\RR$, by 
	\begin{align*}
		\cL_{\om,c} (f)(x):=\sum_{y\in T_\om^{-1}(x)}g_{\om,c}(y)f(y); \quad 
		f\in \cB(I), \
		x\in I.
	\end{align*} 
	Inducting on $n$ gives that the iterates $\cL_{\om,c}^n:\cB(I)\to \cB(I)$ are given by 
	\begin{align}\label{def: formula for L_om^n}
		\cL_{\om,c}^n(f)(x)
		&=(\cL_{\sg^{n-1}(\om),c} \circ\dots\circ\cL_{\om,c})(f)(x)
		=
		\sum_{y\in T_\om^{-n}(x)}g_{\om,c}^{(n)}(y)f(y); \quad 
		f\in \cB(I), \,
		x\in I.
	\end{align}
	For each $\om\in\Om$ we let $\cB^*(I)$ and $\BV^*(I)$ denote the respective dual spaces of $\cB(I)$ and $\BV(I)$. We let $\cL_{\om,c}^*:\cB^*(I)\to\cB^*(I)$ denote the dual transfer operator.
	\begin{definition}\label{def: admissible potential}
		We will say that a measurable potential $\phi_c:\Om\times I\to\RR$ is \textit{admissible} if for $m$-a.e. $\om\in\Om$ we have 
		\begin{flalign} 
			& \inf \phi_{\om,c}, \sup\phi_{\om,c}\in L^1(m),
			\label{cond A1} \tag{A1} 	
			&\\
			& g_{\om,c}\in\BV(I). 
			\label{cond A2} \tag{A2} 
		\end{flalign}
	\end{definition}
	\begin{remark}\label{rem: A1 and A2 hold}
		Note that if $\phi_{\om,c}\in\BV(I)$ for each $\om\in\Om$ then \eqref{cond A2} is immediate. Furthermore, as $\phi_{\om,c}\in\cB(I)$ we also have that $\inf g_{\om,c}^{(n)}>0$ for $m$-a.e. $\om\in\Om$ and each $n\in\NN$.
	\end{remark}
	As an immediate consequence of \eqref{cond A1} we have that
	\begin{flalign}
		\log\inf g_{\om,c}, \log\norm{g_{\om,c}}_\infty\in L^1(m).
		\label{eq: cons of A1} 	
	\end{flalign} 
	Note that since we can write 
	\begin{align*}
		g_{\om,c}^{(n)}:=\prod_{j=0}^{n-1}g_{\sg^j(\om),c}\circ T_\om^j.
	\end{align*}
	for each $n\in\NN$ we must have that $g_{\om,c}^{(n)}\in\BV(I)$. 
	Clearly we have that the sequence $\norm{g_{\om,c}^{(n)}}_\infty$ is submultiplicative, i.e.
	\begin{align*}
		\norm{g_{\om,c}^{(n+m)}}_\infty\leq \norm{g_{\om,c}^{(n)}}_\infty\cdot\norm{g_{\sg^n(\om),c}^{(m)}}_\infty.
	\end{align*}
	Similarly we see that the sequence $\inf g_{\om,c}^{(n)}$ is supermultiplicative. Submultiplicativity and supermultiplicativity of $\norm{g_{\om,c}^{(n)}}_\infty$ and $\inf g_{\om,c}^{(n)}$ together with \eqref{eq: cons of A1} gives that 
	\begin{align}\label{eq: log sup inf g integ}
		\log\norm{g_{\om,c}^{(n)}}_\infty, \log\inf g_{\om,c}^{(n)}\in L^1(m)
	\end{align}
	for each $n\in\NN$.
	Our assumptions \eqref{cond T1} and \eqref{cond LIP} combined with \eqref{eq: log sup inf g integ} implies that 
	\begin{align}\label{eq: log sup inf L integ}
		\log\norm{\cL_{\om,c}^n\ind}_\infty, \log\inf \cL_{\om,c}^n\ind\in L^1(m)
	\end{align}
	for each $n\in\NN$.

	\subsection{Random Measures}\label{sec: rand meas}
	Given a measurable space $Y$, we let $\cP(Y)$ denote the collection of all Borel probability measures on $Y$.
	Recall that $\sB$ denotes the Borel $\sg$-algebra of $I$. Let $\sF\otimes\sB$ denote the product $\sg$-algebra of $\sB$ and $\sF$ on $\Om\times I$. 
	
	Let $\cP_m(\Om\times I)$ denote the set of all probability measures $\mu$ on $\Om\times I$ that have marginal $m$, i.e.  
	$$
	\cP_m(\Om\times I):=\set{\mu\in\cP(\Om\times I):\mu\circ\pi^{-1}_\Om=m},
	$$
	where $\pi_\Om:\Om\times X\to\Om$ is the projection onto the first coordinate. 
	\begin{definition}\label{def: random prob measures}
		A map $\mu:\Om\times\sB\to[0,1]$ with $\Om\times\sB\ni(\om,B)\mapsto \mu_\om(B)$ is said to be a \textit{random probability measure} on $I$ if
		\begin{enumerate}
			\item for every $B\in\sB$, the map $\Om\ni\om\mapsto\mu_\om(B)\in [0,1]$ is measurable, 
			\item for $m$-a.e. $\om\in\Om$, the map $\sB\ni B\mapsto\mu_\om(B)\in [0,1]$ is a Borel probability measure.
		\end{enumerate}
		We let $\cP_\Om(I)$ denote the set of all random probability measures on $I$. We will frequently denote a random measure $\mu$ by $(\mu_\om)_{\om\in\Om}$.
	\end{definition}
	
	The following proposition, which summarizes results of Crauel \cite{crauel_random_2002}, shows that the collection $\cP_m(\Om\times I)$ can be canonically identified with the collection $\cP_\Om(I)$ of all random probability measures on $I$.
	\begin{proposition}[\cite{crauel_random_2002}, Propositions 3.3, 3.6]\label{prop: random measure equiv}
		For each $\mu\in\cP_m(\Om\times I)$ there exists a unique random probability measure $(\mu_\om)_{\om\in\Om}\in\cP_\Om(I)$ such that 
		\begin{align*}
			\int_{\Om\times I} f(\om,x)\,d\mu(\om,x)
			=
			\int_\Om \int_I f(\om,x) \, d\mu_\om(x)\, dm(\om)
		\end{align*}
		for every bounded measurable function $f:\Om\times I\to\RR$.
		
		Conversely, if $(\mu_\om)_{\om\in\Om}\in\cP_\Om(I)$ is a random probability measure on $I$, then for every bounded measurable function $f:\Om\times I\to\RR$ the function 
		$$
		\Om\ni\om\longmapsto \int_I f(\om,x) \, d\mu_\om(x)
		$$ 
		is measurable and 
		$$
		\sF\otimes\sB\ni A\longmapsto\int_\Om \int_I \ind_A(\om,x) \, d\mu_\om(x)\, dm(\om)
		$$
		defines a probability measure in $\cP_m(\Om\times I)$.
	\end{proposition}
	\subsection{Further Assumptions on the Closed System}
	In addition to the assumptions \eqref{cond T1}-\eqref{cond T3}, \eqref{cond A1}, \eqref{cond A2}, \eqref{cond GP}, and \eqref{cond LIP} above, we assume the closed system satisfies the following:
	\begin{enumerate}
		\item[\mylabel{M}{cond M1}] The map $T:\Om\times I\to\Om\times I$ is measurable. 
		
		
		\item[\mylabel{C}{cond C1}]  There exist a random probability measure $\nu_c$ supported on $\Om\times I$ such that 
		\begin{align*}
			\nu_{\sg^n(\om),c}\lt(\cL_{\om,c}^n f\rt)
			=
			\lm_{\om,c}^n \nu_{\om,c}(f)
		\end{align*}
		for $m$-a.e. $\om\in\Om$, all $n\in\NN$, and all $f\in\BV(I)$, where 
		\begin{align*}
			\lm_{\om,c}^n:=\nu_{\sg^n(\om),c}\lt(\cL_{\om,c}^n\ind\rt)=\prod_{j=0}^{n-1}\nu_{\sg^j(\om),c}\lt(\cL_{\sg^j(\om),c}\ind\rt).
		\end{align*}
		Furthermore, $\log\lm_{\om,c}\in L^1(m)$.
	\end{enumerate}
	\begin{remark}
		Note that \eqref{cond M1} together with \eqref{cond A2} implies that for $f\in\BV_\Om(I)$ we have $\cL_c f\in\BV_\Om(I)$.
	\end{remark}
	\begin{remark}
		Examples which satisfy the conditions \eqref{cond T1}-\eqref{cond T3}, \eqref{cond A1}, \eqref{cond A2}, \eqref{cond GP}, \eqref{cond LIP}, \eqref{cond M1}, and \eqref{cond C1} can be found in \cite{AFGTV21}.
	\end{remark}
	
	\subsection{Random Maps with Holes}
	We now wish to introduce holes into the class of finite branched random weighted covering systems. 
	
	Let $H\sub \Om\times I$ be measurable with respect to the product $\sg$-algebra $\sF\otimes\sB$ on $\Om\times I$. 
	For each $\om\in\Om$ the sets $H_\om\sub I$ are uniquely determined by the condition that 
	\begin{align*}
		\set{\om}\times H_\om=H\cap\lt(\set{\om}\times I\rt).
	\end{align*}
	Equivalently we have 
	\begin{align*}
		H_\om=\pi_I(H\cap(\set{\om}\times I)),
	\end{align*}
	where $\pi_I:\Om\times I\to I$ is the projection onto the second coordinate. By definition we have that the sets $H_\om$ are $\nu_{\om,c}$-measurable. 
	Suppose that $0<\nu_c(H)<1$, that is we have 
	\begin{align*}
		0<\int_\Om\nu_{\om,c}(H_\om)\, dm(\om)<1.
	\end{align*}
	Now define
	\begin{align*}
		I_{\om}:=I\bs H_{\om}.
	\end{align*}
	and denote 
	\begin{align*}
		\ind_\om:=\ind_{I_\om}
	\end{align*}
	We then let
	\begin{align*}
		\cI:=H^c=\bigcup_{\om\in\Om}\set{\om}\times I_{\om}.
	\end{align*}
	
	For each $\om\in\Om$ and $n\geq 0$ we define
	\begin{align*}
		X_{\om,n}:&=\set{x\in I: T_\om^j(x)\notin H_{\sg^j(\om)} \text{ for all } 0\leq j\leq n }\\
		&=\set{x\in I_{\om}: T_\om^j(x)\in I_{\sg^j(\om)} \text{ for all } 0\leq j\leq n }\\
		&=\bigcap_{j=0}^n T_\om^{-j}\left(I_{\sg^j(\om)}\right)
	\end{align*}
	to be the set of points in $I_{\om}$ which survive for at least $n$ iterations, i.e. the set of points whose orbit does not land in a hole for at least $n$ iterates. We then naturally define
	\begin{align*}
		X_{\om,\infty}:=\bigcap_{n=0}^\infty X_{\om,n}=\bigcap_{n=0}^\infty T_\om^{-n}(I_{\sg^n(\om)})
	\end{align*}
	to be the set of points which will never land in a hole under iteration of the maps $T_\om$. We call $X_{\om,\infty}$ the \textit{$\om$-surviving set}. 
	By definition, for each $n\in\NN$, we have that 
	\begin{align*}
		T_\om(X_{\om,n})\sub X_{\sg(\om),n-1}
		\quad\text{ and }\quad
		T_\om(X_{\om,\infty})\sub X_{\sg(\om),\infty}.
	\end{align*}
	Note, however, that these survivor sets are, in general, only forward invariant and not backward invariant. 
	For notational convenience for any $0\leq \al\leq \infty$ we set
	\begin{align*}
		\hat{X}_{\om,\al}:=\ind_{X_{\om,\al}}.
	\end{align*}
	The global surviving set is defined as
	\begin{align*}
		\cX_{\al}:=\bigcup_{\om\in\Om}\set{\om}\times X_{\om,\al}
	\end{align*}
	for each $0\leq \al\leq \infty$.
	Then $\cX_{\infty}\sub \cI$ is precisely the set of points that survive under forward iteration of the skew-product map $T$.
	
	For each $\om\in\Om$ we let $\phi_{\om}=\phi_{\om,c}\rvert_{I_{\om}}$, and thus for each $n\in\NN$ this gives
	\begin{align*}
		g_{\om}^{(n)}=g_{\om,c}^{(n)}\rvert_{X_{\om,n-1}}=\exp\lt(S_{n,T}(\phi_\om)\rt)=\prod_{j=0}^{n-1}g_{\sg^j(\om)}\circ T_\om^j.
	\end{align*}
	Now define the open operator $\cL_{\om}:L^1(\nu_{\om,c})\to L^1(\nu_{\sg(\om),c})$, where $\nu_c$ comes from our assumption \eqref{cond C1} on the closed system, by
	\begin{align}\label{eq: def of open tr op}
		\cL_{\om}(f):=\cL_{\om,c}\left(f\cdot\ind_\om\right), \quad f\in L^1(\nu_{\om,c}), \, x\in I.
	\end{align}
	As a consequence of \eqref{eq: def of open tr op}, we have that 
	\begin{align*}
		\cL_\om\ind=\cL_\om\ind_\om.
	\end{align*}
	Iterates of the open operator $\cL_{\om}^n:L^1(\nu_{\om,c})\to L^1(\nu_{\sg^n(\om),c})$ are given by
	\begin{align*}
		\cL_{\om}^n:=\cL_{\sg^{n-1}(\om)}\circ\dots\circ\cL_{\om},
	\end{align*}
	which, using induction, we may write in terms of the closed operator $\cL_{\om,c}$ as
	\begin{align*}
		\cL_{\om}^n(f)=\cL_{\om,c}^n\left(f\cdot\hat{X}_{\om,n-1}\right), \qquad f\in L^1(\nu_{\om,c}).
	\end{align*}
	We define the sets $D_{\om,n}$ to be the support of $\cL_{\sg^{-n}(\om)}^n\ind_{\sg^{-n}(\om)}$, that is, we set
	\begin{align}\label{defn of D sets}
		D_{\om,n}:=\set{x\in I: \cL_{\sg^{-n}(\om)}^n\ind_{\sg^{-n}(\om)}(x)\neq 0}.
	\end{align}
	Note that, by definition, we have 
	\begin{align*}
		D_{\om,n+1}\sub D_{\om,n} 
	\end{align*}
	for each $n\in\NN$, and we similarly define
	\begin{align*}
		D_{\om,\infty}:=\bigcap_{n=0}^\infty D_{\om,n}.
	\end{align*}
		From this moment on we will assume that for $m$-a.e. $\om\in\Om$ we have that 
	\begin{align}
		&D_{\om,\infty}\neq\emptyset.
		\label{cond D}\tag{D}
	\end{align}
	We let
	\begin{align}\label{eq: hat D notation}
		\hat{D}_{\om,\al}:=\ind_{D_{\om,\al}}
	\end{align}
	for each $0\leq\al\leq \infty$. Since $D_{\om,n}$ is the support of $\cL_{\sg^{-n}(\om)}^n\ind_{\sg^{-n}(\om)}$, using the notation of \eqref{eq: hat D notation} we may write
	\begin{align*}
		\cL_{\om}^n(f)=\hat{D}_{\sg^n(\om),n}\cL_{\om}^n(f),
	\end{align*}
	More generally, we have that, for $k>j$, $D_{\sg^{k}(\om),j}$ is the support of $\cL_{\sg^{k-j}(\om)}^{j}\ind$, i.e.
	\begin{align}\label{support of pert tr op}
		\cL_{\sg^{k-j}(\om)}^{j}(f)=\hat{D}_{\sg^k(\om),j}\cL_{\sg^{k-j}(\om)}^{j}(f)
	\end{align}
	for $f\in L^1(\nu_{\sg^k(\om),c})$.
	Note that 
	\begin{align}\label{eq: D sets in terms of X sets}
		D_{\om,n}=T_{\sg^{-n}(\om)}^n(X_{\sg^{-n}(\om),n-1}).
	\end{align}
	Finally, we note that since $g_\om^{(n)}:= g_{\om,c}^{(n)}\rvert_{X_{\om,n-1}}$, for each $n\in\NN$ we have that
	\begin{align}\label{eq: ineq leads to log integ}
		\inf g_{\om,c}^{(n)}
		\leq 
		\inf_{X_{\om,n-1}} g_\om^{(n)}
		\leq
		\norm{g_\om^{(n)}}_\infty
		\leq 
		\norm{g_{\om,c}^{(n)}}_\infty
	\end{align}
	and 
	\begin{align}\label{eq: ineq leads to log integ2}
		\inf g_{\om,c}^{(n)}
		\leq 
		\inf_{D_{\sg^n(\om),\infty}} \cL_\om^n\ind_\om
		\leq 
		\norm{\cL_\om^n\ind_\om}_\infty
		\leq
		\norm{\cL_{\om,c}\ind}_\infty,
	\end{align}
	which, in conjunction with 
	\eqref{eq: log sup inf g integ} and \eqref{eq: log sup inf L integ}, imply that 
	\begin{align}\label{eq: log int open weight and tr op}
		\log\inf_{X_{\om,n-1}}g_\om^{(n)}, \,
		\log\norm{g_\om^{(n)}}_\infty, \,
		\log\inf_{D_{\sg^n(\om),N}}\cL_\om^n\ind_\om, \,
		\log\norm{\cL_\om^n\ind_\om}_\infty\in L^1(m).
	\end{align}
	Note that in light \eqref{eq: log int open weight and tr op}, the Birkhoff Ergodic Theorem implies that the quantities in \eqref{eq: ineq leads to log integ} and \eqref{eq: ineq leads to log integ2} are tempered, e.g. 
	\begin{align*}
		\lim_{|k|\to\infty}\frac{1}{|k|}\log\inf g_{\sg^k(\om)}^{(n)}=0
	\end{align*}
	for $m$-a.e. $\om\in\Om$ and each $n\in\NN$.
	
	\section{Random Conditionally Invariant Probability Measures}\label{sec: RCIM}
	In this section we introduce the notion of a random conditionally invariant measure and give suitable conditions for their existence. We begin with the following definition. 
	\begin{definition}
	We say that a random probability measure $\eta\in \cP_\Om(I)$ is a random conditionally invariant probability measure (RCIM) if
	\begin{align}\label{RCIM def}
		\eta_\om(T_\om^{-n}(A)\cap X_{\om,n})=\eta_{\sg^n(\om)}(A)\eta_\om(X_{\om,n})
	\end{align}
	for all $n\geq 0$, $\om\in\Om$, and all Borel sets $A\sub I$. If a RCIM $\eta$ is absolutely continuous with respect to a random probability measure $\zt$ we call $\eta$ a random absolutely continuous conditionally invariant probability measure (RACCIM) with respect to $\zt$. 
	\end{definition}
	Straight from the definition of a RCIM we make the following observations.
	\begin{observation}\label{O1 RCIM}
		Note that if we plug $A=I_{\om}=X_{\om,0}$ into \eqref{RCIM def} with $n=0$, we have that
		$$
		\eta_\om(I_{\om})=\eta_\om^2(I_{\om}),
		$$
		which immediately implies that $\eta_\om(I_\om)$ id either $0$ or $1$. If $\eta_\om(H_\om)=0$ then we have that $\eta_\om$ is supported in $I_{\om}$.
	\end{observation}
	\begin{observation}\label{O2 RCIM}
		Note that since
		\begin{align*}
			T_\om^{-n}(X_{\sg^n(\om),m})\cap X_{\om,n}=X_{\om,n+m}
		\end{align*}
		for each $n,m\in\NN$ and $\om\in\Om$ we have that if $\eta$ is a RCIM then
		\begin{align*}
			\eta_{\om}(X_{\om,n+m})=\eta_{\om}(X_{\om,n})\eta_{\sg^n(\om)}(X_{\sg^n(\om),m})
		\end{align*}
		for each $n,m\in\NN$. In particular, we have that
		\begin{align*}
			\eta_\om(X_{\om,n})=\prod_{j=0}^{n-1}\eta_{\sg^j(\om)}(X_{\sg^j(\om),1}).
		\end{align*}
	\end{observation}
	In light of Observation~\ref{O2 RCIM}, given a RCIM $\eta$, for each $\om\in\Om$ we let
	$$
	\al_{\om}:=\eta_\om(X_{\om,1}).
	$$
	Thus we have
	\begin{align}\label{eq: def of al_om}
		\al_{\om}^n:=\prod_{j=0}^{n-1}\al_{\sg^j(\om)}=\eta_\om(X_{\om,n}).
	\end{align}
	We now prove a useful identity. 
	\begin{lemma}\label{lem: useful identity}
		Given any $f,h\in\BV(I)$, any $\om\in\Om$, and $n\in\NN$  we have that 
		\begin{align*}
			\int_{I_{\sg^n(\om)}} h \cdot \cL_\om^n f \,d\nu_{\sg^n(\om),c}
			&=
			\lm_{\om,c}^n\int_{X_{\om,n}}f\cdot h\circ T_\om^n\, d\nu_{\om,c}.
		\end{align*}
	\end{lemma}
	\begin{proof}
		To prove the identity we calculate the following:
		\begin{align*}
			\int_{I_{\sg^n(\om)}} h \cdot \cL_\om^n f \,d\nu_{\sg^n(\om),c}
			&=
			\int_I h \cdot\ind_{\sg^n(\om)}\cdot\cL_{\om,c}^n\lt(f\cdot \hat X_{\om,n-1}\rt) \, d\nu_{\sg^n(\om),c}
			\\
			&=
			\int_I \cL_\om^n\lt(f (h\circ T_\om^n) (\ind_{\sg^n(\om)}\circ T_\om^n)\cdot \hat X_{\om,n-1}\rt)\, d\nu_{\sg^n(\om),c}
			\\
			&=
			\lm_{\om,c}^n \int_I f (h\circ T_\om^n)\cdot \hat X_{\om,n}\, d\nu_{\om,c}
			\\
			&=\lm_{\om,c}^n\int_{X_{\om,n}}f\cdot h\circ T_\om^n\, d\nu_{\om,c}.	 	
		\end{align*}
	\end{proof}
	The following lemma gives a useful characterization of RACCIM (with respect to $\nu_c$) in terms of the transfer operators $\cL_{\om}$.
	\begin{lemma}\label{LMD lem 1.1.1}
		Suppose $\eta=\ind_{\cI}h\nu_{c}$ is a random probability measure on $\cI$ absolutely continuous with respect to $\nu_{c}$, whose disintegrations are given by
		$$
		\eta_{\om}=\ind_\om h_{\om}\nu_{\om,c}.
		$$
		Then $\eta$ is a RACCIM (with respect to $\nu_c$) if and only if there exists $\al_{\om}>0$ such that 
		\begin{align}\label{eq1 lem iff accipm}
			\cL_{\om}h_{\om}=\lm_{\om,c}\al_{\om}h_{\sg(\om)}
		\end{align}
		for each $\om\in\Om$.
	\end{lemma}
	\begin{proof}
		Beginning with the ``reverse'' direction, we first suppose \eqref{eq1 lem iff accipm} holds for all $\om\in\Om$. Let $A\in\sB$ (Borel $\sg$-algebra). Using Lemma~\ref{lem: useful identity} gives 
		\begin{align}
			\eta_{\om}(T_\om^{-n}(A)\cap X_{\om,n})&=\int_{X_{\om,n}}(\ind_{A}\circ T_\om^n)\cdot h_{\om}\, d\nu_{\om,c}
			\nonumber\\
			&=(\lm_{\om,c}^n)^{-1}\int_{I_{\sg^n(\om)}}\ind_A\cdot \cL_{\om}^nh_{\om}\, d\nu_{\sg^n(\om),c}
			\nonumber\\
			&=\int_{I_{\sg^n(\om)}}\ind_A\cdot\al_{\om}^nh_{\sg^n(\om)}\, d\nu_{\sg^n(\om),c}
			\nonumber\\
			&=\al_{\om}^n\eta_{\sg^n(\om)}(A).\label{LD lem 1.1 eq1}
		\end{align}
		Inserting $A=I_{\sg^n(\om)}$ into \eqref{LD lem 1.1 eq1} gives
		\begin{align*}
			\eta_{\om}(T_\om^{-n}(I_{\sg^n(\om)})\cap X_{\om,n})
			&=\al_{\om}^n\eta_{\sg^n(\om)}(I_{\sg^n(\om)}).
		\end{align*}
		Observation \ref{O1 RCIM} implies that $\eta_{\sg^n(\om)}(I_{\sg^n(\om)})=1$, and thus 
		\begin{align*}
			\al_{\om}^n
			&=\eta_{\om}(T_\om^{-n}(I_{\sg^n(\om)})\cap X_{\om,n})
			=\eta_\om(X_{\om,n}),
		\end{align*}
		since $T_\om^{-n}(I_{\sg^n(\om)})\cap X_{\om,n}=X_{\om,n}$.
		Thus,  for $A\in\sB$ we have
		\begin{align*}
			\eta_{\om}(T_\om^{-n}(A)\cap X_{\om,n})=\eta_{\sg^n(\om)}(A)\eta_{\om}(X_{\om,n})
		\end{align*}
		as desired.
		
		Now to prove the opposite direction, suppose $\eta_{\om}(\ind_\om h_{\om})$ is a RACCIM. Then by the definition of a RCIM there exists $\al_{\om}$ such that for any $A\in\sB$ we have
		\begin{align*}
			\eta_{\om}(T_\om^{-n}(A)\cap X_{\om,n})=\al_{\om}^n\eta_{\sg^n(\om)}(A).
		\end{align*}
		So we calculate
		\begin{align*}
			(\lm_{\om,c}^n)^{-1}\int_{I_{\sg^n(\om)}}\ind_A\cdot\cL_{\om}^nh_{\om}\, d\nu_{\sg^n(\om),c}
			&=\int_{X_{\om,n}}(\ind_A\circ T_\om^n)h_{\om}\, d\nu_{\om,c}
			=\eta_{\om}(T_\om^{-n}(A)\cap X_{\om,n})
			\\
			&
			=\al_{\om}^n\eta_{\sg^n(\om)}(A)
			=\al_{\om}^n\int_{I_{\sg^n(\om)}}\ind_A\cdot h_{\sg^n(\om)}\, d\nu_{\sg^n(\om),c}.	 	
		\end{align*}
		So we have
		\begin{align*}
			\cL_{\om}^nh_{\om}=\lm_{\om,c}^n\al_{\om}^nh_{\sg^n(\om)},
		\end{align*}
		which completes the proof.
	\end{proof}

	\section{Functionals and Partitions}\label{sec: tr op and Lm}
	In this section we follow \cite{LSV, LMD} and introduce the random functional $\Lm_\om$ that we will later show is equivalent to the conformal measure for the open system. We also introduce certain refinements of the partition of monotonicity which are used to define ``good'' and ``bad'' intervals and are needed to state our main assumptions on the open system. Following the statement of our main hypotheses, we state our main results.  
	
	We begin by defining the functional $\Lm_{\om}:\BV(I)\to\RR$ by
	\begin{align}\label{eq: def of Lm}
		\Lm_{\om}(f):=\lim_{n\to\infty}\inf_{x\in D_{\sg^n(\om),n}}\frac{\cL_{\om}^n(f)(x)}{\cL_{\om}^n(\ind_\om)(x)}, \quad f\in\BV(I).
	\end{align}
	We note that this limit exists as the sequence is bounded and increasing. Indeed, we have that
	\begin{align}\label{eq: sup norm bound ratio}
		-\norm{f}_{\infty}\leq \inf_{x\in D_{\sg^n(\om),n}}\frac{\cL_{\om}^n(f)(x)}{\cL_{\om}^n(\ind_\om)(x)} \leq \norm{f}_{\infty},
	\end{align}
	and to see that the ratio is increasing we note that
	\begin{align}
		&\inf_{x\in D_{\sg^{n+1}(\om),n+1}}
		\frac{\cL_{\om}^{n+1}(f)(x)}
		{\cL_{\om}^{n+1}(\ind_\om)(x)}
		=
		\inf_{x\in D_{\sg^{n+1}(\om),n+1}}
		\frac{\cL_{\sg^n(\om)}\lt(\hat D_{\sg^n(\om),n}\cdot \cL_{\om}^n(f)\rt)(x)}
		{\cL_{\om}^{n+1}(\ind_\om)(x)}
		\nonumber\\
		&\quad=
		\inf_{x\in D_{\sg^{n+1}(\om),n+1}}
		\frac{\cL_{\sg^n(\om)}\lt(\hat D_{\sg^n(\om),n}\cdot \cL_{\om}^n(\ind_\om)
			\cdot\frac{\cL_{\om}^n(f)}{\cL_{\om}^n(\ind_\om)}\rt)(x)}
		{\cL_{\om}^{n+1}(\ind_\om)(x)}
		\nonumber\\
		&\quad\geq
		\inf_{x\in D_{\sg^{n}(\om),n}}
		\frac{\cL_{\om}^{n}(f)(x)}
		{\cL_{\om}^{n}(\ind_\om)(x)}
		\cdot
		\inf_{x\in D_{\sg^{n+1}(\om),n+1}}
		\frac{\cL_{\sg^n(\om)}\lt(\hat D_{\sg^n(\om),n}\cdot \cL_{\om}^n\lt(\ind_\om\rt)\rt)(x)}
		{\cL_{\om}^{n+1}(\ind_\om)(x)}
		\nonumber\\
		&\quad=
		\inf_{x\in D_{\sg^{n}(\om),n}}
		\frac{\cL_{\om}^{n}(f)(x)}
		{\cL_{\om}^{n}(\ind_\om)(x)}.	
		\label{rho^n increasing}
	\end{align}
	In particular, \eqref{eq: sup norm bound ratio} of the above argument gives that
	\begin{align}\label{eq: Lm bdd in sup norm}
		-\norm{f}_{\infty}\leq\inf f\leq \Lm_{\om}(f) \leq \norm{f}_{\infty}.
	\end{align}
	\begin{observation}\label{obs: properties of Lm}
		One can easily check that the functional $\Lm_{\om}$ has the following properties.
		\begin{enumerate}
			\item[\mylabel{1}{Lm prop1}] $\Lm_\om(\ind)=\Lm_{\om}(\ind_\om)=1$.
			\item[\mylabel{2}{Lm prop2}] $\Lm_{\om}$ is continuous with respect to the supremum norm.
			\item[\mylabel{3}{Lm prop3}] $f\geq h$ implies that $\Lm_{\om}(f)\geq \Lm_{\om}(h)$.
			\item[\mylabel{4}{Lm prop4}] $\Lm_{\om}(cf)=c\Lm_{\om}(f)$.
			\item[\mylabel{5}{Lm prop5}] $\Lm_{\om}(f+h)\geq \Lm_{\om}(f)+\Lm_{\om}(h)$.
			\item[\mylabel{6}{Lm prop6}] $\Lm_{\om}(f+a)=\Lm_{\om}(f)+a$ for all $a\in\RR$.
			\item[\mylabel{7}{Lm prop7}] If $A\cap X_{\om,n}=\emptyset$ for some $n\in\NN$ then $\Lm_{\om}(\ind_A)=0$.
		\end{enumerate}
		Furthermore, we note that the homogeneity \eqref{Lm prop4} and super-additivity \eqref{Lm prop5} imply that $\Lm_{\om}$ is convex. In the sequel, we will show that $\Lm_{\om}$ is in fact linear, and can thus be associated with a unique probability measure on $I_{\om}$ via the Riesz Representation Theorem.
	\end{observation}
	\begin{remark}
		Let $f\in\BV(I)$, then for all $x,y\in I_{\om}$ we have
		\begin{align*}
			f(x)\leq f(y)+\var(f).
		\end{align*}
		Using property \eqref{Lm prop2} of $\Lm_{\om}$, together with \eqref{eq: sup norm bound ratio}, implies
		\begin{align}
			f(x)
			\leq 
			\inf f+\var(f)
			\leq
			\Lm_{\om}(f)+\var(f)
			\leq 
			\norm{f}_{\infty} +\var(f).
			\label{f leq Lm(f)+var(f)}
		\end{align}
	\end{remark}
	We set
	\begin{align}\label{rho def}
		\rho_{\om}:=\Lm_{\sg(\om)}\left(\cL_{\om}(\ind_\om)\right).
	\end{align}
	The following propositions concern various estimates of $\rho_{\om}$. We begin by setting
	\begin{align}\label{rho^n def}
		\rho_{\om}^{(n)}:=\inf_{x\in D_{\sg^{n+1}(\om),n}}\frac{\cL_{\sg(\om)}^n(\cL_{\om}(\ind_\om))(x)}{\cL_{\sg(\om)}^{n}(\ind_{\sg(\om)})(x)}.
	\end{align}
	Then, by the definition and \eqref{rho^n increasing}, we have that
	\begin{align}\label{rho^n inc converge to rho}
		\rho_{\om}^{(n)}\nearrow\Lm_{\sg(\om)}(\cL_{\om}(\ind_\om))=\rho_{\om}
	\end{align}
	as $n\to\infty$.
	\begin{remark}\label{rem: rho inequalities}
		Note that \eqref{rho def} and the definition of $\cL_\om$ together immediately imply that
		\begin{align}\label{rho rough up and low bdd}
			\inf_{I_\om} g_\om\leq \rho_{\om}\leq \norm{\cL_{\om}\ind_\om}_{\infty},
		\end{align}
		and similarly, \eqref{rho^n def} implies that
		\begin{align}\label{rho^n rough up and low bdd}
			\inf_{I_\om} g_\om
			\leq 
			\inf_{D_{\sg(\om),1}}\cL_{\om}(\ind_\om)
			=
			\rho_{\om}^{(0)}
			\leq 
			\rho_{\om}^{(n)}
			\leq 
			\norm{\cL_{\om}\ind_\om}_{\infty}
			\leq 
			\#\cZ_\om\norm{g_\om}_\infty,
		\end{align}
		for all $\om\in\Om$ and $n\geq 0$.
		Furthermore, \eqref{rho rough up and low bdd}, together with \eqref{eq: log int open weight and tr op}, gives that 
		\begin{align}\label{eq: rho log int}
			\log\rho_\om\in L^1(m).
		\end{align}
		The ergodic theorem then implies that 
		\begin{align*}
			\lim_{n\to\infty}\frac{1}{n}\log\rho_\om^n=\int_\Om\log\rho_\om\,dm(\om),
		\end{align*}
		where 
		\begin{align*}
			\rho_\om^n:=\prod_{j=0}^{n-1}\rho_{\sg^j(\om)}.
		\end{align*}
	\end{remark}	

	\begin{proposition}\label{prop: D sets stabilize}
		There exists a measurable and finite $m$-a.e. function $N_\infty:\Om\to[1,\infty]$ such that 
		\begin{align*}
			D_{\om,n}=D_{\om,\infty}
		\end{align*}
		for all $n\geq N_\infty(\om)$. Furthermore, this implies that
		\begin{align}\label{eq: L1 pos on D infty}
			\inf_{D_{\om,\infty}}\cL_{\sg^{-n}(\om)}^n\ind_{\sg^{-n}(\om)}>0
		\end{align}
		for all $n\geq N_\infty(\om)$.
	\end{proposition}
	\begin{proof}
		We proceed via contradiction, assuming that there is a sequence $(n_k)_{k=1}^\infty$ in $\NN$ such that
		$$
		D_{\sg^{n_k+1}(\om),n_k+1}\subsetneq D_{\sg^{n_k+1}(\om),n_k}.
		$$
		Let $x_{n_k}\in D_{\sg^{n_k+1}(\om),n_k}\bs D_{\sg^{n_k+1}(\om),n_k+1}$. Then, we have that
		$$
		\rho_{\om}^{(n_k)}=\inf_{x\in D_{\sg^{n_k+1}(\om),n_k}}\frac{\cL_{\sg(\om)}^{n_k}(\cL_{\om}(\ind_\om))(x)}{\cL_{\sg(\om)}^{n_k}(\ind_{\sg(\om)})(x)}
		\leq \frac{\cL_{\sg(\om)}^{n_k}(\cL_{\om}(\ind_\om))(x_{n_k})}{\cL_{\sg(\om)}^{n_k}(\ind_{\sg(\om)})(x_{n_k})}.
		$$
		By our choice of $x_{n_k}$ and by the definition \eqref{defn of D sets}, we have that the numerator of the quantity on the right is zero, while its denominator is strictly positive. As this holds for each  $k\in\NN$, this implies that $\rho_\om=0$ for $m$-a.e. $\om\in\Om$, which contradicts \eqref{eq: rho log int}. Thus, we are done.
	\end{proof}
	\begin{remark}\label{rem: check cond D}
		Note that our assumption \eqref{cond D}, that $D_{\om,\infty}\neq\emptyset$, is satisfied if $T_\om(I_\om)\bus I_{\sg(\om)}$ for $m$-a.e. $\om\in\Om$. 
		Moreover, this also implies that $\rho_\om\geq \inf_{I_{\sg(\om)}}\cL_\om\ind_\om>0$. This occurs, for example if for $m$-a.e. $\om\in\Om$ there exists a full branch, i.e. there exists $Z\in\cZ_\om$ with $T_\om(Z)=I$, outside of the hole $H_\om$, in which case we would have that $D_{\om,\infty}=I$ for $m$-a.e. $\om\in\Om$. 
	\end{remark}
	We now describe various partitions, which depend on the functional $\Lm_\om$, that we will used to obtain a Lasota-Yorke inequality in Section~\ref{sec: LY ineq}. Recall that $\cZ_\om^{(n)}$ denotes the partition of monotonicity of $T_\om^n$. Now, for each $n\in\NN$ and $\om\in\Om$ we let $\sA_\om^{(n)}$ be the collection of all finite partitions of $I$ such that
	\begin{align}\label{eq: def A partition}
		\var_{A_i}(g_{\om}^{(n)})\leq 2\norm{g_{\om}^{(n)}}_{\infty}
	\end{align}
	for each $\cA=\set{A_i}\in\sA_{\om}^{(n)}$.
	
	Given $\cA\in\sA_\om^{(n)}$, let $\widehat\cZ_\om^{(n)}(\cA)$ be the coarsest partition amongst all those finer than $\cA$ and $\cZ_\om^{(n)}$ such that all elements of $\widehat\cZ_\om^{(n)}(\cA)$ are either disjoint from $X_{\om,n-1}$ or contained in $X_{\om,n-1}$. Now, define the following subcollections:
	\begin{align}
		\cZ_{\om,*}^{(n)}&:=\set{Z\in \widehat\cZ_\om^{(n)}(\cA): Z\sub X_{\om,n-1} },
		\label{eq: Z_*}
		\\
		\cZ_{\om,b}^{(n)}&:=\set{Z\in \widehat\cZ_\om^{(n)}(\cA): Z\sub X_{\om,n-1} \text{ and } \Lm_{\om}(\ind_Z)=0 },
		\label{eq: Z_b}
		\\
		\cZ_{\om,g}^{(n)}&:=\set{Z\in \widehat\cZ_\om^{(n)}(\cA): Z\sub X_{\om,n-1} \text{ and } \Lm_{\om}(\ind_Z)>0}.
		\label{eq: Z_g}
	\end{align}
\begin{remark}
	Note that in light of \eqref{eq: def of Lm} and \eqref{rho^n increasing}, for every $Z\in\cZ_{\om,g}^{(n)}$ we may define the (open) covering time $M_\om(Z)\in\NN$ to be the least integer such that 
	\begin{align}\label{eq: def of open covering}
		\inf_{x\in D_{\sg^{M_\om(Z)}(\om),M_\om(Z)}}\frac{\cL_{\om}^{M_\om(Z)}\ind_Z(x)}{\cL_{\om}^{M_\om(Z)}(\ind_\om)(x)} > 0
	\end{align}
	which is finite since the ratio in \eqref{eq: def of open covering} increases to $\Lm_\om(\ind_Z)>0$.
	Conversely, given that the ratio in \eqref{eq: def of open covering} is increasing by \eqref{rho^n increasing} we see that, for $Z\in\cZ_{\om,*}^{(n)}$, if there exists any $N\in\NN$ such that 
	\begin{align*}
		\inf_{x\in D_{\sg^{N}(\om),N}}\frac{\cL_{\om}^{N}\ind_Z(x)}{\cL_{\om}^{N}(\ind_\om)(x)} > 0,
	\end{align*}
	then we must have that $\Lm_\om(\ind_Z)>0$ and equivalently $Z\in\cZ_{\om,g}^{(n)}$.
\end{remark}
	We adapt the following definition from \cite{LMD}.
	\begin{definition}\label{def: contiguous}
		We say that two elements $W, Z\in\cZ_{\om,*}^{(n)}$ are \textit{contiguous} if either $W$ and $Z$ are contiguous in the usual sense, i.e. they share a boundary point, or if they are separated by a connected component of $\cup_{j=0}^{n-1} T_\om^{-j}(H_{\sg^j(\om)})$.
	\end{definition}

	We will consider random open systems that satisfy the following conditions.
	\begin{enumerate}
		\item[\mylabel{Q1}{cond Q1}] 
		For each $\om\in\Om$ and $n\in\NN$ we let $\xi_{\om}^{(n)}$ denote the maximum number of contiguous elements of $\cZ_{\om,b}^{(n)}$.
		We assume 
		\begin{align*}
			\lim_{n\to\infty}\frac{1}{n}\log\norm{g_{\om}^{(n)}}_\infty
			+
			\limsup_{n\to\infty}\frac{1}{n}\log\xi_{\om}^{(n)}
			<
			\lim_{n\to\infty}\frac{1}{n}\log\rho_{\om}^n
			=
			\int_\Om \log\rho_{\om}\, dm(\om)
		\end{align*}
		
		\item[\mylabel{Q2}{cond Q2}] We assume that for each $n\in\NN$ we have $\log\xi_\om^{(n)}\in L^1(m)$.
		
		\item[\mylabel{Q3}{cond Q3}] Let 
		\begin{align}\label{eq: def of dl_om,n}
			\dl_{\om,n}:=\min_{Z\in\cZ_{\om,g}^{(n)}}\Lm_\om(\ind_Z).
		\end{align}
		We assume that, for each $n\in\NN$, $\log\dl_{\om,n}\in L^1(m)$. 
	\end{enumerate}
	\begin{remark}
		\noindent
		\begin{enumerate}[(1)]
			\item Note that since $\lim_{n\to\infty}\frac{1}{n}\log\xi_{\om}^{(n)}\geq 0$, assumption \eqref{cond Q1} implies that 
			\begin{align*}
				\lim_{n\to\infty}\frac{1}{n}\log\norm{g_{\om}^{(n)}}_\infty
				<
				\lim_{n\to\infty}\frac{1}{n}\log\rho_{\om}^n.
			\end{align*}
			
			\item Since $\norm{g_{\om}}_\infty\leq\norm{g_{\om,c}}_\infty$ and $\inf_{D_{\sg(\om), \infty}}\cL_{\om}\ind_\om\leq \rho_\om$ to check \eqref{cond Q1} it suffices to have 
			\begin{align*}
				\lim_{n\to\infty}\frac{1}{n}\log\norm{g_{\om,c}^{(n)}}_\infty 
				+
				\lim_{n\to\infty}\frac{1}{n}\log\xi_{\om}^{(n)}
				< 
				\lim_{n\to\infty}\frac{1}{n}\log\inf_{D_{\sg^n(\om),\infty}}\cL_{\om}^n\ind_\om .
			\end{align*} 
			\item One can use the open covering times defined in \eqref{eq: def of open covering} to check \eqref{cond Q3}. Indeed, note that if 
			\begin{align*}
				N\geq M_{\om,n}:=\max\set{M_\om(Z): Z\in\cZ_{\om,g}^{(n)}}
			\end{align*}
			then we have that 
			\begin{align*}
				\dl_{\om,n}
				&\geq 
				\min_{Z\in\cZ_{\om,g}^{(n)}}\inf_{x\in D_{\sg^N(\om),N}}\frac{\cL_\om^N\ind_Z(x)}{\cL_\om^N\ind_\om(x)}
				\\
				&\geq
				\min_{Z\in\cZ_{\om,g}^{(n)}}
				\frac{\inf_{x\in D_{\sg^N(\om),N}}\cL_\om^N\ind_Z(x)}{\norm{\cL_\om^N\ind_\om}_\infty}
				\\
				&\geq 
				\frac{\inf_{X_{\om,N-1}} g_\om^{(N)}}{\norm{\cL_\om^N\ind_\om}_\infty}
				\\
				&\geq
				\frac{\inf g_{\om,c}^{(N)}}{\norm{\cL_{\om,c}^N\ind}_\infty}>0.
			\end{align*}
			Thus \eqref{cond Q3} holds if $\log \inf g_{\om,c}^{(M_{\om,n})}, \log\norm{\cL_{\om,c}^{M_{\om,n}}\ind_\om}_\infty\in L^1(m)$ for each $n\in\NN$.
		\end{enumerate}
	\end{remark}
	\begin{remark}
		In Section~\ref{sec: examples} we give several alternate hypotheses to our assumptions \eqref{cond Q1}-\eqref{cond Q3} that are more restrictive, but much simpler to check. 
	\end{remark}
	
	\begin{definition}\label{def: random open weighted covering system}
		If the assumptions \eqref{cond T1}-\eqref{cond T3}, \eqref{cond LIP}, \eqref{cond GP}, \eqref{cond A1}-\eqref{cond A2}, \eqref{cond M1}, \eqref{cond C1}, \eqref{cond D}, and \eqref{cond Q1}-\eqref{cond Q3} are satisfied then we call the tuple $(\Om,\sg,m,I,T,\phi_c, H)$ a \textit{random weighted open system}. 
	\end{definition}
	\subsection{Main Results}\label{sec: main results}
	The following are simplified versions of the main results of this paper.
	\begin{thmx}\label{main thm: existence}
		Given a random open weighted covering system $(\Om,\sg,m,I,T,\phi_c, H)$, the following hold. 
		\begin{enumerate}
			\item There exists a unique random probability measure $\nu\in\cP_\Om(I)$ supported in $\cX_\infty$ such that 
			\begin{align*}
				\nu_{\sg(\om)}(\cL_\om f)=\lm_\om\nu_\om(f),
			\end{align*}
			for each $f\in\BV(I)$, where 
			\begin{align*}
				\lm_\om:=\nu_{\sg(\om)}(\cL_\om\ind_\om).
			\end{align*} 	
			Furthermore, we have that $\log\lm_\om\in L^1(m)$.
			\item There exists a function $q\in\BV_\Om(I)$ such that $\nu(q)=1$ and for $m$-a.e. $\om\in\Om$ we have 
			\begin{align*}
				\cL_\om q_\om=\lm_\om q_{\sg(\om)}.
			\end{align*}
			Moreover, $q$ is unique modulo $\nu$. 
			\item The measure $\mu:=q\nu$ is a $T$-invariant and ergodic random probability measure supported in $\cX_{\infty}$.
			\item There exists a unique random absolutely continuous conditionally invariant probability measure $\eta$, which is supported on $\cI$, and whose disintegrations are given by
			\begin{align*}
				\eta_\om(f):=\frac{\nu_{\om,c}\lt(\ind_\om q_{\om} f\rt)}{\nu_{\om,c}\lt(\ind_\om q_\om\rt)}
			\end{align*} 
			for all $f\in\BV(I)$.
		\end{enumerate}
	\end{thmx}
	Theorem~\ref{main thm: existence} follows from results in Sections~\ref{sec: conf and inv meas}, \ref{sec: dec of cor}, and \ref{sec: exp press}.
	We also show that the operator cocycle is quasi-compact.
	\begin{thmx}\label{main thm: quasicompact}
		With the same hypotheses as Theorem~\ref{main thm: existence},
		for each $f\in\BV(I)$ there exists a measurable function $\Om\ni\om\mapsto D(\om)\in(0,\infty)$ and $\kp\in(0,1)$ such that for $m$-a.e. $\om\in\Om$ and all $n\in\NN$ we have
		\begin{align*}
			\norm{\lt(\lm_\om^n\rt)^{-1}\cL_{\om}^n f - \nu_{\om}(f)q_{\sg^n(\om)}}_\infty\leq D(\om)\norm{f}_\infty\kp^n.
		\end{align*}
		Furthermore, for all $A\in\sB$ we have 
		\begin{align*}
			\absval{
			\nu_{\om,c}\lt(T_\om^{-n}(A)\,\rvert\, X_{\om,n}\rt)
			-
			\eta_{\sg^n(\om)}(A)	
			}
			\leq 
			D(\om)\kp^n,
		\end{align*}
		and
		\begin{align*}
			\absval{
				\eta_\om\lt(A\,\rvert\, X_{\om,n}\rt)
				-
				\mu_\om(A)	
			}
			\leq 
			D(\om)\kp^n	.		
		\end{align*}
	\end{thmx}
	For the proof of Theorem~\ref{main thm: quasicompact}, as well as a more general statement, see Theorem~\ref{thm: exp convergence of tr op} and Corollary~\ref{cor: exp conv of eta}.
	From quasi-compactness we easily deduce the exponential decay of correlations for the invariant measure $\mu$. 
	\begin{thmx}\label{main thm: exp dec of corr}
		With the same hypotheses as Theorem~\ref{main thm: existence}, there exists a measurable function $\Om\ni\om\mapsto C(\om)\in(0,\infty)$ such that for every $h\in \BV(I)$, every $f\in L^1(\mu)$, every $\vkp\in(\kp,1)$, with $\kp$ as in Theorem~\ref{main thm: quasicompact}, every $n\in\NN$, and for $m$-a.e. $\om\in\Om$ we have 
		\begin{align*}
			\absval{
				\mu_{\om}
				\lt(\lt(f_{\sg^{n}(\om)}\circ T_{\om}^n\rt)h \rt)
				-
				\mu_{\sg^{n}(\om)}(f_{\sg^{n}(\om)})\mu_{\om}(h)
			}
			\leq C(\om)
			\norm{f_{\sg^n(\om)}}_{L^1(\mu_{\sg^n(\om)})}\norm{h}_\infty\vkp^n.
		\end{align*} 
	\end{thmx}
Theorem~\ref{main thm: exp dec of corr} is proven in Section~\ref{sec: dec of cor}.
	\begin{definition}\label{def: escape rate}
		Given a random probability measure $\vrho$ on $\cI$, for each $\om\in\Om$, we define the lower and upper escape rates respectively by the following:
		\begin{align*}
			\ul R(\vrho_\om):=-\liminf_{n\to\infty}\frac{1}{n}\log \vrho_\om(X_{\om,n})
			\quad \text{ and } \quad
			\ol R(\vrho_\om):=-\limsup_{n\to\infty}\frac{1}{n}\log \vrho_\om(X_{\om,n}).
		\end{align*}
		If $\ul R(\vrho_\om)=\ol R(\vrho_\om)$, we say the escape rate exists and denote the common value by $R(\vrho_\om)$. 
	\end{definition}
	\begin{definition}\label{def: expected pressure}
		Given a potential $\phi_c$ on the closed system we define the expected pressure of the closed and open systems respectively by 
		\begin{align*}
			\cEP(\phi_c):=\int_\Om\log\lm_{\om,c}\, dm(\om)
			\quad\text{ and }\quad
			\cEP(\phi):=\int_\Om\log\lm_\om\, dm(\om). 
		\end{align*}
	\end{definition}
	\begin{thmx}\label{main thm: escape rate}
		With the same hypotheses as Theorem~\ref{main thm: existence}, for $m$-a.e. $\om\in\Om$ we have that 
		\begin{align*}
			R(\nu_{\om,c})=R(\eta_\om)=\cEP(\phi_c)-\cEP(\phi).
		\end{align*}
	\end{thmx}
Theorem~\ref{main thm: escape rate} is proven in Section~\ref{sec: exp press}.
	\begin{definition}\label{main def: bdd dist}
		We will say that the weight function $g_c$ has the \textit{Bounded Distortion Property} if for each $\om\in\Om$ there exists $K_\om\geq 1$ such that for all $n\in\NN$, all $Z\in\cZ_{\om}^{(n)}$, and all $x,y\in Z$ we have that 
		\begin{align*}
			\frac{g_{\om,c}^{(n)}(x)}{g_{\om,c}^{(n)}(y)}\leq K_\om.
		\end{align*}
	\end{definition}
	\begin{definition}\label{main def: large images}
		We will say that the map $T$ has \textit{large images} if for each $\om\in\Om$ we have 
		\begin{align*}
			\inf_{n\in\NN}\inf_{Z\in\cZ_\om^{(n)}}\nu_{\sg^n(\om),c}\lt(T_\om^n(Z)\rt)>0.
		\end{align*}
		$T$ is said to have \textit{large images with respect to $H$} if for each $\om\in\Om$, each $n\in\NN$, and each $Z\in\cZ_\om^{(n)}$ with $Z\cap X_{\om,\infty}\neq\emptyset$ we have 
		\begin{align*}
			T_\om^n(Z\cap X_{\om,n-1})\bus X_{\sg^n(\om),\infty}.
		\end{align*}
	\end{definition}
	Finally we prove a formula for the Hausdorff dimension of the survivor set $X_{\om,\infty}$ for $m$-a.e. $\om\in\Om$ in the spirit of Bowen. 
	\begin{thmx}\label{main thm: Bowens formula}
		Given a random open weighted covering system $(\Om,\sg,m,I,T,\phi_c, H)$ such that $g_c=1/|T'|$ has bounded distortion, 
		then there exists a unique $h\in[0,1]$ such that $\cEP(t)>0$ for all $0\leq t<h$ and $\cEP(t)<0$ for all $h<t\leq 1$.
		
		Furthermore, if $T$ has large images and large images with respect to $H$, then for $m$-a.e. $\om\in\Om$
		\begin{align*}
			\HD(X_{\om,\infty})=h,		
		\end{align*}
		where $\HD(A)$ denotes the Hausdorff dimension of the set $A$.
	\end{thmx}
	The proof of Theorem~\ref{main thm: Bowens formula} appears in Section~\ref{sec: bowen}. 

	\section{Random Birkhoff Cones and Hilbert Metrics}\label{sec:cones}
	In this section we first recall the theory of convex cones first used by Birkhoff in \cite{birkhoff_lattice_1940}, and then present the random cones on which our operator $\cL_{\om}$ will act as a contraction. We begin with a definition. 
	\begin{definition}
		Given a vector space $\cV$, we call a subset $\cC\sub\cV$ a \textit{convex cone} if $\cC$ satisfies the following:
		\begin{enumerate}
			\item $\cC\cap-\cC=\emptyset$,
			\item for all $\al>0$, $\al\cC=\cC$, 
			\item $\cC$ is convex,
			\item for all $f, h\in\cC$ and all $\al_n\in\RR$ with $\al_n\to\al$ as $n\to\infty$, if $h-\al_nf\in\cC$ for each $n\in\NN$, then $h-\al f\in\cC\cup\set{0}$.
		\end{enumerate}
	\end{definition}
	\begin{lemma}[Lemma 2.1 \cite{LSV}]
		The relation $\leq $ defined on $\cV$ by 
		$$
		f\leq h \text{ if and only if } h-f\in\cC\cup\set{0}
		$$
		is a partial order satisfying the following:
		\begin{flalign*}
			& f\leq 0\leq f  \implies f=0,
			\tag{i}
			&\\
			& \lm>0 \text{ and } f\geq 0 \iff \lm f\geq 0,
			\tag{ii} \\
			& f\leq h \iff 0\leq h-f,
			\tag{iii}\\
			& \text{for all } \alpha_n\in\RR \text{ with } \al_n\to\al, \, \al_nf\leq h \implies \al f\leq h,
			\tag{iv} \\
			& f\geq 0 \text{ and } h\geq 0 \implies f+h\geq 0.
			\tag{v}
		\end{flalign*}
	\end{lemma}
	The Hilbert metric on $\cC$ is given by the following definition.
	\begin{definition}
		Define a distance $\Ta(f,h)$ by 
		\begin{align*}
			\Ta(f,h):=\log\frac{\bt(f,h)}{\al(f,h)},
		\end{align*}
		where 
		\begin{align*}
			\al(f,h):=\sup\set{a>0: af\leq h} 
			\quad\text{ and }\quad
			\bt(f,h):=\inf\set{b>0: bf\geq h}.
		\end{align*}
	\end{definition}
	Note that $\Ta$ is a pseudo-metric as two elements in the cone may be at an infinite distance from each other. Furthermore, $\Ta$ is a projective metric because any two proportional elements must be zero distance from each other. The next theorem, which is due to Birkhoff \cite{birkhoff_lattice_1940}, shows that every positive linear operator that preserves the cone is a contraction provided that the diameter of the image is finite. 
	\begin{theorem}[\cite{birkhoff_lattice_1940}]\label{thm: cone distance contraction}
		Let $\cV_1$ and $\cV_2$ be vector spaces with convex cones $\cC_1\sub\cV_1$ and $\cC_2\sub\cV_2$ and a positive linear operator $\cL:\cV_1\to\cV_2$ such that $\cL(\cC_1)\sub \cC_2$. If $\Ta_i$ denotes the Hilbert metric on the cone $\cC_i$ and if 
		$$
		\Dl=\sup_{f,h\in\cC_1}\Ta_2(\cL f,\cL h),
		$$
		then 
		$$
		\Ta_2(\cL f,\cL h)\leq \tanh\lt(\frac{\Dl}{4}\rt)\Ta_1(f,h).
		$$
		for all $f,h\in\cC_1$.
	\end{theorem}
	Note that it is not clear whether $(\cC,\Ta)$ is complete.
	The following lemma of \cite{LSV} addresses this problem by linking the metric $\Ta$ with a suitable norm $\norm{\spot}$ on $\cV$.
	\begin{lemma}[\cite{LSV}, Lemma 2.2]\label{lem: birkhoff cone contraction}
		Let $\norm{\spot}$ be a norm on $\cV$ such that for all $f,h\in\cV$ if $-f\leq h\leq f$, then $\norm{h}\leq \norm{f}$, and let  $\vrho:\cC\to (0,\infty)$ be a homogeneous and order-preserving function, which means that for all $f,h\in\cC$ with $f\leq h$ and all $\lm>0$ we have 
		$$
		\vrho(\lm f)=\lm\vrho(f) \qquad\text{ and }\qquad \vrho(f)\leq \vrho(h).
		$$
		Then, for all $f,h\in\cC$ $\vrho(f)=\vrho(h)>0$ implies that 
		$$
		\norm{f-h}\leq \lt(e^{\Ta(f,h)}-1\rt)\min\set{\norm{f},\norm{h} }.
		$$
	\end{lemma}
	\begin{remark}
		Note that the choice $\vrho(\spot)=\norm{\spot}$ satisfies the hypothesis, however from this moment on we shall make the choice of $\vrho=\Lm_{\om}$.
	\end{remark}
	\begin{definition}\label{def: + and a cones}
		
		For each $a>0$ and $\om\in\Om$ let 
		\begin{align}\label{eq: def of a cones}
			\sC_{\om,a}:=\set{f\in\BV(I): f\geq 0,\, \var(f)\leq a\Lm_{\om}(f)}.
		\end{align}
		To see that this cone is non-empty, we note that the function $f+c\in\sC_{\om,a}$ for $f\in\BV(I)$ and $c\geq a^{-1}\var(f)-\inf_{X_\om}f$. We also define the cone 
		\begin{align*}
			\sC_{\om,+}:=\set{f\in\BV(I): f\geq 0}.
		\end{align*}
		
	\end{definition} 
	Let $\Ta_{\om,a}$ and $\Ta_{\om,+}$ denote the Hilbert metrics induced on the respective cones $\sC_{\om,a}$ and $\sC_{\om,+}$. 
	For each $\om\in\Om$, $a>0$, and any set $Y\sub \sC_{\om,a}$ we let
	\begin{align*}
		\diam_{\om,a}(Y):=\sup_{x,y\in Y} \Ta_{\om,a}(x,y)
	\end{align*}
	and 
	\begin{align*}
		\diam_{\om,+}(Y):=\sup_{x,y\in Y} \Ta_{\om,+}(x,y)
	\end{align*}
	denote the diameter of $Y$ in the respective cones $\sC_{\om,a}$ and $\sC_{\om,+}$ with respect to the respective metrics $\Ta_{\om,a}$ and $\Ta_{\om,+}$.
	The following lemma collects together the main properties of these metrics. 
	\begin{lemma}[\cite{LSV}, Lemmas 4.2, 4.3, 4.5]\label{lem: summary of cone dist prop}
		For $f, h\in\sC_{\om,+}$ the $\Ta_{\om,+}$ distance between $f,h$ is given by 
		\begin{align*}
			\Ta_{\om,+}(f,h)=\log\sup_{x,y\in X_\om}\frac{f(y)h(x)}{f(x)h(y)}
		\end{align*}
		If $f, h\in\sC_{\om,a}$, then
		\begin{align}\label{eq: Ta+ leq Ta}
			\Ta_{\om,+}(f,h)\leq \Ta_{\om,a}(f,h), 
		\end{align}
		and if $f\in\sC_{\om,\eta a}$, for $\eta\in (0,1)$, we then have
		\begin{align*}
			\Ta_{\om,a}(\ind,f)\leq \log\frac{\norm{f}_\infty+\eta\Lm_\om(f)}{\min\set{\inf_{X_\om} f, (1-\eta)\Lm_\om(f)}}.
		\end{align*}
	\end{lemma}


\section{Lasota-Yorke Inequalities}\label{sec: LY ineq}
The main goal of this section is to prove a Lasota-Yorke type inequality. We adopt the strategy of \cite{AFGTV21}, where we first prove a less-refined Lasota-Yorke inequality with (random) coefficients that behave in a difficult manner, and then, using the first inequality, prove a second inequality with measurable random coefficients and uniform decay on the variation as in \cite{buzzi_exponential_1999}. 

We now prove a Lasota-Yorke type inequality following the approach of \cite{LMD} utilizing the ``good'' and ``bad'' interval partitions defined in \eqref{eq: Z_*}-\eqref{eq: Z_g}. 
\begin{lemma}\label{ly ineq}	
	For all $\om\in\Om$, all $f\in\BV(I)$, and all $n\in\NN$ there exist positive, measurable constants $A_{\om}^{(n)}$ and $B_{\om}^{(n)}$ such that 
	\begin{align*}
		\var(\cL_{\om}^nf)\leq A_{\om}^{(n)}\var(f)+B_{\om}^{(n)}\Lm_{\om}(|f|), 
	\end{align*}
	where
	\begin{align*}
		A_{\om}^{(n)}:=(9+16\xi_{\om}^{(n)})\norm{g_{\om}^{(n)}}_{\infty}
	\end{align*}
	and 
	\begin{align*}
		B_{\om}^{(n)}:=8(2\xi_{\om}^{(n)}+1)\norm{g_{\om}^{(n)}}_{\infty}
		\dl_{\om,n}^{-1}.
	\end{align*}
\end{lemma}
\begin{proof}
	Since $\cL_{\om}^n(f)=\cL_{\om,0}^n(f\cdot\hat X_{\om,n-1})$, if $Z\in\widehat\cZ_\om^{(n)}(\cA)\bs\cZ_{\om,*}^{(n)}$, then  $Z\cap X_{\om,n-1}=\emptyset$, and thus, we have $\cL_{\om}^n(f\ind_Z)=0$ for each $f\in\BV(I)$. Thus, considering only intervals $Z$ in $\cZ_{\om,*}^{(n)}$, we are able to write 
	\begin{align}\label{eq: ly ineq 1}
	\cL_{\om}^nf=\sum_{Z\in\cZ_{\om,*}^{(n)}}(\ind_Z f g_{\om}^{(n)})\circ T_{\om,Z}^{-n}
	\end{align} 
	where 
	$$	
	T_{\om,Z}^{-n}:T_\om^n(I_{\om})\to Z
	$$ 
	is the inverse branch which takes $T_\om^n(x)$ to $x$ for each $x\in Z$. Now, since 
	$$
	\ind_Z\circ T_{\om,Z}^{-n}=\ind_{T_\om^n(Z)},
	$$
	we can rewrite \eqref{eq: ly ineq 1} as 
	\begin{align}\label{eq: ly ineq 2}
	\cL_{\om}^nf=\sum_{Z\in\cZ_{\om,*}^{(n)}}\ind_{T_\om^n(Z)} \lt((f g_{\om}^{(n)})\circ T_{\om,Z}^{-n}\rt).
	\end{align}
	So,
	\begin{align}\label{var tr op sum}
	\var(\cL_{\om}^nf)\leq \sum_{Z\in\cZ_{\om,*}^{(n)}}\var\lt(\ind_{T_\om^n(Z)} \lt((f g_{\om}^{(n)})\circ T_{\om,Z}^{-n}\rt)\rt).
	\end{align}
	Now for each $Z\in\cZ_{\om,*}^{(n)}$, using \eqref{eq: def A partition}, we have 
	\begin{align}
	&\var\lt(\ind_{T_\om^n(Z)} \lt((f g_{\om}^{(n)})\circ T_{\om,Z}^{-n}\rt)\rt)
	\leq \var_Z(f g_{\om}^{(n)})+2\sup_Z\absval{f g_{\om}^{(n)}}
	\nonumber\\
	&\qquad\qquad\leq 3\var_Z(f g_{\om}^{(n)})+2\inf_Z\absval{f g_{\om}^{(n)}}
	\nonumber\\
	&\qquad\qquad\leq 3\norm{g_{\om}^{(n)}}_{\infty}\var_Z(f)+3\sup_Z|f|\var_Z(g_{\om}^{(n)})+2\norm{g_{\om}^{(n)}}_{\infty}\inf_Z|f|
	\nonumber\\
	&\qquad\qquad\leq 
	3\norm{g_{\om}^{(n)}}_{\infty}\var_Z(f)+6\norm{g_{\om}^{(n)}}_{\infty}\sup_Z|f|+2\norm{g_{\om}^{(n)}}_{\infty}\inf_Z|f|
	\nonumber\\
	&\qquad\qquad\leq 
	9\norm{g_{\om}^{(n)}}_{\infty}\var_Z(f)+8\norm{g_{\om}^{(n)}}_{\infty}\inf_Z|f|.
	\label{var ineq over partition}
	\end{align}
	Now, using \eqref{var ineq over partition}, we may further estimate \eqref{var tr op sum} as
	\begin{align}
	\var(\cL_{\om}^nf)
	&\leq 
	\sum_{Z\in\cZ_{\om,*}^{(n)}} \lt(9\norm{g_{\om}^{(n)}}_{\infty}\var_Z(f)+8\norm{g_{\om}^{(n)}}_{\infty}\inf_Z|f|\rt)
	\nonumber\\
	&\leq 
	9\norm{g_{\om}^{(n)}}_{\infty}\var(f)+8\norm{g_{\om}^{(n)}}_{\infty}
	\lt(\sum_{Z\in\cZ_{\om,g}^{(n)}}\inf_Z|f|+\sum_{Z\in\cZ_{\om,b}^{(n)}}\inf_Z|f|\rt).
	\label{sum over *partition}
	\end{align}
	In order to investigate each of the two sums in the line above, we first note that as $\cZ_{\om,g}^{(n)}$ is finite then, by definition, there exists a constant $\dl_{\om,n}>0$ (defined by \eqref{eq: def of dl_om,n}) such that 
	\begin{align*}
	\inf_{Z\in\cZ_{\om,g}^{(n)}}\Lm_{\om}(\ind_Z)\geq 2\dl_{\om,n}>0.
	\end{align*}
	So, we may choose $N_{\om,n}\in\NN$ such that for $x\in D_{\sg^{N_{\om,n}}(\om),N_{\om,n}}$ we have 
	\begin{align}\label{eq: ratio bigger than bt_om,n}
	\inf_{Z\in\cZ_{\om,g}^{(n)}}\frac{\cL_{\om}^{N_{\om,n}}(\ind_Z)(x)}{\cL_{\om}^{N_{\om,n}}(\ind_\om)(x)}\geq \dl_{\om,n}.
	\end{align}
	Note that since this ratio is increasing we have that \eqref{eq: ratio bigger than bt_om,n} holds for all $\~N\geq N_{\om,n}$.
	Then for each $x\in D_{\sg^{N_{\om,n}}(\om),N_{\om,n}}$ and $Z\in\cZ_{\om,g}^{(n)}$ we have
	\begin{align*}
	\cL_{\om}^{N_{\om,n}}(|f|\ind_Z)(x)&\geq \inf_Z|f|\cL_{\om}^{N_{\om,n}}(\ind_Z)(x)
	\geq \inf_Z|f|\dl_{\om,n}\cL_{\om}^{N_{\om,n}}(\ind_\om)(x).
	\end{align*}
	In particular, for each $x\in D_{\sg^{N_{\om,n}}(\om),N_{\om,n}}$, we see that 
	\begin{align}
	\sum_{Z\in\cZ_{\om,g}^{(n)}}\inf_Z|f|
	&\leq \dl_{\om,n}^{-1}\sum_{Z\in\cZ_{\om,g}^{(n)}}\frac{\cL_{\om}^{N_{\om,n}}(|f|\ind_Z)(x)}{\cL_{\om}^{N_{\om,n}}(\ind_\om)(x)}
	\leq 
	\dl_{\om,n}^{-1}\frac{\cL_{\om}^{N_{\om,n}}(|f|)(x)}{\cL_{\om}^{N_{\om,n}}(\ind_\om)(x)}.\label{est of sum over good}
	\end{align}
	We are now interested in finding appropriate upper bounds for the sum 
	\begin{align*}
	\sum_{Z\in\cZ_{\om,b}^{(n)}}\inf_Z|f|.
	\end{align*}
	However, we must first be able to associate each of the elements of $\cZ_{\om,b}^{(n)}$ with one of the elements of $\cZ_{\om,g}^{(n)}$. To that end, let $Z_*$ and $Z^*$ denote the elements of $\cZ_{\om,*}^{(n)}$ that are the furthest to the left and the right respectively.  
	Now, enumerate each of the elements of $\cZ_{\om,g}^{(n)}$, $Z_1,\dots, Z_k$ (clearly $k$ depends on $\om$, $n$, and $\cA$), such that $Z_{j+1}$ is to the right of $Z_j$ for $j=1,\dots, k-1$.
	Given $Z_j\in\cZ_{\om,g}^{(n)}$ ($1\leq j\leq k$), with $Z_j\neq Z^*$  let $J_{\om,+}(Z_j)$ be the union of all contiguous elements $Z\in\cZ_{\om,b}^{(n)}$ which are to the right of $Z_j$ and also to the left of $Z_{j+1}$. In other words, $J_{\om,+}(Z_j)$ is the union of all elements of $\cZ_{\om,b}^{(n)}$ between $Z_j$ and $Z_{j+1}$. Similarly, for $Z_j\neq Z_*$, we define $J_{\om,-}(Z_j)$ be the union of all contiguous elements $Z\in\cZ_{\om,b}^{(n)}$ which are to the left of $Z_j$ and also to the right of $Z_{j-1}$. Now, we note that our assumption \eqref{cond Q1} implies that each $J_{\om,-}(Z)$ and $J_{\om,+}(Z)$ ($Z\in \cZ_{\om,g}^{(n)}$) is the union of at most $\xi_{\om}^n$ contiguous elements of $\cZ_{\om,b}^{(n)}$.  
	For $Z\in\cZ_{\om,g}^{(n)}$ let  
	$$
	J_{\om,-}^*(Z)=Z\cup J_{\om,-}(Z) 
	\qquad\text{ and }\qquad
	J_{\om,+}^*(Z)=Z\cup J_{\om,+}(Z). 
	$$
	Then for $W\sub J_{\om,-}^*(Z)$ we have 
	\begin{align}\label{inf over union BV ineq}
	\inf_W |f|\leq \inf_{Z}|f|+\var_{J_{\om,-}^*(Z)}(f).
	\end{align}
	We obtain a similar inequality for $W\sub J_{\om,+}^*(Z)$. 
	We now consider the following two cases. 
	\begin{enumerate}
		\item[\mylabel{Case 1:}{Case 1}] At least one of the intervals $Z_*$ and $Z^*$ is an element of $\cZ_{\om,g}^{(n)}$. 
		\item[\mylabel{Case 2:}{Case 2}] Neither of the intervals $Z_*$, $Z^*$ is an element of $\cZ_{\om,g}^{(n)}$.
	\end{enumerate}
	If we are in the first case, we assume without loss of generality that $Z_1=Z_*$, and thus every element $Z\in\cZ_{\om,b}^{(n)}$ is contained in exactly one union $J_{\om,+}(Z_j)$ for some $Z_j\in\cZ_{\om,g}^{(n)}$ for some $1\leq j\leq k$. If $Z_1\neq Z_*$ and instead we have that $Z_k=Z^*$ we could simply replace $J_{\om,+}(Z_j)$ with $J_{\om,-}(Z_j)$ in the previous statement. In view of \eqref{inf over union BV ineq}, Case 1 leads to the conclusion that  
	\begin{align*}
	\sum_{Z\in\cZ_{\om,b}^{(n)}}\inf_Z|f| \leq 	\xi_{\om}^{(n)}\lt(\sum_{Z\in\cZ_{\om,g}^{(n)}}\inf_Z|f|+\var_{J_{\om,-}^*(Z)}(f)\rt).
	\end{align*}	
	If we are instead in the second case, then for each $Z\in\cZ_{\om,b}^{(n)}$ to the left of $Z_k$ there is exactly one $Z_j$, $1\leq j\leq k$, such that $Z\sub J_{\om,-}(Z_j)$. This leaves each of the elements $Z\in\cZ_{\om,b}^{(n)}$ to the right of $Z_k$ uniquely contained in the union $J_{\om,+}(Z_k)$. Thus, Case 2 yields  
	\begin{align*}
	\sum_{Z\in\cZ_{\om,b}^{(n)}}\inf_Z|f| &\leq 	\xi_{\om}^{(n)}\lt(\inf_{Z_k}|f|+\var_{J_{\om,+}^*(Z_k)}(f)+\sum_{Z\in\cZ_{\om,g}^{(n)}}\inf_Z|f|+\var_{J_{\om,-}^*(Z)}(f)\rt)\\
	&\leq 2\xi_{\om}^{(n)}\lt(\var(f)+\sum_{Z\in\cZ_{\om,g}^{(n)}}\inf_Z|f|\rt).
	\end{align*}
	Hence, either case gives that 
	\begin{align}\label{est of sum over bad}
	\sum_{Z\in\cZ_{\om,b}^{(n)}}\inf_Z|f| \leq 2\xi_{\om}^{(n)}\lt(\var(f)+\sum_{Z\in\cZ_{\om,g}^{(n)}}\inf_Z|f|\rt).
	\end{align}
	Inserting \eqref{est of sum over good} and \eqref{est of sum over bad} into \eqref{sum over *partition} gives 
	\begin{align*}
	\var(\cL_{\om}^nf)
	&\leq 9\norm{g_{\om}^{(n)}}_{\infty}\var(f)\\
	&\qquad\quad
	+8\norm{g_{\om}^{(n)}}_{\infty}\lt(2\xi_{\om}^{(n)}\lt(\var(f)+\sum_{Z\in\cZ_{\om,g}^{(n)}}\inf_Z|f|\rt)+\dl_{\om,n}^{-1}\frac{\cL_{\om}^{N_{\om,n}}(|f|)(x)}{\cL_{\om}^{N_{\om,n}}(\ind_\om)(x)}\rt)
	\\
	&\leq 
	(9+16\xi_{\om}^{(n)})\norm{g_{\om}^{(n)}}_{\infty}\var(f)
	+
	8(2\xi_{\om}^{(n)}+1)\norm{g_{\om}^{(n)}}_{\infty}
	\dl_{\om,n}^{-1}\frac{\cL_{\om}^{N_{\om,n}}(|f|)(x)}{\cL_{\om}^{N_{\om,n}}(\ind_\om)(x)}.
	\end{align*}
	In view of \eqref{rho^n increasing}, taking the infimum over $x\in D_{\sg^{N_{\om,n}}(\om),N_{\om,n}}$ allows us to replace the ratio $\frac{\cL_{\om}^{N_{\om,n}}(|f|)(x)}{\cL_{\om}^{N_{\om,n}}(\ind_\om)(x)}$ with $\Lm_{\om}(|f|)$, that is, we have
	\begin{align*}
	&\var(\cL_{\om}^nf)
	\leq 
	(9+16\xi_{\om}^{(n)})\norm{g_{\om}^{(n)}}_{\infty}\var(f)
	+
	8(2\xi_{\om}^{(n)}+1)\norm{g_{\om}^{(n)}}_{\infty}
	\dl_{\om,n}^{-1}\Lm_{\om}(|f|).
	\end{align*}
	Setting 
	\begin{align}\label{eq: def of A and B in ly ineq}
		A_{\om}^{(n)}:=(9+16\xi_{\om}^{(n)})\norm{g_{\om}^{(n)}}_{\infty}
	\qquad\text{ and }\qquad
		B_{\om}^{(n)}:=8(2\xi_{\om}^{(n)}+1)\norm{g_{\om}^{(n)}}_{\infty}
		\dl_{\om,n}^{-1}
	\end{align}
	finishes the proof.
\end{proof}
\begin{remark}
	As a consequence of Lemma~\ref{ly ineq} we have that 
	\begin{align}\label{eq: L_om is a weak contraction on C_+}
		\cL_\om\lt(\sC_{\om,+}\rt)\sub \sC_{\sg(\om),+},
	\end{align}
	and thus $\cL_{\om}$ is a weak contraction on $\sC_{\om,+}$. 
\end{remark}
Define the random constants
\begin{align}\label{eq: def of Q and K}
Q_{\om}^{(n)}:=\frac{A_\om^{(n)}}{\rho_{\om}^n}
\quad \text{ and }\quad
K_{\om}^{(n)}:=\frac{B_\om^{(n)}}{\rho_{\om}^n}. 
\end{align}
In light of our assumption \eqref{cond Q1} on the potential and number of contiguous bad intervals, we see that $Q_\om^{(n)}\to 0$ exponentially quickly for each $\om\in\Om$.

The following proposition now follows from \eqref{eq: log int open weight and tr op}, \eqref{eq: rho log int}, and assumptions \eqref{cond Q2}-\eqref{cond Q3}. 
\begin{proposition}\label{prop: log integr of Q and K}
	For each $n\in\NN$, $\log^+ Q_\om^{(n)}, \log K_\om^{(n)}\in L^1_m(\Om)$. 
\end{proposition}

\begin{lemma}\label{LMD l3.6}
	For each $f\in\BV(I)$ and each $n,k\in\NN$ we have
	\begin{align}
		\Lm_{\sg^k(\om)}\lt(\cL_{\om}^kf\rt)&\geq \Lm_{\sg^k(\om)}\lt(\cL_{\om}^k\ind_\om\rt)\cdot\Lm_{\om}(f).
		\label{eq: Lm^k1 ineq}
	\end{align}
	Furthermore, we have that 
	\begin{align}\label{LMD l3.6 key ineq}
		\rho_{\om}^n\cdot\Lm_{\om}(f)\leq \Lm_{\sg^n(\om)}(\cL_{\om}^n f).
	\end{align}
	In particular, this yields
	\begin{align*}
		\rho_{\om}^n\leq \Lm_{\sg^n(\om)}(\cL_{\om}^n\ind_\om).
	\end{align*}
\end{lemma}
\begin{proof}
	For each $f\in\BV(I)$ with $f\geq 0$, $k\in\NN$, and $x\in D_{\sg^{n+k}(\om),n}$ we have
	\begin{align*}
		\frac{\cL_{\sg^k(\om)}^n\lt(\cL_{\om}^kf\rt)(x)}
		{\cL_{\sg^k(\om)}^n\lt(\ind_{\sg^k(\om)}\rt)(x)}
		&=
		\frac{\cL_{\sg^n(\om)}^k\lt(\cL_{\om}^nf\rt)(x)}
		{\cL_{\sg^k(\om)}^n\lt(\ind_{\sg^k(\om)}\rt)(x)}
		\\
		&=
		\frac{\cL_{\sg^n(\om)}^k\lt(\hat D_{\sg^n(\om),n}\cdot\frac{\cL_{\om}^nf}{\cL_{\om}^n\ind_\om}\cdot \cL_{\om}^n\ind_\om\rt)(x)}
		{\cL_{\sg^k(\om)}^n\lt(\ind_{\sg^k(\om)}\rt)(x)}
		\\
		&\geq
		\frac{\cL_{\sg^k(\om)}^n\lt(\cL_{\om}^k\ind_\om\rt)(x)}
		{\cL_{\sg^k(\om)}^n\lt(\ind_{\sg^k(\om)}\rt)(x)}
		\cdot\inf_{D_{\sg^n(\om),n}}
		\frac{\cL_{\om}^n\lt(f\rt)}
		{\cL_{\om}^n\lt(\ind_\om\rt)}.
	\end{align*}
	Taking the infimum over $x\in D_{\sg^{n+k}(\om),n}$ and letting $n\to\infty$ gives
	\begin{align}
		\Lm_{\sg^k(\om)}\lt(\cL_{\om}^kf\rt)
		&\geq 
		\Lm_{\sg^k(\om)}\lt(\cL_{\om}^k\ind_\om\rt)\cdot\Lm_{\om}(f),
		\label{eq1 LMD l3.6}
	\end{align}
	proving the first claim. Now to see the second claim we note that as \eqref{eq1 LMD l3.6} holds for all $\om\in\Om$ with $k=1$, we must also have
	\begin{align}
		\Lm_{\sg^{n+1}(\om)}\lt(\cL_{\sg^n(\om)}f\rt)&\geq \Lm_{\sg^{n+1}(\om)}\lt(\cL_{\sg^n(\om)}\ind_{\sg^n(\om)}\rt)\cdot\Lm_{\sg^n(\om)}(f)
		=\rho_{\sg^n(\om)}\cdot\Lm_{\sg^n(\om)}(f)
		\label{eq1 LMD l3.6 all fibers}
	\end{align}
	for any $f\in\BV(I)$ and each $n\in\NN$. Proceeding via induction, using \eqref{eq1 LMD l3.6} as the base case, we now suppose that
	\begin{align}\label{eq: induction step}
		\Lm_{\sg^n(\om)}\lt(\cL_{\om}^nf\rt)\geq \rho_{\om}^{n}\cdot\Lm_{\om}(f)
	\end{align}
	holds for $n\geq 1$.
	Using \eqref{eq1 LMD l3.6 all fibers} and \eqref{eq: induction step}, we see
	\begin{align*}
		\Lm_{\sg^{n+1}(\om)}\lt(\cL_{\sg^n(\om)}(\cL_{\om}^nf)\rt)
		&\geq
		\rho_{\sg^n(\om)}\cdot\Lm_{\sg^n(\om)}(\cL_{\om}^nf)
		\\
		&\geq
		\rho_{\om}^{n+1}\cdot\Lm_{\om}(f).
	\end{align*}	
	Considering $f=\ind_\om$ proves the final claim, and thus we are done. 
\end{proof}

Define the normalized operator $\cL_\om:L^1(\nu_{\om,c})\to L^1(\nu_{\sg(\om),c})$ by 
\begin{align}\label{eq: def of tilde cL norm op}
	\~\cL_\om f:=\rho_\om^{-1}\cL_\om f; 
	\qquad f\in L^1(\nu_{\om,c}).
\end{align}
In light of Lemma~\ref{LMD l3.6}, for each $\om\in\Om$, $n\in\NN$, and $f\in\BV(I)$ we have that 
\begin{align}\label{eq: equivariance prop of Lm}
	\Lm_\om(f)\leq \Lm_{\sg^n(\om)}\lt(\~\cL_\om^n f\rt) 
\end{align}
Now, considering the normalized operator, we arrive at the following immediate corollary.
\begin{corollary}\label{cor: normalized LY ineq 1}
	For all $\om\in\Om$, all $f\in\BV(X_\om)$, and all $n\in\NN$ we have 
	\begin{align*}
	\var(\~\cL_{\om}^nf)\leq Q_{\om}^{(n)}\var(f)+K_{\om}^{(n)}\Lm_{\om}(|f|).
	\end{align*}
\end{corollary}

\begin{definition}
	Since $Q_\om^{(n)}\to 0$ exponentially fast by our assumption \eqref{cond Q1}, we let $N_*\in\NN$ be the minimum integer $n\geq 1$ such that
	\begin{align}\label{eq: def of N}
	-\infty< \int_\Om \log Q_\om^{(n)} dm(\om) <0, 
	\end{align}
	and we define the number
	\begin{align}\label{eq: def of ta}
	\ta:=-\frac{1}{N_*}\int_\Om \log Q_\om^{(N_*)} dm(\om).
	\end{align}
\end{definition}

\begin{remark}\label{rem: alternate hypoth}
	As we are primarily interested in pushing forward in blocks of length $N_*$ we are able to weaken two or our main hypotheses. In particular, we may replace \eqref{cond Q2} and \eqref{cond Q3} with the following: 
	\begin{enumerate}
	\item[\mylabel{Q2'}{cond Q2'}] We have $\log\xi_\om^{(N_*)}\in L^1(m)$.
	
	\item[\mylabel{Q3'}{cond Q3'}] We have $\log\dl_{\om,N_*}\in L^1(m)$, where $\dl_{\om,n}$ is defined by \eqref{eq: def of dl_om,n}. 
	\end{enumerate}
\end{remark}

In light of Corollary~\ref{cor: normalized LY ineq 1} we may now find an appropriate upper bound for the BV norm of the normalized transfer operator. 
\begin{lemma}\label{lem: buzzi LY1}
	There exists a measurable function $\om\mapsto L_{\om}\in(0,\infty)$ with $\log L_{\om}\in L^1_m(\Om)$ such that for all $f\in\BV(I)$ and each $1\leq n\leq N_*$ we have 
	\begin{align}\label{eq: BV norm bound using L}
		\norm{\~\cL_{\om}^n f}_\BV \leq L_{\om}^n\lt(\var(f)+\Lm_{\sg^n(\om)}\lt(\~\cL_{\om}^n f\rt)\rt).
	\end{align}
where 
\begin{align*}
	L_\om^n=L_\om L_{\sg(\om)}\cdots L_{\sg^{n-1}(\om)}\geq 6^n.
\end{align*}
\end{lemma}

\begin{proof}
	Corollary~\ref{cor: normalized LY ineq 1} and \eqref{eq: equivariance prop of Lm} give 
	\begin{align*}
		\norm{\~\cL_{\om}^n f}_\BV
		&=
		\var(\~\cL_{\om}^n f)+\norm{\~\cL_{\om}^n f}_\infty
		\leq 
		2\var(\~\cL_{\om}^n f)+\Lm_{\sg^n(\om)}\lt(\~\cL_{\om}^n f\rt)
		\\
		&\leq
		2\lt(Q_{\om}^{(n)}\var(f)+K_{\om}^{(n)}\Lm_{\om}(|f|)\rt)+\Lm_{\sg^n(\om)}\lt(\~\cL_{\om}^n f\rt)
		\\
		&\leq 
		2Q_{\om}^{(n)}\var(f)+\lt(2K_{\om}^{(n)}+1\rt)\Lm_{\sg^n(\om)}\lt(\~\cL_{\om}^n f\rt)
	\end{align*}
	Now, set
	\begin{align*}
		\~L_{\om}^{(n)}:=\max\set{6, 2Q_{\om}^{(n)}, 2K_{\om}^{(n)}+1}
	\end{align*}
	Finally, setting 
	\begin{align}\label{eq: defn of L_om^n}
		L_\om:=\max\set{\~L_{\om}^{(j)}: 1\leq j\leq N_*}
	\end{align}
	and 
	\begin{align*}
		L_\om^n:=\prod_{j=0}^{n-1}L_{\sg^j(\om)}
	\end{align*}
	for all $n\geq 1$ suffices. The $\log$-integrability of $L_\om^n$ follows from Proposition~\ref{prop: log integr of Q and K}.
\end{proof}
We now define the number $\zt>0$ by 
\begin{align}\label{eq: def of zt}
	\zt:=\frac{1}{N_*}\int_\Om \log L_\om^{N_*} dm(\om). 
\end{align}

The constants $B_\om^{(n)}$ and $K_\om^{(n)}$ in the Lasota-Yorke inequalities from Lemma~\ref{ly ineq} and  Corollary~\ref{cor: normalized LY ineq 1} grow to infinity with $n$, making them difficult to use. Furthermore, the rate of decay of the $Q_\om^{(n)}$ in Corollary~\ref{cor: normalized LY ineq 1} may depend on $\om$. To remedy these difficulties we prove another, more useful, Lasota-Yorke inequality in the style of Buzzi \cite{buzzi_exponential_1999}.

\begin{proposition}\label{prop: LY ineq 2}
	For each $\ep>0$ there exists a measurable, $m$-a.e. finite function $C_\ep(\om)>0$ such that for $m$-a.e. $\om\in\Om$, each $f\in\BV(I)$, and all $n\in\NN$ we have 
	\begin{align*}
	\var(\~\cL_{\sg^{-n}(\om)}^nf)\leq C_\ep(\om)e^{-(\ta-\ep)n}\var(f)+C_\ep(\om)\Lm_\om(\~\cL_{\sg^{-n}(\om)}^nf).
	\end{align*}
\end{proposition} 
As the proof of Proposition~\ref{prop: LY ineq 2} follows similarly to that of Proposition~4.8 of \cite{AFGTV21}, using \eqref{eq: equivariance prop of Lm} to obtain $\Lm_\om(\~\cL_{\sg^{-n}(\om)}^nf)$ rather than $\Lm_{\sg^{-n}(\om)}(f)$, we leave it to the dedicated reader. 

\section{Cone Invariance on Good Fibers}\label{Sec: good fibers}\label{sec:good}

In this section we follow Buzzi's approach \cite{buzzi_exponential_1999}, and describe the good behavior across a large measure set of fibers. In particular, we will show that, for sufficiently many iterates $R_*$, the normalized transfer operator $\~\cL_\om^{R_*}$ uniformly contracts the cone $\sC_{\om,a}$ on ``good'' fibers $\om$ for cone parameters $a>0$ sufficiently large. 
Recall that the numbers $\ta$ and $\zt$ are given by 
\begin{align*}
\ta:=-\frac{1}{N_*}\int_\Om \log Q_\om^{(N_*)} dm(\om)>0
\quad \text{ and } \quad
\zt:=\frac{1}{N_*}\int_\Om \log L_\om^{N_*} dm(\om)>0. 
\end{align*}
Note that Lemma~\ref{lem: buzzi LY1} and the ergodic theorem imply that 
\begin{align}\label{eq: zt geq log 6}
\log 6\leq \zt = \lim_{n\to\infty}\frac{1}{nN_*}\sum_{k=0}^{n-1}\log L_{\sg^{kN_*}(\om)}^{N_*}.
\end{align}
The following definition is adapted from \cite[Definition~2.4]{buzzi_exponential_1999}.
\begin{definition}
	We will say that $\omega$ is \textit{good} with respect to the numbers $\ep$, $a$, $B_*$, and $R_a=q_aN_*$ if the following hold:
	\begin{flalign} 
	& B_*q_ae^{-\frac{\ta}{2}R_a}\leq \frac{1}{3}.
	\tag{G1}\label{G1} 
	&\\
	& \frac{1}{R_a}\sum_{k=0}^{\sfrac{R_a}{N_*}-1}\log L_{\sg^{kN_*}(\om)}^{N_*} 
	\in[\zt-\ep,\zt+\ep], 
	\tag{G2}\label{G2} 
	\end{flalign}
\end{definition}
Now, we denote
\begin{align}\label{eq: def of ep_0}
\ep_0:=\min\set{1, \frac{\ta}{2}}.
\end{align}
The following lemma describes the prevalence of the good fibers as well as how to find them. 
\begin{lemma}\label{lem: constr of Om_G}
	Given $\ep<\ep_0$ and $a>0$, there exist parameters $B_*$ and  $R_a$ (both of which depend on $\ep$) such that there is a set $\Om_G\sub \Om$ of good $\om$ with $m(\Om_G)\geq 1-\sfrac{\ep}{4}$. 
\end{lemma}
\begin{proof}
	We begin by letting 
	\begin{align}\label{eq: def of Om_1}
	\Om_1=\Om_1(B_*):=\set{\om\in\Om: C_\ep(\om)\leq B_*},
	\end{align}
	where $C_\ep(\om)>0$ is the $m$-a.e. finite measurable constant coming from Proposition~\ref{prop: LY ineq 2}.
	Choose $B_*$ sufficiently large such that $m(\Om_1)\geq 1-\sfrac{\ep}{8}$. 
	Noting that $\ep<\ta/2$ by \eqref{eq: def of ep_0}, we set $R_0=q_0N_*$ and choose $q_0$ sufficiently large such that 
	\begin{align*}
	B_*q_0e^{-(\ta-\ep)R_0}\leq B_*q_0e^{-\frac{\ta}{2}R_0}\leq\frac{1}{3}.
	\end{align*}
	Now let $q_1\geq q_0$ and define the set 
	\begin{align*}
	\Om_2=\Om_2(q_1):=\set{\om\in\Om: \eqref{G2} \text{ holds for the value } R_1=q_1N_* }.
	\end{align*}
	Now choose $q_a\geq q_1$ such that $m(\Om_2(q_a))\geq 1-\sfrac{\ep}{8}$. Set $R_a:=q_aN_*$. Set 
	\begin{align}\label{eq: def of good set}
		\Om_G:=\Om_2\cap \sg^{-R_a}(\Om_1).
	\end{align}
	Then $\Om_G$ is the set of all $\om\in\Om$ which are good with respect to the numbers $B_*$ and $R_a$, and 
	$m(\Om_G)\geq 1-\sfrac{\ep}{4}$.  
\end{proof}

In what follows, given a value $B_*$, we will consider cone parameters 
\begin{align}\label{eq cone param a}
a\geq a_0:=6B_* 
\end{align}
and we set 
\begin{align}\label{eq: def of R*}
	q_*=q_{a_0} 
	\quad\text{ and }\quad 
	R_*:=R_{a_0}=q_*N_*.
\end{align}

Note that \eqref{G1} together with Proposition~\ref{prop: LY ineq 2} implies that, for $\ep<\ep_0$ and $\om\in\Om_G$, we have 
\begin{align}
	\var(\~\cL_\om^{R_*}f)
	&\leq 
	B_*e^{-(\ta-\ep)R_*}\var(f)+B_*\Lm_{\sg^{R_*}(\om)}(\~\cL_\om^{R_*} f)
	\nonumber\\
	&\leq 
	B_*q_*e^{-\frac{\ta}{2}R_*}\var(f)+B_*\Lm_{\sg^{R_*}(\om)}(\~\cL_\om^{R_*} f)
	\nonumber\\
	&\leq 
	\frac{1}{3}\var(f)+B_*\Lm_{\sg^{R_*}(\om)}(\~\cL_\om^{R_*} f).
	\label{G1 cons}
\end{align}

The next lemma shows that the normalized operator is a contraction on the fiber cones $\sC_{\om,a}$ and that the image has finite diameter.
\begin{lemma}\label{lem: cone invariance for good om}
	If $\om$ is good with respect to the numbers $\ep$, $a_0$, $B_*$, and $R_*$, then for each $a\geq a_0$ we have 
	\begin{align*}
	\~\cL_{\om}^{R_*}(\sC_{\om,a})\sub \sC_{\sg^{R_*}(\om), \sfrac{a}{2}}\sub \sC_{\sg^{R_*}(\om),a}.
	\end{align*}
\end{lemma}
\begin{proof}
	For $\om$ good and $f\in\sC_{\om,a}$, \eqref{G1 cons} and \eqref{eq cone param a} give
	\begin{align*}
	\var(\~\cL_{\om}^{R_*} f)
	&\leq 
	\frac{1}{3}\var(f) + B_*\Lm_{\sg^{R_*}(\om)}(\~\cL_\om^{R_*}f) 
	\\&
	\leq 
	\frac{a}{3}\Lm_{\om}(f) + \frac{a}{6}\Lm_{\sg^{R_*}(\om)}(\~\cL_\om^{R_*}f) 
	\\
	&
	\leq\frac{a}{2}\Lm_{\sg^{R_*}(\om)}(\~\cL_\om^{R_*} f).
	\end{align*}
	Hence we have 
	\begin{align*}
	\~\cL_{\om}^{R_*}(\sC_{\om,a})\sub \sC_{\sg^{R_*}(\om), \sfrac{a}{2}}\sub \sC_{\sg^{R_*}(\om),a}
	\end{align*}
	as desired.
\end{proof}

\section{Density Estimates and Cone Invariance on Bad Fibers}\label{sec:bad}
In this section we recall the notion of ``bad'' fibers from \cite{AFGTV21, buzzi_exponential_1999}. We show that for fibers in the small measure set, $\Om_B:=\Om\bs\Om_G$, the cone $\sC_{\om,a}$ of positive functions is invariant after sufficiently many iterations for sufficiently large parameters $a>0$. We accomplish this by introducing the concept of bad blocks (coating intervals), which we then show make up a relatively small portion of an orbit. 
As the content of this section is adapted from the closed dynamical setting of Section 7 of \cite{AFGTV21}, we do not provide proofs.

Recall that $R_*$ is given by \eqref{eq: def of R*}. Following Section~7 of \cite{AFGTV21}, and using the same justifications therein, we define the measurable function $y_*:\Om\to\NN$ so that 
\begin{align}\label{eq: def of y_*}
	0\leq y_*(\om)<R_*
\end{align}
is the smallest integer such that for either choice of sign $+$ or $-$ we have 
\begin{flalign} 
& \lim_{n\to\infty} \frac{1}{n}\#\set{0\leq k< n: \sg^{\pm kR_*+y_*(\om)}(\om)\in\Om_G} >1-\ep,
\label{def y*1} 
&\\
& \lim_{n\to\infty} \frac{1}{n}\#\set{0\leq k< n: C_\ep\lt(\sg^{\pm kR_*+y_*(\om)}(\om)\rt)\leq B_*} >1-\ep.
\label{def y*2} 
\end{flalign}
Clearly, $y_*:\Om\to\NN$ is a measurable function such that 
\begin{flalign} 
& y_*(\sg^{y_*(\om)}(\om))=0,
\label{prop j*1} 
&\\
& y_*(\sg^{R_*}(\om))=y_*(\om).
\label{prop j*2} 
\end{flalign}
In particular, \eqref{prop j*1} and \eqref{prop j*2} together imply that 
\begin{align}\label{prop j*3}
y_*(\sg^{y_*(\om)+kR_*}(\om))=0
\end{align}
for all $k\in\NN$.
Let 
\begin{align}\label{def Gm}
\Gm(\om):=q_*\prod_{k=0}^{q_*-1} L_{\sg^{kN_*}(\om)}^{N_*},
\end{align}
where $q_*$ is given by \eqref{eq: def of R*},
and for each $\om\in\Om$, given $\ep>0$, we define the \textit{coating length} $\ell(\om)=\ell_\ep(\om)$ as follows:
\begin{itemize}
	\item if $\om\in\Om_G$ then set $\ell(\om):=1$, 
	\item if $\om\in\Om_B$ then 
	\begin{align}\label{def: coating length}
	\ell(\om):=\min\set{n\in\NN: \frac{1}{n}\sum_{0\leq k< n} \lt(\ind_{\Om_B}\log \Gm\rt)(\sg^{kR_*}(\om)) \leq \log q_*+\zt R_*\sqrt{\ep}},
	\end{align}
	where $\zt$ is as in \eqref{eq: def of zt}.
	If the minimum is not attained we set $\ell(\om)=\infty$.
\end{itemize}
Since $L_\om^{N_*}\geq 6^{N_*}$ by Lemma~\ref{lem: buzzi LY1}, we must have that 
\begin{align}\label{eq: Bm geq 6^R}
\Gm(\om)\geq q_*6^{R_*}
\end{align}
for all $\om\in\Om$. 
It follows from Lemma~\ref{lem: buzzi LY1} (applied repeatedly to $q_*$ blocks of length $N_*$) that for all $\om\in\Om$ we have 
\begin{align}\label{eq: LY ineq for bad fibers}
\var(\~\cL_\om^{R_*}f)
&\leq
\lt(\prod_{k=0}^{q_*-1} L_{\sg^{kN_*}(\om)}^{N_*}\rt)\var(f)
+ 
\sum_{j=0}^{q_*-1}\lt(\prod_{k=j}^{q_*-1} L_{\sg^{kN_*}(\om)}^{N_*} \rt)\Lm_{\sg^{R_*}(\om)}(f)
\nonumber\\
&\leq \Gm(\om)(\var(f)+\Lm_{\sg^{R_*}(\om)}(f)).  
\end{align}
Furthermore, if $\om\in\Om_G$ it follows from \eqref{G2} that 
\begin{align}\label{eq: up and low bds for Gamma on good om}
\log q_*+R_*(\zt-\ep)\leq \log\Gm(\om)\leq \log q_*+R_*(\zt+\ep).
\end{align}
The following proposition collects together some of the key properties of the coating length $\ell(\om)$. 
\begin{proposition}\label{prop: ell(om) props}
	For all $\ep>0$ sufficiently small the number $\ell(\om)$ satisfies the following.
	\begin{flalign*}
	& \text{For }m\text{-a.e. } \om\in\Om \text{ such that } y_*(\om)=0 \text{ we have } \ell(\om)<\infty,
	\tag{i}\label{prop: ell(om) props item i}
	&\\
	& \text{If }\om\in\Om_B \text{ then }\ell(\om)\geq 2.
	\tag{ii} \label{prop: ell(om) props item ii}
	\end{flalign*}
\end{proposition}
\begin{remark}\label{rem: y*=0 implies ell finite}
	Given $\om_0\in\Om$, for each $j\geq 0$ let $\om_{j+1}=\sg^{\ell(\om_j)R_*}(\om_j)$. 
	As a consequence of Proposition~\ref{prop: ell(om) props} \eqref{prop: ell(om) props item i} and \eqref{prop j*2}, we see that for $m$-a.e. $\om_0\in\Om$ with $y_*(\om_0)=0$, we must have that $\ell(\om_j)<\infty$ for all $j\geq 0$.
\end{remark}
\begin{definition}\label{def: coating intervals}	
	We will call a (finite) sequence $\om, \sg(\om), \dots, \sg^{\ell(\om)R_*-1}(\om)$ of $\ell(\om)R_*$ fibers a \textit{good block} (originating at $\om$) if $\om\in\Om_G$ (which implies that $\ell(\om)=1$). If, on the other hand, $\om\in\Om_B$ we call such a sequence a \textit{bad block}, or coating interval, originating at $\om$.
\end{definition}
For $\ep>0$ sufficiently small we have that $\sfrac{\zt \sqrt{\ep}}{\log 6}<1$,
and so we may define the number
\begin{align*}
	\gm(\ep):=\frac{\log q_*+\zt R_*\sqrt{\ep}}{\log q_*+R_*\log 6}<1.
\end{align*}
Since $q_*\to\infty$ as $\ep\to 0$ (since $q_*$ was chosen in \eqref{eq: def of R*} after Lemma~\ref{lem: constr of Om_G} depending on $\ep$) and since $R_*=q_*N_*$, for $\ep>0$ sufficiently small there exists $\gm<1$ such that 
\begin{align}\label{eq: def of gm constant}
	\gm(\ep)<\gm<1.
\end{align}

We now wish to show that the normalized operator $\~\cL_\om$ is weakly contracting (i.e. non-expanding) on the fiber cones $\sC_{\om,a}$ for sufficiently large values of $a>a_0$. We obtain this cone invariance on blocks of length $\ell(\om)R_*$, however in order to obtain cone contraction with a finite diameter image we will have to travel along several such blocks. For this reason we introduce the following notation. 

Given $\om\in\Om$ with $y_*(\om)=0$ for each $k\geq 1$ we define the length
\begin{align*}
	\Sg_\om^{(k)}:=\sum_{j=0}^{k-1} \ell(\om_j)R_*
\end{align*}
where $\om_0:=\om$ and for each $j\geq 1$ we set $\om_j:=\sg^{\Sg_\om^{(j-1)}}(\om)$.
This construction is justified as we recall from Proposition~\ref{prop: ell(om) props} that for $m$-a.e. $\om\in\Om$ with $y_*(\om)=0$ we have that $\ell(\om)<\infty$. 
The next lemma was adapted from Lemma~7.5 of \cite{AFGTV21}.
\begin{lemma}\label{lem: cone cont for coating blocks}
	For $\ep>0$ sufficiently small, each $N\in\NN$, and $m$-a.e. $\om\in\Om$ with $y_*(\om)=0$ we have that 
	\begin{align}\label{eq: important coating block ineq}	
		\var\lt(\~\cL_\om^{\Sg_\om^{(N)}}f\rt)
		&\leq
		\lt(\frac{1}{3}\rt)^{\Sg_\om^{(N)}/R_*}\var(f)
		+
		\frac{a_*}{6}\Lm_{\sg^{\Sg_\om^{(N)}}(\om)}(\~\cL_\om^{\Sg_\om^{(N)}} f).
	\end{align}
	Moreover, we have that 
	\begin{align}\label{eq: cone inv for large bad blocks}
		\~\cL_\om^{\Sg_\om^{(N)}}(\sC_{\om,a_*})\sub \sC_{\sg^{\Sg_\om^{(N)}}(\om), \sfrac{a_*}{2}},
	\end{align}
	where 
	\begin{align}\label{eq: def of a_*}
		a_*=a_*(\ep):=2a_0q_*e^{\zt R_*\sqrt{\ep}}=12B_*q_*e^{\zt R_*\sqrt{\ep}}.
	\end{align}
\end{lemma}

\begin{proof}
	 Throughout the proof we will denote $\ell_i=\ell(\om_i)$ and $L_i=\sum_{k=0}^{i-1}\ell_k$ for each $0\leq i<N$. Then $\Sg_\om^{(N)}=L_NR_*$.
	 Using \eqref{G1 cons} on good fibers and \eqref{eq: LY ineq for bad fibers} on bad fibers, for any $p\geq 1$ and $f\in\sC_{\om,+}$ we have 
	\begin{align}\label{eq: var coating length}
	\var(\~\cL_\om^{pR_*} f)
	&\leq
	\lt(\prod_{j=0}^{p-1} \Phi_{\sg^{jR_*}(\om)}^{(R_*)}\rt)\var(f)
	+ 
	\sum_{j=0}^{p-1}\lt(D_{\sg^{jR_*}(\om)}^{(R_*)}\cdot \prod_{k=j+1}^{p-1}\Phi_{\sg^{kR_*}(\om)}^{(R_*)}\rt)\Lm_{\sg^{pR_*}(\om)}(\~\cL_\om^{pR_*} f),
	\end{align}
	where 
	\begin{equation}\label{eq: def of Phi_tau^R}
	\Phi_\tau^{(R_*)}=
	\begin{cases}
	B_* e^{-(\ta-\ep)R_*} &\text{for } \tau\in\Om_G\\
	\Gm(\tau) &\text{for } \tau\in\Om_B
	\end{cases}
	\end{equation}
	and 
	\begin{equation}\label{eq: def of D_tau^R}
	D_\tau^{(R_*)}=
	\begin{cases}
	B_* &\text{for } \tau\in\Om_G\\
	\Gm(\tau) &\text{for } \tau\in\Om_B.
	\end{cases}
	\end{equation}
	For any $0\leq i<N$ and  $0\leq j<\ell_i$ we can write 
	\begin{align*}
	\sum_{0\leq k< \ell_i}\lt(\ind_{\Om_B}\log \Gm\rt)(\sg^{kR_*}(\om_i))
	=
	\sum_{0\leq k< j }\lt(\ind_{\Om_B}\log \Gm\rt)(\sg^{kR_*}(\om_i))
	+
	\sum_{j\leq k< \ell_i }\lt(\ind_{\Om_B}\log \Gm\rt)(\sg^{kR_*}(\om_i)).
	\end{align*}
	The definition of $\ell(\om_i)$, \eqref{def: coating length}, then implies that 
	\begin{align*}
	\frac{1}{j}\sum_{0\leq k< j }\lt(\ind_{\Om_B}\log \Gm\rt)(\sg^{kR_*}(\om_i))
	> 
	\log q_*+\zt R_*\sqrt{\ep},
	\end{align*}
	and consequently that 
	\begin{align}\label{eq: avg sum log Gm for bad om}
	\frac{1}{\ell_i-j}\sum_{j\leq k< \ell_i}\lt(\ind_{\Om_B}\log \Gm\rt)(\sg^{kR_*}(\om_i))
	\leq
	\log q_*+\zt R_*\sqrt{\ep}. 
	\end{align}
	Now, using \eqref{eq: Bm geq 6^R}, \eqref{eq: avg sum log Gm for bad om}, and \eqref{eq: def of gm constant} we see that for $\ep$ sufficiently small, the proportion of bad blocks is given by 
	\begin{align}
	\frac{1}{\ell_i-j}\#\set{j\leq k<\ell_i: \sg^{kR_*}(\om_i)\in\Om_B}
	&=
	\frac{1}{\ell_i-j}\sum_{j\leq k<\ell_i} \lt(\ind_{\Om_B}\rt)(\sg^{kR_*}(\om_i))
	\nonumber\\
	&\leq
	\frac{1}{(\ell_i-j)R_*\log 6}\sum_{j\leq k<\ell_i} \lt(\ind_{\Om_B}\log \Gm\rt)(\sg^{kR_*}(\om_i))
	\leq \gm.
	\label{eq: proportion bad blocks}
	\end{align}
	In view of \eqref{eq: def of Phi_tau^R}, using \eqref{eq: avg sum log Gm for bad om}, \eqref{eq: proportion bad blocks}, for any $0\leq i<N$ and $0\leq j<\ell_i$ we have 
	\begin{align}
	\prod_{k=j}^{\ell_i-1} \Phi_{\sg^{kR_*}(\om_i)}^{(R_*)}
	&=
	\prod_{\substack{j\leq k<\ell_i \\ \sg^{kR_*}(\om_i)\in\Om_G}}
	B_*e^{-(\ta-\ep)R_*}
	\cdot 
	\prod_{\substack{j\leq k<\ell_i \\ \sg^{kR_*}(\om_i)\in\Om_B}}
	\Gm(\sg^{kR_*}(\om_i))
	\nonumber\\
	&\leq 
	\lt(B_*e^{-(\ta-\ep)R_*}\rt)^{(1-\gm)(\ell_i-j)}\cdot \exp\lt(\lt(\log q_*+\zt R_*\sqrt{\ep}\rt)(\ell_i-j)\rt)
	\nonumber\\
	&=
	\lt(B_*^{1-\gm}q_*\exp\lt(\lt(\zt \sqrt{\ep}-(\ta-\ep)(1-\gm)\rt)R_*\rt)\rt)^{\ell_i-j}
	\nonumber\\
	&< 
	\lt(B_*q_*\exp\lt(\lt(\zt \sqrt{\ep}-(\ta-\ep)(1-\gm)\rt)R_*\rt)\rt)^{\ell_i-j}.
	\label{eq: est for coating lengths 1}
	\end{align}
	Now for any $0\leq j<L_N$ there must exist some $0\leq i_0<N$ and some $0\leq j_0<\ell_{i_0+1}$
	such that $L_{i_0-1}+j_0= j<L_{i_0}$. Thus, using \eqref{eq: est for coating lengths 1} we can write
	\begin{align}
		\prod_{k=j}^{L_N-1} \Phi_{\sg^{kR_*}(\om)}^{(R_*)}
		&=
		\prod_{k=j_0}^{\ell_{i_0+1}-1} \Phi_{\sg^{kR_*}(\om_{i_0})}^{(R_*)}
		\cdot
		\prod_{i=i_0+1}^{N-1} \prod_{k=0}^{\ell_i-1} \Phi_{\sg^{kR_*}(\om_i)}^{(R_*)}
		\nonumber\\
		&< 
		\lt(B_*q_*\exp\lt(\lt(\zt \sqrt{\ep}-(\ta-\ep)(1-\gm)\rt)R_*\rt)\rt)^{\ell_{i_0+1}-j_0}
		\cdot
		\nonumber\\
		&\qquad\cdot
		\prod_{i=i_0+1}^{N-1} \lt(B_*q_*\exp\lt(\lt(\zt \sqrt{\ep}-(\ta-\ep)(1-\gm)\rt)R_*\rt)\rt)^{\ell_i}
		\nonumber\\
		&=
		\lt(B_*q_*\exp\lt(\lt(\zt \sqrt{\ep}-(\ta-\ep)(1-\gm)\rt)R_*\rt)\rt)^{L_N-j}.
		\label{eq: est for coating lengths 1 full}
	\end{align}

	Now, since $B_*,\Gm(\om)\geq 1$ for all $\om\in\Om$, using \eqref{eq: def of D_tau^R} and \eqref{eq: avg sum log Gm for bad om}, we have that for $0\leq i<N$ and $0\leq j<\ell_i$ 
	\begin{align}\label{eq: est for coating lengths 2}
	D_{\sg^{jR_*}(\om_i)}^{(R_*)}
	\leq
	B_*\Gm(\sg^{jR_*}(\om_i))
	\leq 
	B_*\cdot\prod_{\substack{j\leq k< \ell_i \\ \sg^{kR_*}(\om_i)\in\Om_B}} \Gm(\sg^{kR_*}(\om_i))
	\leq
	B_*\lt(q_*e^{\zt R_*\sqrt{\ep}}\rt)^{(\ell_i-j)}.		
	\end{align}
	Similarly to the reasoning used to obtain \eqref{eq: est for coating lengths 1 full}, for any $0\leq j<L_N$ we see that we can improve \eqref{eq: est for coating lengths 2} so that we have 
	\begin{align}\label{eq: est for coating lengths 2 full}
	D_{\sg^{jR_*}(\om_i)}^{(R_*)}
	\leq
	B_*\lt(q_*e^{\zt R_*\sqrt{\ep}}\rt)^{L_N-j}.		
\end{align}
	Thus, inserting \eqref{eq: est for coating lengths 1 full} and \eqref{eq: est for coating lengths 2 full} into \eqref{eq: var coating length} (with $p=L_N$) we see that
	\begin{align*}
	&\var\lt(\~\cL_\om^{\Sg_\om^{(N)}} f\rt)
	\leq
	\lt(\prod_{j=0}^{L_N-1} \Phi_{\sg^{jR_*}(\om)}^{(R_*)}\rt)\var(f)
	+ 
	\sum_{j=0}^{L_N-1}\lt(D_{\sg^{jR_*}(\om)}^{(R_*)}\cdot \prod_{k=j+1}^{L_N-1}\Phi_{\sg^{kR_*}(\om)}^{(R_*)}\rt)\Lm_{\sg^{\Sg_\om^{(N)}}(\om)}\lt(\~\cL_\om^{\Sg_\om^{(N)}} f\rt),
	\\
	&\leq 
	\lt(B_*q_*\exp\lt(\lt(\zt \sqrt{\ep}-(\ta-\ep)(1-\gm)\rt)R_*\rt)\rt)^{L_N}\var(f)
	\\
	&\qquad
	+
	\Lm_{\sg^{\Sg_\om^{(N)}}(\om)}\lt(\~\cL_\om^{\Sg_\om^{(N)}} f\rt)\sum_{j=0}^{L_N-1}
	B_*\lt(q_*e^{\zt R_*\sqrt{\ep}}\rt)^{(L_N-j)}
	\cdot 
	\lt(B_*q_*\exp\lt(\lt(\zt \sqrt{\ep}-(\ta-\ep)(1-\gm)\rt)R_*\rt)\rt)^{L_N-j-1}
	\\
	&
	=
	\lt(B_*q_*\exp\lt(\lt(\zt \sqrt{\ep}-(\ta-\ep)(1-\gm)\rt)R_*\rt)\rt)^{L_N}\var(f)
	\\
	&\qquad
	+
	B_*q_*e^{\zt R_*\sqrt{\ep}}\cdot \Lm_{\sg^{\Sg_\om^{(N)}}(\om)}\lt(\~\cL_\om^{\Sg_\om^{(N)}} f\rt)\sum_{j=0}^{L_N-1}		
	\lt(B_*q_*\exp\lt(\lt(2\zt \sqrt{\ep}-(\ta-\ep)(1-\gm)\rt)R_*\rt)\rt)^{L_N-j-1}.
	\end{align*}
	Therefore, taking $\ep>0$ sufficiently small\footnote{Any $\ep<\min\set{\lt(\frac{\log 6}{4\zt}\rt)^2, \lt(\frac{\ta}{8\zt}\rt)^2}$ such that $\frac{\sqrt{\ep}\zt}{2\log 6}\leq \gm$, which implies $-\frac{\ta}{2}> 2\sqrt{\ep}\zt-(\ta-\ep)(1-\gm) > \sqrt{\ep}\zt-(\ta-\ep)(1-\gm)$, will suffice; see Observation~7.4 of \cite{AFGTV21} for details.} 
	in conjunction with \eqref{G1}, we have that
	\begin{align*}
	\var\lt(\~\cL_\om^{\Sg_\om^{(N)}}f\rt)
	&\leq
	\lt(B_*q_*e^{-\frac{\ta}{2}R_*}\rt)^{L_N}\var(f)
	+
	B_*q_*e^{\zt R_*\sqrt{\ep}}\cdot \Lm_{\sg^{\Sg_\om^{(N)}}(\om)}\lt(\~\cL_\om^{\Sg_\om^{(N)}} f\rt)\sum_{j=0}^{L_N-1}		
	\lt(B_*q_*e^{-\frac{\ta}{2}R_*}\rt)^{L_N-j-1} 
	\\
	&\leq 
	\lt(\frac{1}{3}\rt)^{L_N}\var(f)
	+
	B_*q_*e^{\zt R_*\sqrt{\ep}}\cdot \Lm_{\sg^{\Sg_\om^{(N)}}(\om)}\lt(\~\cL_\om^{\Sg_\om^{(N)}} f\rt)\sum_{j=0}^{L_N-1}		
	\lt(\frac{1}{3}\rt)^{L_N-j-1},
	\end{align*} 
	and so we must have that 
	\begin{align*}
	\var\lt(\~\cL_\om^{\Sg_\om^{(N)}} f\rt)
	&\leq 
	\lt(\frac{1}{3}\rt)^{L_N}\var(f)+2B_*q_*e^{\zt R_*\sqrt{\ep}}\Lm_{\sg^{\Sg_\om^{(N)}}(\om)}\lt(\~\cL_\om^{\Sg_\om^{(N)}} f\rt)
	\\
	&=
	\lt(\frac{1}{3}\rt)^{L_N}\var(f)+\frac{a_*}{6}\Lm_{\sg^{\Sg_\om^{(N)}}(\om)}\lt(\~\cL_\om^{\Sg_\om^{(N)}} f\rt),
	\end{align*}
	which proves the first claim.
	Thus, for any $f\in\sC_{\om,a_*}$ we have that 
	\begin{align*}
	\var(\~\cL_\om^{\Sg_\om^{(N)}} f)
	&\leq 
	\lt(\frac{1}{3}\rt)^{L_N}\Lm_{\om}(f)+\frac{a_*}{6}\Lm_{\sg^{\Sg_\om^{(N)}}(\om)}\lt(\~\cL_\om^{\Sg_\om^{(N)}} f\rt)
	\\
	&\leq
	\frac{a_*}{3}\Lm_{\sg^{\Sg_\om^{(N)}}(\om)}\lt(\~\cL_\om^{\Sg_\om^{(N)}} f\rt)+\frac{a_*}{6}\Lm_{\sg^{\Sg_\om^{(N)}}(\om)}\lt(\~\cL_\om^{\Sg_\om^{(N)}} f\rt), 
	\end{align*}
	where we have used the fact that $a_*> 1$, and consequently we have 
	\begin{align*}
	\~\cL_\om^{\Sg_\om^{(N)}}(\sC_{\om,a_*})\sub \sC_{\sg^{\Sg_\om^{(N)}}(\om),\sfrac{a_*}{2}}\sub \sC_{\sg^{\Sg_\om^{(N)}}(\om),a_*}
	\end{align*}
	as desired.
\end{proof}

The next lemma shows that the total length of the bad blocks take up only a small proportion of an orbit, however before stating the result we establish the following notation. For each  $n\in\NN$ we let $K_n\geq 0$ be the integer such that 
\begin{align}\label{eq: n KR+zt decomp}
n=K_nR_*+h(n)
\end{align}
where $0\leq h(n)<R_*$ is a remainder term. Given $\om_0\in\Om$, let 
\begin{align}\label{eq: om_j notation}
\om_j=\sg^{\ell(\om_{j-1})R_*}(\om_{j-1})
\end{align} 
for each $j\geq 1$. Then for each $n\in\NN$ we can break the $n$-length $\sg$-orbit of $\om_0$ in $\Om$ into $k_{\om_0}(n)+1$ blocks of length $\ell(\om_j)R_*$ (for $0\leq j\leq k_{\om_0}(n)$) plus some remaining block of length $r_{\om_0}(n)R_*$ where $0\leq r_{\om_0}(n)<\ell(\om_{k_{\om_0}(n)+1})$ plus a remainder segment of length $h(n)$, i.e. we can write 
\begin{align}\label{eq: n length orbit in om_j}
n=\sum_{0\leq j\leq k_{\om_0}(n)}\ell(\om_j)R_* +r_{\om_0}(n)R_* + h(n); 
\end{align}
see Figure~\ref{figure: om_j}. We also note that \eqref{eq: n KR+zt decomp} and \eqref{eq: n length orbit in om_j} imply that 
\begin{align}\label{eq: Kn equality}
K_n=\sum_{0\leq j\leq k_{\om_0}(n)}\ell(\om_j) +r_{\om_0}(n).
\end{align}
\begin{figure}[h]
	\centering	
	\begin{tikzpicture}[y=1cm, x=1.5cm, thick, font=\footnotesize]    
	\draw[line width=1.2pt,  >=latex'](0,0) -- coordinate (x axis) (10,0);       
	
	\draw (0,-4pt)   -- (0,4pt) {};
	\node at (0, -.5) {$\om_0$};
	
	\node at (.5, .5) {$\ell(\om_0)R_*$};
	
	\draw (1,-4pt)   -- (1,4pt) {};
	\node at (1,-.5) {$\om_1$};

	\node at (1.625,.5) {$\ell(\om_1)R_*$};
	
	\draw (2.25,-4pt)  -- (2.25,4pt)  {};
	\node at (2.25,-.5) {$\om_2$};
	
	\node at (2.625,.3) {$\cdots$};
	\node at (2.625,-.3) {$\cdots$};
	
	\draw (3,-4pt)  -- (3,4pt)  {};
	\node at (3,-.5) {$\om_{k-1}$};
	
	\node at (3.825,.5) {$\ell(\om_{k-1})R_*$};
	
	\draw (4.75,-4pt)  -- (4.75,4pt)  {};
	\node at (4.75,-.5) {$\om_{k_{\om_0}(n)}$};
	
	\node at (5.75,.5) {$\ell(\om_{k_{\om_0}(n)})R_*$};	
	
	\draw (6.75,-4pt)  -- (6.75,4pt)  {};
	\node at (6.75,-.5) {$\om_{k+1}$};
	
	\node at (7.5,.5) {$r_{\om_0}(n)R_*$};
	
	\draw (8.25,-4pt)  -- (8.25,4pt)  {};
	\node at (8.25,-.5) {$\sg^{K_nR_*}(\om_0)$};
	
	\node at (8.75,.5) {$h(n)$};
	
	\draw (9.25,-4pt)  -- (9.25,4pt)  {};
	\node at (9.25,-.5) {$\sg^{n}(\om_0)$};
	
	\draw (10,-4pt)  -- (10,4pt)  {};
	\node at (10,-.5) {$\om_{k+2}$};

	\draw[decorate,decoration={brace,amplitude=8pt,mirror}] 
	(6.75,-1)  -- (10,-1) ; 
	\node at (8.125,-1.6){$\ell(\om_{k+1})R_*$};
	
	\draw[decorate,decoration={brace,amplitude=8pt}] 
	(0,1)  -- (9.25,1) ; 
	\node at (4.625,1.6){$n$};
	
	\end{tikzpicture}
	\caption{The decomposition of $n=\sum_{0\leq j\leq k_{\om_0}(n)}\ell(\om_j)R_*+r_{\om_0}(n)R_* +h(n)$ and the fibers $\om_j$.}
	\label{figure: om_j}
\end{figure}
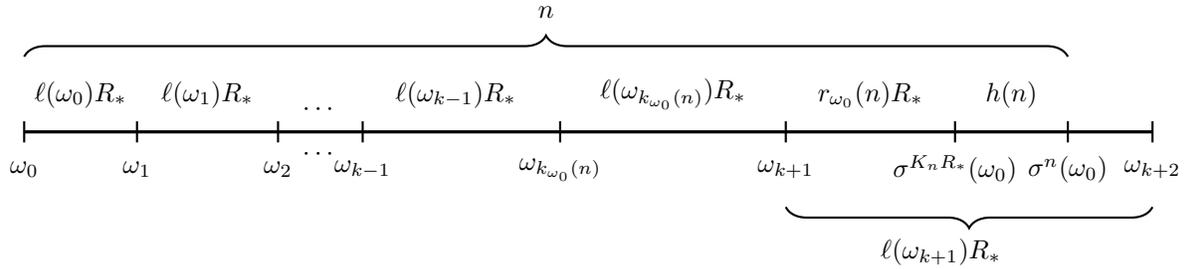
The proof of the following lemma is nearly identical to that of Lemma~7.6 of \cite{AFGTV21}, and therefore it shall be omitted. 
\begin{lemma}\label{lem: proportion of bad blocks}
	There exists a measurable function $N_0:\Om\to\NN$ such that for all $n\geq N_0(\om_0)$ and for $m$-a.e. $\om_0\in\Om$ with $y_*(\om_0)=0$  we have 
	\begin{align*}
	&E_{\om_0}(n):=
	\sum_{\substack{0\leq j \leq k_{\om_0}(n) \\ \om_j\in\Om_B}}\ell(\om_j) +r_{\om_0}(n)
	<
	Y\cdot \ep K_n
	\leq 
	\frac{Y}{R_*}\ep n
	\end{align*}
	where 
	\begin{align}\label{eq: def of Y}
	Y=Y_\ep:=\frac{2(\log q_*+(2+\zt) R_*)}{\log q_*+\zt R_*\sqrt{\ep}},
	\end{align}
	and where $K_n$ is as in \eqref{eq: n KR+zt decomp}, 
	$\om_j$ is as in \eqref{eq: om_j notation}, and $k_{\om_0}(n)$ and $r_{\om_0}(n)$ are as in \eqref{eq: n length orbit in om_j}.
\end{lemma}	
To end this section we note that 
\begin{align}\label{eq: def ep 4}
	\ep \cdot Y_\ep\to 0
\end{align}
as $\ep\to 0$. For the remainder of the document we will assume that $\ep>0$ is always taken sufficiently small such that the results of Section~\ref{sec:bad} apply.

\section{Further Properties of $\Lm_\om$}\label{sec: props of Lm}
In this section we prove some additional properties of the functional $\Lm_\om$ that will be necessary in Section~\ref{sec: fin diam} to obtain cone contraction with finite diameter. In particular, in the main result of this section, which is a version of Lemma~3.11 of \cite{LMD} and dates back to \cite[Lemma 3.2]{liverani_decay_1995-1}, we show that for a function $f\in\sC_{\om,a_*}$ there exists a partition element on which the function $f$ takes values at least as large as $\Lm_\om(f)/4$.

Now we prove the following upper and lower bounds for $\Lm_{\sg^{\Sg_\om^{(k)}}(\om)}\lt(\~\cL_{\om}^{\Sg_\om^{(k)}} f\rt)$.
\begin{lemma}\label{lem: LMD 3.8 temporary}
	For $m$-a.e. $\om\in\Om$ such that $y_*(\om)=0$, and each $k\in\NN$ we have that 
	\begin{align*}
	\Lm_{\sg^{\Sg_\om^{(k)}}(\om)}\lt(\~\cL_{\om}^{\Sg_\om^{(k)}}\ind_\om\rt)\Lm_{\om}(f)
	\leq
	\Lm_{\sg^{\Sg_\om^{(k)}}(\om)}\lt(\~\cL_{\om}^{\Sg_\om^{(k)}} f\rt)
	\leq
	a_*\Lm_{\sg^{\Sg_\om^{(k)}}(\om)}\lt(\~\cL_{\om}^{\Sg_\om^{(k)}}\ind_\om\rt)\Lm_{\om}(f).
	\end{align*}
\end{lemma}
\begin{proof}
	From Lemma~\ref{LMD l3.6} we already see that the first inequality holds.
	Now, fix $\om\in\Om$ (with $y_*(\om)=0$) and $k\in\NN$. 
	To see the other inequality we first let $n, N\in\NN$ and $x\in D_{\sg^N(\om),N+n}$, then
	\begin{align*}
	\frac{\~\cL_{\om}^{N+n}(f)(x)}
	{\~\cL_{\sg^{N}(\om)}^{n}(\ind_{\sg^{N}(\om)})(x)}
	&=
	\frac{\~\cL_{\om}^{N+n}(f)(x)}
	{\~\cL_{\om}^{N+n}(\ind_\om)(x)}
	\cdot\frac{\~\cL_{\om}^{N+n}(\ind_\om)(x)}{\~\cL_{\sg^{N}(\om)}^{n}(\ind_{\sg^{N}(\om)})(x)}
	\\
	&=
	\frac{\~\cL_{\om}^{N+n}(f)(x)}
	{\~\cL_{\om}^{N+n}(\ind_\om)(x)}
	\cdot
	\frac{\~\cL_{\sg^{N}(\om)}^{n}\lt(\cL_{\om}^{N}(\ind_\om)\cdot \ind_{\sg^{N}(\om)}\rt)(x)}
	{\~\cL_{\sg^{N}(\om)}^{n}(\ind_{\sg^{N}(\om)})(x)}
	\\
	&\leq
	\frac{\~\cL_{\om}^{N+n}(f)(x)}
	{\~\cL_{\om}^{N+n}(\ind_\om)(x)}
	\cdot
	\norm{\~\cL_{\om}^{N}\ind_\om}_\infty.
	\end{align*}
	Now taking the infimum over $x\in D_{\sg^N(\om),N+n}$ and letting $n\to\infty$ gives
	\begin{align}
	\Lm_{\sg^{N}(\om)}(\~\cL_{\om}^{N}f) \leq \norm{\~\cL_{\om}^{N}\ind_\om}_\infty \Lm_{\om}(f).
	\label{eq: a* bound 1}
	\end{align}
	Now, set $N=\Sg_\om^{(k)}$.
	Since $\ind_\om\in\sC_{\om,a_*}$, \eqref{eq: cone inv for large bad blocks} from Lemma~\ref{lem: cone cont for coating blocks}  implies that
	\begin{align}
		\norm{\~\cL_{\om}^{\Sg_\om^{(k)}}\ind_\om}_\infty
		&\leq 
		\var(\~\cL_{\om}^{\Sg_\om^{(k)}}(\ind_\om))
		+
		\Lm_{\sg^{\Sg_\om^{(k)}}(\om)}(\~\cL_{\om}^{\Sg_\om^{(k)}}(\ind_\om))
		\nonumber\\
		&\leq 
		\lt(\frac{a_*}{2}+1\rt)\Lm_{\sg^{\Sg_\om^{(k)}}(\om)}(\~\cL_{\om}^{\Sg_\om^{(k)}}(\ind_\om))
		\nonumber\\
		&\leq
		a_*\Lm_{\sg^{\Sg_\om^{(k)}}(\om)}(\~\cL_{\om}^{\Sg_\om^{(k)}}(\ind_\om)),
		\label{eq: a* bound 2}
	\end{align}
	where we have used the fact that $a_*>2$ which follows from \eqref{eq: def of a_*}.
	Combining \eqref{eq: a* bound 2} with \eqref{eq: a* bound 1}, we see that 
	\begin{align*}
		\Lm_{\sg^{\Sg_\om^{(k)}}(\om)}(\~\cL_{\om}^{\Sg_\om^{(k)}} f)
		\leq
		a_*\Lm_{\sg^{\Sg_\om^{(k)}}(\om)}(\~\cL_{\om}^{\Sg_\om^{(k)}}\ind_\om)\Lm_{\om}(f)
	\end{align*}
	completing the proof.
\end{proof}

\begin{lemma}\label{lem: LMD 3.10}
	For each $\dl>0$ and each $\om\in\Om$ there exists $N_{\om,\dl}$ such that for each $n\geq N_{\om,\dl}$, $\cZ_{\om}^{(n)}$ has the property that
	\begin{align*}
		\sup_{Z\in\cZ_{\om}^{(n)}}\Lm_{\om}(\ind_Z)\leq \dl.
	\end{align*}
\end{lemma}
\begin{proof}
	Choose $N_{\om,\dl}\in\NN$ such that
	$$
	\frac{\norm{g_{\om}^{(n)}}_\infty}{\rho_{\om}^n}\leq\dl
	$$
	for each $n\geq N_{\om,\dl}$. Now, fix some $n\geq N_{\om,\dl}$ and let $m\in\NN$.
	Then, for $Z\in\cZ_{\om}^{(n)}$ we have
	\begin{align*}
		\cL_{\om}^n\ind_Z(x)\leq \norm{g_{\om}^{(n)}}_\infty \leq \dl\rho_{\om}^n.
	\end{align*}
	For each $x\in D_{\sg^{n+m}(\om),n+m}\sub D_{\sg^{n+m}(\om),m}$ we have
	\begin{align*}
		\frac
		{\cL_{\om}^{n+m}\ind_Z(x)}
		{\cL_{\om}^{n+m}\ind_\om(x)}
		&\leq
		\frac
		{\norm{\cL_{\om}^n\ind_Z}_\infty\cL_{\sg^n(\om)}^{m}\ind_{\sg^n(\om)}(x)}
		{\cL_{\om}^{n+m}\ind_\om(x)}
		\leq
		\frac
		{\dl\rho_{\om}^n\cL_{\sg^n(\om)}^{m}\ind_{\sg^n(\om)}(x)}
		{\cL_{\sg^n(\om)}^m\lt(\cL_{\om}^{n}\ind_\om\rt)(x)}
		\\
		&
		\leq
		\dl\rho_{\om}^n\frac{1}{
			\inf_{y\in D_{\sg^{n+m}(\om),m}}
			\frac
			{\cL_{\sg^n(\om)}^m\lt(\cL_{\om}^{n}\ind_\om\rt)(y)}
			{\cL_{\sg^n(\om)}^{m}\ind_{\sg^n(\om)}(y)}
		}.
	\end{align*}
	In view of Lemma~\ref{LMD l3.6}, taking the infimum over $x\in D_{\sg^{n+m}(\om),n+m}$ and letting $m\to \infty$ gives
	\begin{align*}
		\Lm_{\om}(\ind_Z)\leq \dl\rho_{\om}^n\cdot\frac{1}{\Lm_{\sg^n(\om)}(\cL_{\om}^n\ind_\om)}
		\leq
		\dl\rho_{\om}^n(\rho_{\om}^n)^{-1}= \dl.
	\end{align*}
\end{proof}

We are now ready to prove the main result of this section, a random version of Lemma~3.11 in \cite{LMD}. Let
\begin{align}\label{def: delta0}
	\dl_0:=\frac{1}{8a_*^3}.
\end{align}
\begin{lemma}\label{lem: bound on partition ele}
	For $m$-a.e. $\om\in\Om$ with $y_*(\om)=0$, for all $\dl<\dl_0$, all $n\geq N_{\om,\dl}$ (where $N_{\om,\dl}$ is as in Lemma~\ref{lem: LMD 3.10}), and all $f\in\sC_{\om,a_*}$ there exists $Z_f\in\cZ_{\om,g}^{(n)}$ such that 
	\begin{align*}
	\inf_{Z_f}f \geq \frac{1}{4}\Lm_{\om}(f).
	\end{align*}
\end{lemma}
\begin{proof}
	We shall the prove the lemma via contradiction. To that end suppose that the conclusion is false, that is we suppose that 
	\begin{align}\label{eq: contr on part ele}
		\inf_Z f< \frac{\Lm_\om(f)}{4}  
	\end{align}
	for all $Z\in\cZ_{\om,g}^{(n)}$. 
	Then, for each $n\geq N_{\om,\dl}$ and each $k\in\NN$ such that $n<\Sg_\om^{(k)}$, using \eqref{eq: contr on part ele} we can write 
	\begin{align}
	\~\cL_{\om}^{\Sg_\om^{(k)}} f
	&=
	\sum_{Z\in\cZ_{\om}^{(n)}} \~\cL_{\om}^{\Sg_\om^{(k)}} (f\ind_Z)
	=
	\sum_{Z\in\cZ_{\om,*}^{(n)}} \~\cL_{\om}^{\Sg_\om^{(k)}} (f\ind_Z)
	\nonumber\\
	&=
	\sum_{Z\in\cZ_{\om,g}^{(n)}} \~\cL_{\om}^{\Sg_\om^{(k)}} (f\ind_Z)
	+
	\sum_{Z\in\cZ_{\om,b}^{(n)}} \~\cL_{\om}^{\Sg_\om^{(k)}} (f\ind_Z)
	\nonumber\\
	&\leq 
	\frac{\Lm_{\om}(f)}{4}
	\sum_{Z\in\cZ_{\om,g}^{(n)}} \~\cL_{\om}^{\Sg_\om^{(k)}} (\ind_Z)
	+
	\sum_{Z\in\cZ_{\om,g}^{(n)}} \~\cL_{\om}^{\Sg_\om^{(k)}} (\ind_Z)\var_Z(f)
	+
	\norm{f}_\infty
	\sum_{Z\in\cZ_{\om,b}^{(n)}} \~\cL_{\om}^{\Sg_\om^{(k)}} (\ind_Z).
	\label{eq: spec part ele 1}
	\end{align}	
	Now for $Z\in\cZ_{\om,b}^{(n)}$, Lemma~\ref{lem: LMD 3.8 temporary} implies that  
	\begin{align}
		\Lm_{\sg^{\Sg_\om^{(k)}}(\om)}\lt(\~\cL_{\om}^{\Sg_\om^{(k)}} (\ind_Z)\rt)
		\leq
		a_*\Lm_{\sg^{\Sg_\om^{(k)}}(\om)}\lt(\~\cL_{\om}^{\Sg_\om^{(k)}}\ind_\om\rt)\Lm_{\om}(\ind_Z)
		=0.
		\label{eq: functional over bad int eq 0}
	\end{align}
	Thus, for $Z\in\cZ_{\om,b}^{(n)}$, using \eqref{eq: functional over bad int eq 0} and \eqref{eq: important coating block ineq}, applied along the blocks $\Sg_\om^{(k)}$, we have that 
	\begin{align}
		\~\cL_{\om}^{\Sg_\om^{(k)}} (\ind_Z)
		&\leq 
		\Lm_{\sg^{\Sg_\om^{(k)}}(\om)}\lt(\~\cL_{\om}^{\Sg_\om^{(k)}} (\ind_Z)\rt)+\var\lt(\~\cL_{\om}^{\Sg_\om^{(k)}} (\ind_Z)\rt)
		\nonumber\\
		&=\var\lt(\~\cL_{\om}^{\Sg_\om^{(k)}} (\ind_Z)\rt)
		\nonumber\\
		&\leq 2\lt(\frac{1}{3}\rt)^{\Sg_\om^{(k)}/R_*}
		+
		a_*\Lm_{\sg^{\Sg_\om^{(k)}}(\om)}\lt(\~\cL_{\om}^{\Sg_\om^{(k)}} (\ind_Z)\rt)
		=2\lt(\frac{1}{3}\rt)^{\Sg_\om^{(k)}/R_*}.		
		\label{eq: spec part ele 2}
	\end{align}
	Note that the right-hand side above goes to zero as $k\to\infty$.
	On the other hand, for $Z\in \cZ_{\om,g}^{(n)}$ we again use \eqref{eq: important coating block ineq} in conjunction with Lemma~\ref{lem: LMD 3.8 temporary} to get that 
	\begin{align}
		\~\cL_{\om}^{\Sg_\om^{(k)}} (\ind_Z)
		&\leq 
		\Lm_{\sg^{\Sg_\om^{(k)}}(\om)}\lt(\~\cL_{\om}^{\Sg_\om^{(k)}} (\ind_Z)\rt)+\var\lt(\~\cL_{\om}^{\Sg_\om^{(k)}} (\ind_Z)\rt)
		\nonumber\\
		&\leq 
		2\lt(\frac{1}{3}\rt)^{\Sg_\om^{(k)}/R_*}
		+
		a_*\Lm_{\sg^{\Sg_\om^{(k)}}(\om)}\lt(\~\cL_{\om}^{\Sg_\om^{(k)}} (\ind_Z)\rt)
		\nonumber\\
		&\leq  
		2\lt(\frac{1}{3}\rt)^{\Sg_\om^{(k)}/R_*}+a_*^2\Lm_{\sg^{\Sg_\om^{(k)}}(\om)}\lt(\~\cL_{\om}^{\Sg_\om^{(k)}}\ind_\om\rt)\Lm_{\om}(\ind_Z).
		\label{eq: spec part ele 3}
	\end{align}
	Substituting \eqref{eq: spec part ele 3} and \eqref{eq: spec part ele 2} into \eqref{eq: spec part ele 1}, applying the functional $\Lm_{\sg^{\Sg_\om^{(k)}}}$ to both sides yields
	\begin{align}
		&\Lm_{\sg^{\Sg_\om^{(k)}}(\om)}\lt(\~\cL_{\om}^{\Sg_\om^{(k)}}f\rt)
		\leq 
		\Lm_{\sg^{\Sg_\om^{(k)}}(\om)}\lt(\~\cL_{\om}^{\Sg_\om^{(k)}}\ind_\om\rt)
		\cdot 
		\frac{\Lm_{\om}(f)}{4}
		\nonumber\\
		&\qquad
		+
		\sum_{Z\in\cZ_{\om,g}^{(n)}} \lt(\lt(a_*\Lm_{\om}(\ind_Z) +2\lt(\frac{1}{3}\rt)^{\Sg_\om^{(k)}/R_*}+a_*^2\Lm_{\om}(\ind_Z)\rt)\var_Z(f)\rt)\Lm_{\sg^{\Sg_\om^{(k)}}(\om)}\lt(\~\cL_{\om}^{\Sg_\om^{(k)}}\ind_\om\rt)
		\nonumber\\
		&\quad\qquad
		+
		\sum_{Z\in\cZ_{\om,b}^{(n)}} 2\lt(\frac{1}{3}\rt)^{\Sg_\om^{(k)}/R_*}\norm{f}_\infty\Lm_{\sg^{\Sg_\om^{(k)}}(\om)}\lt(\~\cL_{\om}^{\Sg_\om^{(k)}}\ind_\om\rt).
		\label{eq: spec part ele 4}
	\end{align}
	Dividing \eqref{eq: spec part ele 4} on both sides by $\Lm_{\sg^{\Sg_\om^{(k)}}(\om)}\lt(\~\cL_{\om}^{\Sg_\om^{(k)}}\ind_\om\rt)$,
	letting $k\to\infty$, and using Lemmas~\ref{lem: LMD 3.8 temporary}, \ref{lem: LMD 3.10}, and \ref{lem: cone cont for coating blocks} gives us 
	\begin{align*}
		\Lm_{\om}(f)
		&\leq
		\frac{\Lm_{\om}(f)}{4}
		+
		\sum_{Z\in\cZ_{\om,g}^{(n)}} (a_*+a_*^2)\Lm_{\om}(\ind_Z)\var_Z(f)
		\nonumber\\
		&\leq
		\frac{\Lm_{\om}(f)}{4}
		+
		(a_*+a_*^2)\var(f)\sup_{Z\in\cZ_{\om,g}^{(n)}}\Lm_{\om}(\ind_Z)
		\nonumber\\
		&\leq
		\lt(\frac{1}{4}
		+
		2a_*^3\dl\rt)\Lm_{\om}(f).
	\end{align*}
	Given our choice \eqref{def: delta0} of $\dl<\dl_0$ we arrive at the contradiction
	\begin{align*}
		\Lm_{\om}(f)\leq \frac{1}{2}\Lm_{\om}(f), 
	\end{align*}
	and thus we are done.
\end{proof}
\section{Finding Finite Diameter Images}\label{sec: fin diam}
We now find a large measure set of fibers $\om\in \Om_F\sub\Om$ for which the image of the cone $\sC_{\om,a_*}$ has a finite diameter image after sufficiently many iterates of the normalized operator $\~\cL_\om$. Towards accomplishing this task we first recall that for each $\om\in\Om$ with $y_*(\om)=0$ and each $k\geq 0$ 
\begin{align*}
	\Sg_\om^{(k)}:=\sum_{j=0}^{k-1} \ell(\om_j)R_*
\end{align*}
where $\om_0:=\om$ and for each $j\geq 1$ we set $\om_j:=\sg^{\Sg_\om^{(j-1)}}(\om)$. For each $\om\in\Om$ with $y_*(\om)=0$, we define the number
\begin{align}\label{eq: def of Sg_om}
	\Sg_\om:=\min\set{\Sg_\om^{(k)}: 
	\inf_{x\in D_{\sg^{\Sg_\om^{(k)}}(\om), \Sg_\om^{(k)}}}\frac{\cL_{\om}^{\Sg_\om^{(k)}}\ind_Z(x)}{\cL_{\om}^{\Sg_\om^{(k)}}\ind_\om(x)}
	\geq
	\frac{\Lm_{\om}(\ind_Z)}{2} \text{ for all } Z\in\cZ_{\om,g}^{(N_{\om,\dl_0})}
	}.
\end{align}
Note that by definition we must have that $\Sg_\om\geq N_{\om,\dl_0}$.
Recall from the proof of Lemma~\ref{lem: constr of Om_G} that the set $\Om_1=\Om_1(B_*)$ is given by 
\begin{align}\label{eq: def of Om_1 part 2}
	\Om_1:=\set{\om\in\Om: C_\ep(\om)\leq B_*},
\end{align}
where $C_\ep(\om)$ comes from Proposition~\ref{prop: LY ineq 2} and $B_*$ was chosen sufficiently large such that $m(\Om_1)\geq 1-\sfrac{\ep}{8}$. 
For $\al_*>0$ and $C_*\geq 1$ we consider the following
\begin{flalign} 
	& \Lm_{\sg^{\Sg_\om}(\om)}\lt(\~\cL_\om^{\Sg_\om}\ind_Z\rt)\geq \al_* \text{ for all } Z\in\cZ_{\om,g}^{(N_{\om,\dl_0})},
	\tag{F1}\label{F1} 
	&\\
	& C_*^{-1}\leq \inf_{D_{\sg^{\Sg_\om}(\om), \Sg_\om}}\~\cL_{\om}^{\Sg_\om}\ind_\om\leq \norm{\~\cL_{\om}^{\Sg_\om}\ind_\om}_\infty\leq  C_*.
	\tag{F2}\label{F2} 
\end{flalign}
Now, we define the set $\Om_3$, depending on parameters $S_*=kR_*$ for some $k\in\NN$, $\al_*>0$, and $C_*\geq 1$, by
\begin{align}\label{eq: def of Om_3}
	\Om_3=\Om_3(S_*, \al_*, C_*):=\set{\om\in\Om: \Sg_\om\leq S_*, \text{ and } \eqref{F1}-\eqref{F2} \text{ hold }}, 
\end{align} 
and choose $S_*=kR_*$, $\al_*>0$, and $C_*\geq 1$ such that $m(\Om_3)\geq 1-\sfrac{\ep}{8}$. Finally, we define
\begin{align}\label{eq: def of Om_F}
	\Om_F:=\Om_1\cap \Om_3, 
\end{align}
which must of course have measure $m(\Om_F)\geq 1-\ep$. Furthermore, in light of the definition of $\Om_G$ from \eqref{eq: def of good set}, we have that 
\begin{align*}
	\sg^{-R_*}(\Om_F)\sub\Om_G.
\end{align*}

\begin{lemma}\label{lem: cone contraction for good om}
	For all $\om\in\Om_F$ such that $y_*(\om)=0$ we have that 
	\begin{align*}
		\~\cL_{\om}^{\Sg_\om}\sC_{\om,a_*}\sub \sC_{\sg^{\Sg_\om}(\om), \sfrac{a_*}{2}} \sub \sC_{\sg^{\Sg_\om}(\om), a_*}
	\end{align*}
	with 
	\begin{align}\label{eq: def of Dl}
		\diam_{\sg^{\Sg_\om}(\om),a_*}\lt(\~\cL_{\om}^{\Sg_\om}\sC_{\om,a_*}\rt)
		\leq
		\Dl
		:=
		2\log \frac{8C_*^2a_*(3+a_*)}{\al_*}<\infty.
	\end{align}
\end{lemma}
\begin{proof}
	The invariance follows from Lemma~\ref{lem: cone cont for coating blocks}. 
	To show that the diameter is finite we first note that for $0\not\equiv f\in\sC_{\om,a_*}$ we must have that $\Lm_\om(f)>0$ by definition. Now, 
	Lemma~\ref{lem: summary of cone dist prop} implies that for $f\in\sC_{\om,a_*}$ we have 
	\begin{align}
		\Ta_{\sg^{\Sg_\om}(\om),a_*}(\~\cL_{\om}^{\Sg_\om}f, \ind_{\sg^{\Sg_\om}(\om)})
		\leq 
		\log \frac{\norm{\~\cL_{\om}^{\Sg_\om}f}_\infty + \frac{1}{2}\Lm_{\sg^{\Sg_\om}(\om)}\lt(\~\cL_{\om}^{\Sg_\om}f\rt)}
		{\min\set{\inf_{D_{\sg^{\Sg_\om}(\om), \Sg_\om}} \~\cL_{\om}^{\Sg_\om}f, \frac{1}{2}\Lm_{\sg^{\Sg_\om}(\om)}\lt(\~\cL_{\om}^{\Sg_\om}f\rt)}}.
		\label{eq: diam est 1}
	\end{align}
	Using Lemmas~\ref{lem: cone cont for coating blocks} and \ref{lem: LMD 3.8 temporary} and \eqref{F2} we bound the numerator by  
	\begin{align}
		\norm{\~\cL_{\om}^{\Sg_\om}f}_\infty + \frac{1}{2}\Lm_{\sg^{\Sg_\om}(\om)}\lt(\~\cL_{\om}^{\Sg_\om}f\rt)
		&\leq 
		\var\lt(\~\cL_{\om}^{\Sg_\om}f\rt)
		+
		\frac{3}{2}\Lm_{\sg^{\Sg_\om}(\om)}\lt(\~\cL_{\om}^{\Sg_\om}f\rt)
		\nonumber\\
		&\leq 
		\frac{3a_*+a_*^2}{2}\Lm_{\sg^{\Sg_\om}(\om)}\lt(\~\cL_{\om}^{\Sg_\om}\ind_\om\rt)\Lm_{\om}(f)\nonumber\\
		&\leq 
		\frac{C_*a_*(3+a_*)}{2}\Lm_{\om}(f).
		\label{eq: diam est 2}
	\end{align}
	To find a lower bound for the denominator we first note that for each $f\in\sC_{\om,a_*}$, by Lemma~\ref{lem: bound on partition ele} there exists $Z_f\in\cZ_{\om,g}^{(N_{\om,\dl_0})}$ such that 
	\begin{align}\label{eq: lower bound for f on Z_f}
		\inf f\rvert_{Z_f}\geq \frac{\Lm_{\om}(f)}{4}.
	\end{align}
	Thus, using \eqref{eq: lower bound for f on Z_f}, for each $x\in D_{\sg^{\Sg_\om}(\om), \Sg_\om}$ we have that 	
	\begin{align*}
		\inf_{x\in D_{\sg^{\Sg_\om}(\om), \Sg_\om}}\~\cL_{\om}^{\Sg_\om} f(x)
		&\geq 
		\inf_{x\in D_{\sg^{\Sg_\om}(\om), \Sg_\om}}\~\cL_{\om}^{\Sg_\om}( f\ind_{Z_f})(x)
		\\
		&\geq
		\inf_{Z_f} f \cdot		
		\inf_{x\in D_{\sg^{\Sg_\om}(\om), \Sg_\om}}\~\cL_{\om}^{\Sg_\om}\ind_{Z_f}(x)
		\\
		&\geq
		\frac{\Lm_{\om}(f)}{4}
		\inf_{x\in D_{\sg^{\Sg_\om}(\om), \Sg_\om}}\~\cL_{\om}^{\Sg_\om}\ind_{Z_f}(x)
		\\
		&\geq 
		\frac{\Lm_{\om}(f)}{4}
		\inf_{y\in D_{\sg^{\Sg_\om}(\om), \Sg_\om}}\frac{\~\cL_{\om}^{\Sg_\om}\ind_{Z_f}(y)}{\~\cL_{\om}^{\Sg_\om}\ind_\om(y)}
		\inf_{z\in D_{\sg^{\Sg_\om}(\om), \Sg_\om}} \~\cL_{\om}^{\Sg_\om}\ind_\om(z).
	\end{align*}
	In light of conditions \eqref{F1}-\eqref{F2} we in fact have that 
	\begin{align}
		\inf_{D_{\sg^{\Sg_\om}(\om), \Sg_\om}}\~\cL_{\om}^{\Sg_\om} f
		\geq 
		\frac{\al_*\Lm_{\om}(f)}{8C_*}>0.
		\label{eq: diam est 3}
	\end{align}
	Combining the estimates \eqref{eq: diam est 2} and \eqref{eq: diam est 3} with \eqref{eq: diam est 1} gives 
	\begin{align*}
		\Ta_{\sg^{\Sg_\om}(\om),a_*}(\~\cL_{\om}^{\Sg_\om}f, \ind_{\sg^{\Sg_\om}(\om)})
		\leq 
		\log\frac{8C_*^2a_*(3+a_*)}{\al_*}<\infty.
	\end{align*}
	Taking the supremum over all functions $f\in\sC_{\om,a_*}$, and applying the triangle inequality finishes the proof. 
\end{proof}

To end this section we recall from Section~\ref{sec:bad} that $0\leq y_*(\om)< R_*$ is chosen to be the smallest integer such that for either choice of sign $+$ or $-$ we have 
\begin{flalign} 
	& \lim_{n\to\infty} \frac{1}{n}\#\set{0\leq k< n: \sg^{\pm kR_*+y_*(\om)}(\om)\in\Om_G} >1-\ep,
	\label{def y*1 part 2} 
	&\\
	& \lim_{n\to\infty} \frac{1}{n}\#\set{0\leq k< n: C_\ep\lt(\sg^{\pm kR_*+y_*(\om)}(\om)\rt)\leq B_*} >1-\ep.
	\label{def y*2 part 2} 
\end{flalign}
In light of the definition of $\Om_F$ \eqref{eq: def of Om_F} and using the same reasoning as in Section~\ref{sec:bad} for the existence of $y_*$ (see Section~7 of \cite{AFGTV21}), for each $\om\in\Om$, we now let $0\leq v_*(\om)<R_*$ be the least integer such that for either choice of sign $+$ or $-$ we have that the following hold:
\begin{flalign}
	& \lim_{n\to\infty} \frac{1}{n}\#\set{0\leq k< n: \sg^{\pm kR_*+v_*(\om)}(\om)\in\Om_G} >1-\ep,
	\label{def v*1} 
	&\\
	&\lim_{n\to\infty} \frac{1}{n}\#\set{0\leq k< n: \sg^{\pm kR_*+v_*(\om)}(\om)\in\Om_F} >1-\ep.
	\label{def v*3}
\end{flalign}
Two significant properties of $v_*$ are the following:
\begin{flalign}
	& v_*(\sg^{v_*(\om)}(\om))=0,
	\label{eq: v* prop1} 
	&\\
	& \text{if }v_*(\om)=0, \text{ then } y_*(\om)=0.
	\label{eq: v* prop2} 
\end{flalign}
\section{Conformal and Invariant Measures}\label{sec: conf and inv meas}
We are now ready to bring together all of the results from Sections~\ref{sec: LY ineq}-\ref{sec: fin diam} to establish the existence of conformal and invariant measures supported in the survivor set $X_{\om,\infty}$. We follow the methods of \cite{AFGTV21} and  \cite{LMD}, and we begin with the following technical lemma from which the rest of our results will follow. 
\begin{lemma}\label{lem: exp conv in C+ cone}
	Let $f,h\in\BV_\Om(I)$, let $\ep>0$ sufficiently small such that the results of Section~\ref{sec:bad} apply, and let $V:\Om\to(0,\infty)$ be a measurable function. 
	Suppose that for each $n\in\NN$, each $|p|\leq n$, each $l\geq 0$, and for $m$-a.e. $\om\in\Om$ we have
	$f_{\sg^p(\om)}\in\sC_{\sg^p(\om), +}$ with $\var(f_{\sg^p(\om)})\leq e^{\ep n}V(\om)$ and  $h_{\sg^{p-l}(\om)}\in\sC_{\sg^{p-l}(\om), +}$ with $\var(h_{\sg^{p-l}(\om)})\leq e^{\ep (n+l)}V(\om)$.
	Then there exists $\vta\in(0,1)$ and a measurable function $N_3:\Om\to\NN$ such that for all $n\geq N_3(\om)$, all $l\geq 0$, and all $|p|\leq n$ we have 
	\begin{align}\label{eq: lem 8.1 ineq}
		\Ta_{\sg^{n+p}(\om), +}\lt(\~\cL_{\sg^p(\om)}^n f_{\sg^p(\om)}, \~\cL_{\sg^{p-l}(\om)}^{n+l} h_{\sg^{p-l}(\om)}\rt) \leq \Dl\vta^n.
	\end{align} 
	Furthermore, $\Dl$, defined in \eqref{eq: def of Dl}, and $\vta$ do not depend on $V$.
\end{lemma}
\begin{proof}
	We begin by noting that by \eqref{eq: L_om is a weak contraction on C_+} for each $l\geq 0$ we have that $\~\cL_{\sg^{p-l}(\om)}^l h_{\sg^{p-l}(\om)}\in\sC_{\sg^{p}(\om),+}$ for each $h_{\sg^{p-l}(\om)}\in\sC_{\sg^{p-l}(\om),+}$, and let 
	\begin{align*}
	h_l=\~\cL_{\sg^{p-l}(\om)}^l h_{\sg^{p-l}(\om)}.
	\end{align*}
	Set $v_*=v_*(\sg^p(\om))$ (defined in Section~\ref{sec:bad}) and let $d_*=d_*(\sg^p(\om))\geq 0$ be the smallest integer that satisfies 
	\begin{flalign} 
	& v_*+d_*R_*\geq \frac{\ep n+\log V(\om)}{\ta-\ep}, 
	\label{eq: d prop1} 
	&\\
	& \sg^{p+v_*+d_*R_*}(\om)\in\Om_F.
	\label{eq: d prop2}
	\end{flalign}
	where $\ta$ was defined in \eqref{eq: def of ta}.
	Choose 
	\begin{align}\label{eq: coose N_1 ge log V}
	N_1(\om)\ge \frac{\log V(\om)}{\ep}
	\end{align} 
	and let $n\geq N_1(\om)$. Now using \eqref{eq: coose N_1 ge log V} to write
	\begin{align*}
	\frac{4\ep n}{\ta}=\frac{\ep n+\ep n}{\ta/2}\geq \frac{\ep n+\log V(\om)}{\ta/2}
	\end{align*}
	and then using \eqref{eq: def of ep_0} and \eqref{eq: def of y_*}, we see that \eqref{eq: d prop1} is satisfied for any $d_*R_*\geq 4\epsilon n/\ta$.
	Using \eqref{def v*3}, the construction of $v_*$, and the ergodic decomposition of $\sigma^{R_*}$ following \eqref{def v*3}, we have for $m$-a.e. $\omega\in\Om$ there is an infinite, increasing sequence of integers $d_j\ge 0$ satisfying \eqref{eq: d prop2}. Furthermore, \eqref{def v*3} implies that 
	\begin{align*}
	\lim_{n\to\infty}\frac{1}{n/R_*}\#\set{0\leq k< \frac{n}{R_*}: \sg^{\pm kR_*+v_*(\om)}(\om)\notin\Om_F} <\ep,
	\end{align*}
	and thus for $n\in\NN$ sufficiently large (depending measurably on $\om$), say $n\geq N_2(\om)\geq N_1(\om)$, we have that 
	\begin{align*}
	\#\set{0\leq k< \frac{n}{R_*}: \sg^{\pm kR_*+v_*(\om)}(\om)\notin\Om_F} < \frac{\ep n}{R_*}.
	\end{align*}
	Thus  the smallest integer $d_*$ satisfying \eqref{eq: d prop1} and \eqref{eq: d prop2} also satisfies 
	\begin{align}\label{eq: dR is Oepn}
	d_*R_*\leq\frac{4\ep n}{\ta}+\ep n=\lt(\frac{4+\ta}{\ta }\rt)\ep n.
	\end{align}
	Let 
	\begin{align}\label{eq: def of hat y_*}
	\hat v_*=v_*+d_*R_*.
	\end{align}
	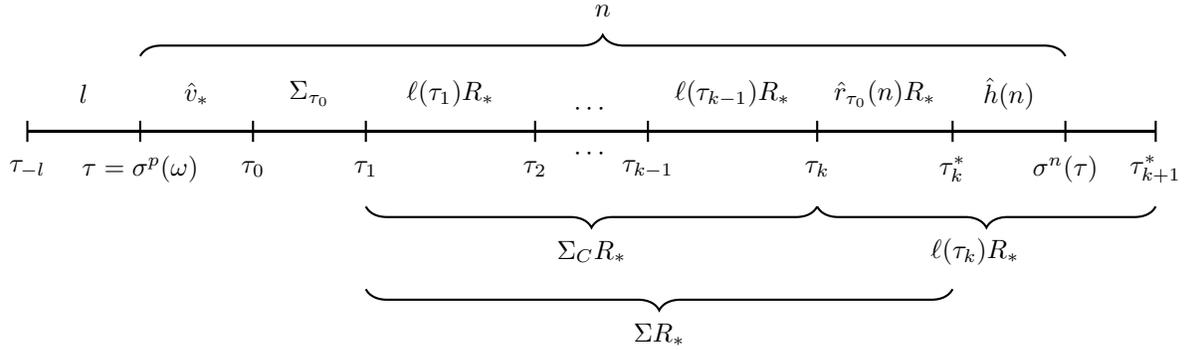
\begin{figure}[h]
		\centering	
		\begin{tikzpicture}[y=1cm, x=1.5cm, thick, font=\footnotesize]    
		\draw[line width=1.2pt,  >=latex'](0,0) -- coordinate (x axis) (10,0);       
		
		\draw (0,-4pt)   -- (0,4pt) {};
		\node at (0, -.5) {$\tau_{-l}$};
		
		\node at (.5, .5) {$l$};
		
		\draw (1,-4pt)   -- (1,4pt) {};
		\node at (1,-.5) {$\tau=\sg^p(\om)$};
		
		\node at (1.5,.5) {$\hat v_*$};
		
		\draw (2,-4pt) -- (2,4pt) {};
		\node at (2,-.5) {$\tau_0$};
		
		\node at (2.5,.5) {$\Sg_{\tau_0}$};
		
		\draw (3,-4pt)  -- (3,4pt)  {};
		\node at (3,-.5) {$\tau_1$};
		
		\node at (3.75,.5) {$\ell(\tau_{1})R_*$};
		
		\draw (4.5,-4pt)  -- (4.5,4pt)  {};
		\node at (4.5,-.5) {$\tau_{2}$};

		\node at (5,.3) {$\cdots$};
		\node at (5,-.3) {$\cdots$};

		\draw (5.5,-4pt)  -- (5.5,4pt)  {};
		\node at (5.5,-.5) {$\tau_{k-1}$};

		\node at (6.25,.5) {$\ell(\tau_{k-1})R_*$};	
		
		\draw (7,-4pt)  -- (7,4pt)  {};
		\node at (7,-.5) {$\tau_{k}$};
		
		\node at (7.6,.5) {$\hat r_{\tau_0}(n)R_*$};
		
		\draw (8.2,-4pt)  -- (8.2,4pt)  {};
		\node at (8.2,-.5) {$\tau_{k}^*$};
		\node at (8.7,.5) {$\hat h(n)$};
		
		\draw (9.2,-4pt)  -- (9.2,4pt)  {};
		\node at (9.2,-.5) {$\sg^n(\tau)$};
		
		\draw (10,-4pt)  -- (10,4pt)  {};
		\node at (10,-.5) {$\tau_{k+1}^*$};

		
		\draw[decorate,decoration={brace,amplitude=8pt,mirror}] 
		(3,-2.1)  -- (8.2,-2.1) ; 
		\node at (5.6,-2.7){$\Sg R_*$};		
		
		\draw[decorate,decoration={brace,amplitude=8pt,mirror}] 
		(3,-1) -- (7,-1) ; 
		\node at (5,-1.6){$\Sg_C R_*$};

		\draw[decorate,decoration={brace,amplitude=8pt,mirror}] 
		(7,-1)  -- (10,-1) ; 
		\node at (8.4,-1.6){$\ell(\tau_{k})R_*$};
		
		\draw[decorate,decoration={brace,amplitude=8pt}] 
		(1,1)  -- (9.2,1) ; 
		\node at (5.1,1.6){$n$};
		
		\end{tikzpicture}
		\caption{The fibers $\tau_j$ and the decomposition of $n=\hat v_*+\Sg_{\tau_0}+\Sg R_*+\hat h(n)$.}
		\label{figure: om_k}
	\end{figure}
	Now, we wish to examine the iteration of our operator cocycle along a collection $\Sg R_*$ of blocks, each of length $\ell(\om)R_*$, so that the images of $\~\cL_\om^{\ell(\om)R_*}$ are contained in $\sC_{\sg^{\ell(\om)R_*}(\om),\sfrac{a_*}{2}}$ as in Lemma~\ref{lem: cone cont for coating blocks}; see Figure~\ref{figure: om_k}.
	
	We begin by establishing some simplifying notation. To that end, set $\tau=\sg^{p}(\om)$, $\tau_{-l}=\sg^{p-l}(\om)$, and $\tau_0=\sg^{p+\hat v_*}(\om)$; see Figure~\ref{figure: om_k}. Note that in light of \eqref{eq: v* prop1}, \eqref{eq: v* prop2}, and \eqref{eq: def of hat y_*} we have that 
	\begin{align}\label{eq: y_* tau_0}
		v_*(\tau_0)=y_*(\tau_0)=0.
	\end{align}
	Now, by our choice of $d_*$, we have that if $f_\tau\in\sC_{\tau,+}$ with $\var(f_\tau)\leq e^{\ep n}V(\om)$, then 
	\begin{align}\label{eq: f in a cone}
	\~\cL_\tau^{\hat v_*} f_\tau\in\sC_{\tau_0,a_*}.
	\end{align}
	Indeed, applying Proposition~\ref{prop: LY ineq 2}, \eqref{eq: d prop1}, \eqref{eq: d prop2}, and the definition of $\Om_F$ \eqref{eq: def of Om_F}, we have 
	\begin{align*}
	\var\lt(\~\cL_{\tau}^{\hat v_*}f_\tau\rt)
	&\leq 
	C_{\ep}(\sg^{\hat v_*}(\tau))e^{-(\ta-\ep)\hat v_*}\var(f_\tau)
	+C_{\ep}(\sg^{\hat v_*}(\tau))\Lm_{\sg^{\hat v_*}(\tau)}\lt(\~\cL_{\tau}^{\hat v_*}f_\tau\rt)
	\\
	&\leq 
	B_*e^{-(\ta-\ep)\hat v_*}\var(f_\tau)+B_*\Lm_{\sg^{\hat v_*}(\tau)}\lt(\~\cL_{\tau}^{\hat v_*}f_\tau\rt)
	\\&
	\leq 
	B_*\frac{\var(f_\tau)}{e^{\ep n}V(\om)}+B_*\Lm_{\sg^{\hat v_*}(\tau)}\lt(\~\cL_{\tau}^{\hat v_*}f_\tau\rt) 
	\\
	&\leq
	B_*+B_*\Lm_{\sg^{\hat v_*}(\tau)}\lt(\~\cL_{\tau}^{\hat v_*}f_\tau\rt)  
	\\
	&
	\leq 2B_*\Lm_{\sg^{\hat v_*}(\tau)}\lt(\~\cL_{\tau}^{\hat v_*}f_\tau\rt)\leq \frac{a_*}{6}\Lm_{\sg^{\hat v_*}(\tau)}\lt(\~\cL_{\tau}^{\hat v_*}f_\tau\rt), 
	\end{align*}
	where we recall that $a_*> 12B_*$ is defined in \eqref{eq: def of a_*}. 
	A similar calculation yields that if $h_{\tau_{-l}}\in\sC_{\tau_{-l},+}$ with $\var(h_{\tau_{-l}})\leq e^{\ep (n+l)}V(\om)$, then $\~\cL_{\tau_{-l}}^{l+\hat v_*} h_{\tau_{-l}}\in\sC_{\tau_0,a_*}$.

	We now set $\tau_1=\sg^{\Sg_{\tau_0}}(\tau_0)$ and for each $j\geq 2$ let $\tau_j=\sg^{\ell(\tau_{j-1})R_*}(\tau_{j-1})$.
	Note that since $\tau_0\in\Om_F$, we have that $\Sg_{\tau_0}\leq S_*$. 
	
	As there are only finitely many blocks (good and bad) that will occur within an orbit of length $n$, let $k\geq 1$ be the integer such that 
	\begin{align*}
	\hat v_*+\Sg_{\tau_0}+\sum_{j=1}^{k-1} \ell(\tau_j)R_* \leq n < \hat v_*+\Sg_{\tau_0}+\sum_{j=1}^{k} \ell(\tau_j)R_*,
	\end{align*}
	and let 
	\begin{equation*}
	\Sg_C:=\sum_{j=1}^{k-1} \ell(\tau_j)
	\qquad \text{ and } \qquad
	\hat r_{\tau_0}(n):=r_{\tau_0}(n-\hat v_*)
	\end{equation*} 
	where $r_{\tau_0}(n-\hat v_*)$ is the number defined in \eqref{eq: n length orbit in om_j}.
	Finally setting 
	$$
	\Sg=\Sg_C+\hat r_{\tau_0}(n),
	\qquad 
	\hat h(n):=n-\hat v_*-\Sg_{\tau_0}-\Sg R_*,
	\qquad \text{ and } \qquad
	\tau_k^*:=\sg^{\hat r_{\tau_0}(n)}(\tau_k),
	$$
	we have the right decomposition of our orbit length $n$ into blocks which do not expand distances in the fiber cones $\sC_{\om,a_*}$ and $\sC_{\om,+}$.
	Now let
	\begin{align}\label{eq: def of N_3}
	n\geq N_3(\om):=\max\set{N_2(\om), \frac{R_*}{\ep}, \frac{S_*}{\ep}}.
	\end{align}
	Since $v_*, \hat h(n)\leq R_*$, by \eqref{eq: def of N_3}, \eqref{eq: dR is Oepn}, and for 
	\begin{align}\label{key}
		\ep<\frac{\ta}{8(1+\ta)}
	\end{align}
	sufficiently small, we must have that
	\begin{align}\label{eq: SgR is big O ep n}
	\Sg R_*
	&=
	n-\hat v_*-\Sg_{\tau_0}-\hat h(n)
	=
	n-v_*-d_*R_*-\Sg_{\tau_0}-\hat h(n)
	\nonumber\\
	&\geq 
	n-\lt(\frac{4+\ta}{\ta}\rt)\ep n-2R_*-S_*
	\geq 
	n-\lt(\frac{4+\ta}{\ta}\rt)\ep n-3\ep n
	\nonumber\\
	&\geq 
	n\lt(1-4\ep\lt(\frac{1+\ta}{\ta}\rt)\rt) > \frac{n}{2}.
	\end{align}
	Now we note that since $\~\cL_{\tau_k^*}^{\hat h(n)}(\sC_{\tau_k^*,+})\sub\sC_{\sg^{n+p}(\om), +}$ we have that $\~\cL_{\tau_k^*}^{\hat h(n)}$ is a weak contraction, and hence, we have 
	\begin{align}\label{eq: L^r a weak contraction on C_+}
	\Ta_{\sg^{n+p}(\om), +}\lt(\~\cL_{\tau_k^*}^{\hat h(n)} f' , \~\cL_{\tau_k^*}^{\hat h(n)} h' \rt)\leq \Ta_{\tau_k^*,+}(f',h'), 
	\qquad f',h'\in\Ta_{\tau_k^*,+}.	
	\end{align}
	Recall that $E_{\tau_1}(n-\hat v_*-\Sg_{\tau_0})$, defined in Lemma~~\ref{lem: proportion of bad blocks}, is the total length of the bad blocks of the $n-\hat v_*$ length orbit starting at $\tau_0$, i.e. 
	\begin{align*}
	E_{\tau_1}(n-\hat v_*-\Sg_{\tau_0})&=\sum_{\substack{1\leq j< k \\ \tau_j\in\Om_B}} \ell(\tau_j) +r_{\tau_0}(n-\hat v_*).
	\end{align*}
	Lemma~\ref{lem: proportion of bad blocks} then gives that 
	\begin{align}\label{eq: est of S}
	E_{\tau_1}(n-\hat v_*-\Sg_{\tau_0})<Y\ep\Sg. 
	\end{align}
	We are now poised to calculate \eqref{eq: lem 8.1 ineq}, but first we note that we can write
	\begin{align}
	n&=\hat v_* + \Sg_{\tau_0} + \Sg R_* +\hat h(n)
	\nonumber\\
	&=\hat v_* + \Sg_{\tau_0} +\Sg_C R_*+ \hat r_{\tau_0}(n)+ \hat h(n)
	\label{eq: lem 8.1 decomp of n final ineq}	
	\end{align}
	and that the number of good blocks contained in the orbit of length $n-\hat v_*-\Sg_{\tau_0}$ is given by 
	\begin{align}\label{eq: number of good blocks}
	\Sg_G:=\#\set{1\leq j\leq k: \tau_j\in\Om_G}
	=\Sg-E_{\tau_1}(n-\hat v_*-\Sg_{\tau_0})
	\leq \Sg_C.
	\end{align} 
	Now, using \eqref{eq: lem 8.1 decomp of n final ineq} we combine (in order) \eqref{eq: L^r a weak contraction on C_+}, \eqref{eq: Ta+ leq Ta}, and Theorem~\ref{thm: cone distance contraction} (repeatedly) in conjunction with the fact that $\tau_0\in\Om_F$ to see that
	\begin{align}
	&\Ta_{\sg^{n+p}(\om), +}
	\lt(\~\cL_{\tau}^{n} (f_\tau), \~\cL_{\tau_{-l}}^{n+l} (h_{\tau_{-l}}) \rt)
	\nonumber\\
	&\quad
	=
	\Ta_{\sg^{n+p}(\om), +}
	\lt(
	\~\cL_{\tau_k}^{\hat h(n)} \circ \~\cL_{\tau_1}^{\Sg R_*} \circ \~\cL_{\tau_0}^{\Sg_{\tau_0}} \circ \~\cL_{\tau}^{\hat v_*} (f_\tau)
	,
	\~\cL_{\tau_k}^{\hat h(n)} \circ \~\cL_{\tau_1}^{\Sg R_*} \circ \~\cL_{\tau_0}^{\Sg_{\tau_0}} \circ \~\cL_{\tau}^{\hat v_*}\circ \~\cL_{\tau_{-l}}^{l} (h_{\tau_{-l}}) 
	\rt)
	\nonumber\\
	&\quad
	\leq
	\Ta_{\tau_k^*, +}
	\lt(
	\~\cL_{\tau_1}^{\Sg R_*} \circ \~\cL_{\tau_0}^{\Sg_{\tau_0}} \circ \~\cL_{\tau}^{\hat v_*} (f_\tau)
	,
	\~\cL_{\tau_1}^{\Sg R_*} \circ \~\cL_{\tau_0}^{\Sg_{\tau_0}} \circ \~\cL_{\tau}^{\hat v_*} (h_l) 
	\rt)	
	\nonumber\\
	&\quad
	\leq
	\Ta_{\tau_k^*, a_*}
	\lt(
	\~\cL_{\tau_k}^{\hat r_{\tau_0}(n)} \circ \~\cL_{\tau_1}^{\Sg_C R_*} \circ \~\cL_{\tau_0}^{\Sg_{\tau_0}} \circ \~\cL_{\tau}^{\hat v_*} (f_\tau)
	,
	\~\cL_{\tau_k}^{\hat r_{\tau_0}(n)} \circ \~\cL_{\tau_1}^{\Sg_C R_*} \circ \~\cL_{\tau_0}^{\Sg_{\tau_0}} \circ \~\cL_{\tau}^{\hat v_*} (h_l) 
	\rt)
	\nonumber\\
	&\quad
	\leq 
	\Ta_{\tau_k, a_*}
	\lt(
	\~\cL_{\tau_1}^{\Sg_C R_*} \circ \~\cL_{\tau_0}^{\Sg_{\tau_0}} \circ \~\cL_{\tau}^{\hat v_*} (f_\tau)
	,
	\~\cL_{\tau_1}^{\Sg_C R_*} \circ \~\cL_{\tau_0}^{\Sg_{\tau_0}} \circ \~\cL_{\tau}^{\hat v_*} (h_l) 
	\rt)
	\nonumber\\	
	&\quad
	\leq
	\lt(\tanh\lt(\frac{\Dl}{4}\rt)\rt)^{\Sg_G}
	\Ta_{\tau_1, a_*}
	\lt(
	\~\cL_{\tau_0}^{\Sg_{\tau_0}} \circ \~\cL_{\tau}^{\hat v_*} (f_\tau)
	,
	\~\cL_{\tau_0}^{\Sg_{\tau_0}} \circ \~\cL_{\tau}^{\hat v_*} (h_l) 
	\rt).
	\label{eq: cone Cauchy 1}
	\end{align} 
	Now since $\tau_0\in \Om_F$ and in light of \eqref{eq: f in a cone}, applying Lemma~\ref{lem: cone contraction for good om} allows us to estimate the $\Ta_{\tau_1,a_*}$ term in the right hand side of \eqref{eq: cone Cauchy 1} to give 
	\begin{align}\label{eq: cone Cauchy Ta leq Dl}
	\Ta_{\sg^{n+p}(\om), +}
	\lt(\~\cL_{\tau}^{n} (f_\tau), \~\cL_{\tau_{-l}}^{n+l} (h_{\tau_{-l}}) \rt)
	\leq \lt(\tanh\lt(\frac{\Dl}{4}\rt)\rt)^{\Sg_G}\Dl.
	\end{align}
	Using \eqref{eq: number of good blocks}, the fact that $E_{\tau_1}(n-\hat v_*-\Sg_{\tau_0})\geq 1$, \eqref{eq: est of S}, and \eqref{eq: SgR is big O ep n}, we see that
	\begin{align}
	\Sg_G
	&= 
	\Sg-E_{\tau_1}(n-\hat v_*-\Sg_{\tau_0})
	\nonumber\\
	&\geq
	\Sg-Y\ep\Sg
	\nonumber\\
	&=
	\Sg\lt(1-Y\ep\rt)
	\nonumber\\
	&\geq
	\frac{\lt(1-Y\ep\rt)n}{2R_*}.
	\label{eq: cone Cauchy exponent}
	\end{align}
	In light of \eqref{eq: def ep 4}, for  all $\ep>0$ sufficiently small we have that $1-Y\ep>0$. 
	Finally, inserting \eqref{eq: cone Cauchy Ta leq Dl} and \eqref{eq: cone Cauchy exponent} into \eqref{eq: cone Cauchy 1} gives
	\begin{align*}
	&\Ta_{\sg^{n+p}(\om), +}
	\lt(\~\cL_{\tau}^{n} (f_\tau), \~\cL_{\tau_{-l}}^{n+l} (h_{\tau_{-l}}) \rt)
	\leq 
	\Dl \vta^n,
	\end{align*}
	where 
	$$
	\vta:=\lt(\tanh\lt(\frac{\Dl}{4}\rt)\rt)^{\frac{\lt(1-Y\ep\rt)}{2R_*}}<1,
	$$
	which completes the proof.
\end{proof}
Combining Lemma~\ref{lem: exp conv in C+ cone} together with Lemma~\ref{lem: birkhoff cone contraction} gives the following immediate corollary.
\begin{corollary}\label{cor: exp conv in sup norm}
		Suppose $\ep>0$, $V:\Om\to(0,\infty)$, $f_{\sg^p(\om)}\in\sC_{\sg^p(\om),+}$, and $h_{\sg^{p-l}(\om)}\in\sC_{\sg^{p-l}(\om),+}$ all satisfy the hypotheses of Lemma~\ref{lem: exp conv in C+ cone}. Then there exists $\kp\in(0,1)$ such that for $m$-a.e. $\om\in\Om$, all $n\geq N_3(\om)$, all $l\geq 0$, and all $|p|\leq n$ we have 
		\begin{align*}
			\norm{\~\cL_{\sg^p(\om)}^n f_{\sg^p(\om)} -\~\cL_{\sg^{p-l}(\om)}^{n+l} h_{\sg^{p-l}(\om)}}_\infty
			&\leq 
			\norm{\~\cL_{\sg^p(\om)}^n f_{\sg^p(\om)}}_\infty \lt(e^{\Dl\vta^n}-1\rt).
		\end{align*} 
\end{corollary}

Notice that if we wish to apply Lemma~\ref{lem: exp conv in C+ cone} (or Corollary~\ref{cor: exp conv in sup norm}) repeatedly iterating in the forward direction, i.e. taking $p=0$ so that we push forward starting from the $\om$ fiber, then we only need that $f\in\sC_{\om,+}$ and do not need to be concerned with the assumption on the variation. Indeed, as $p=0$ is fixed, then we will have $\var(f)\leq \var(f)\cdot e^{\ep n}$ for any $n\geq 1$. However, if we wish to apply Lemma~\ref{lem: exp conv in C+ cone} repeatedly with $p=-n$ for $n$ increasing to $\infty$, then we will need to consider special functions $f$. 
\begin{definition}\label{def: set D of tempered BV func}
	We let the set $\cD$ denote the set of functions $f\in\BV_\Om$ such that for each $\ep>0$ there exists a measurable function $V_{f,\ep}:\Om\to(0,\infty)$ such that the following hold for all $n\in\ZZ$ with $|n|$ sufficiently large:
	\begin{flalign}
		& \var(f_{\sg^n(\om)})\leq V_{f,\ep}(\om)e^{\ep|n|},
		\label{cond cD1}\tag{\cD1}
		&\\
		&\Lm_{\sg^n(\om)}(|f_{\sg^n(\om)}|)\geq V_{f,\ep}^{-1}(\om)e^{-\ep |n|}.
		\label{cond cD2}\tag{\cD2}
	\end{flalign}
	Let $\cD^+\sub\cD$ denote the collection of all functions $f\in\cD$ such that $f_\om\geq 0$ for each $\om\in\Om$.
\end{definition}
\begin{remark}
	Note that the space $\cD$ is nonempty. In particular, $\cD$ contains any function $f:\Om\times I\to\RR$ such that $f_\om$ is equal to some fixed function $f\in\BV(I)$ with $0<\inf|f|$. More generally, $\cD$ contains any functions $f:\Om\times I\to \RR$ such that $\log\var(f_\om)$, $\log\Lm(|f_\om|)\in L^1(m)$. 
\end{remark}
\begin{remark}\label{rem: combining D1 and D2}
	Note that if $f\in \cD$ then taking $V_f(\om)=V_{f,\ep}^2(\om)$ measurable and $\ep'=\ep/2$ we have that
	\begin{align*}
		\frac{\var(f_{\sg^{-n}(\om)})}{\Lm_{\sg^{-n}(\om)}(f_{\sg^{-n}(\om)})}
		\leq
		V_{f}(\om)\frac{\var(f_\om)}{\Lm_{\om}(f_\om)}e^{\ep' n}.
	\end{align*}
\end{remark}
In the following corollary we establish the existence of an invariant density.
\begin{corollary}\label{cor: exist of unique dens q}
	There exists a function $q\in\BV_\Om$ and a measurable function $\lm:\Om\to\RR^+$ such that for $m$-a.e. $\om\in\Om$
	\begin{align}\label{eq: q invariant density}
		\cL_\om q_\om = \lm_\om q_{\sg(\om)}
		\qquad\text{ and }\qquad 
		\Lm_\om(q_\om)=1.
	\end{align}
	 Furthermore, we have that $\log\lm_\om\in L^1(m)$ and for $m$-a.e. $\om\in\Om$, $\lm_\om\geq \rho_\om$. 
\end{corollary}
\begin{proof}
	First we note that for any $f\in\cD_+$, Lemma~\ref{LMD l3.6} and Remark~\ref{rem: combining D1 and D2} give that 
	\begin{align*}
		\var\lt(f_{\om,n}\rt)
		&=
		\frac{\rho_{\sg^{-n}(\om)}^n}{\Lm_{\om}\lt(\cL_{\sg^{-n}(\om)}^n f_{\sg^{-n}(\om)}\rt)}\var\lt(f_{\sg^{-n}(\om)}
			\rt)
		\leq 
		\frac{\var(f_{\sg^{-n}(\om)})}{\Lm_{\sg^{-n}(\om)}(f_{\sg^{-n}(\om)})}
		\leq
		V_{f}(\om)\frac{\var(f_\om)}{\Lm_{\om}(f_\om)}e^{\ep n}
	\end{align*}	
	for all $n\in\NN$ sufficiently large, say for $n\geq N_4(\om)$, and some measurable $V_f:\Om\to(0,\infty)$, where
	\begin{align*}
		f_{\om,n}:=\frac{f_{\sg^{-n}(\om)}\rho_{\sg^{-n}(\om)}^n}{\Lm_{\om}\lt(\cL_{\sg^{-n}(\om)}^n f_{\sg^{-n}(\om)}\rt)}\in\sC_{\sg^{-n}(\om),+}.
	\end{align*}
	Thus, Corollary~\ref{cor: exp conv in sup norm} (with $p=-n$ and $V(\om)=V_f(\om)\var(f_\om)/\Lm_\om(f_\om)$) gives that 
	$$
		(\~\cL_{\sg^{-n}(\om)}^n f_{\om,n})_{n\in\NN}
		=
		\lt(\frac{\cL_{\sg^{-n}(\om)}^n f_{\sg^{-n}(\om)}}{\Lm_{\om}\lt(\cL_{\sg^{-n}(\om)}^n f\rt)}\rt)_{n\in\NN}
	$$ 
	forms a Cauchy sequence in $\sC_{\om,+}$, and therefore there must exist some $q_{\om,f}\in\sC_{\om,+}$ with 
	\begin{align}\label{eq: constr of q}
		q_{\om,f}:=\lim_{n\to\infty}\frac{\cL_{\sg^{-n}(\om)}^n f_{\sg^{-n}(\om)}}{\Lm_\om\lt(\cL_{\sg^{-n}(\om)}^n f_{\sg^{-n}(\om)}\rt)}.
	\end{align}
	By construction we have that $\Lm_\om(q_{\om,f})=1$.
	Now, in view of calculating $\cL_\om q_{\om,f}$, we note that \eqref{eq: a* bound 1} (with $N=1$ and $f=\cL_{\sg^{-n}(\om)}^{n} f_{\sg^{-n}(\om)}$) gives that 
	\begin{align}\label{eq: 11.21 up bd}
		\frac{\Lm_{\sg(\om)}\lt(\cL_{\sg^{-n}(\om)}^{n+1} f_{\sg^{-n}(\om)}\rt)}{\Lm_{\om}\lt(\cL_{\sg^{-n}(\om)}^{n} f_{\sg^{-n}(\om)}\rt)}
		\leq 
		\frac{\norm{\cL_\om\ind_\om}_\infty\Lm_{\om}\lt(\cL_{\sg^{-n}(\om)}^{n} f_{\sg^{-n}(\om)}\rt)}{\Lm_{\om}\lt(\cL_{\sg^{-n}(\om)}^{n} f_{\sg^{-n}(\om)}\rt)}
		=
		\norm{\cL_\om\ind_\om}_\infty. 
	\end{align}
	Lemma~\ref{LMD l3.6} (with $k=1$ and $f=\cL_{\sg^{-n}(\om)}^{n} f_{\sg^{-n}(\om)}$) implies that 
	\begin{align*}
		\frac{\Lm_{\sg(\om)}\lt(\cL_{\sg^{-n}(\om)}^{n+1} f_{\sg^{-n}(\om)}\rt)}{\Lm_{\om}\lt(\cL_{\sg^{-n}(\om)}^{n} f_{\sg^{-n}(\om)}\rt)}
		\geq
		\frac{\rho_\om\Lm_{\om}\lt(\cL_{\sg^{-n}(\om)}^{n} f_{\sg^{-n}(\om)}\rt)}{\Lm_{\om}\lt(\cL_{\sg^{-n}(\om)}^{n} f_{\sg^{-n}(\om)}\rt)}
		=
		\rho_\om,
	\end{align*}
	and thus, together with \eqref{eq: 11.21 up bd}, we have  
	\begin{align}\label{eq: lm geq rho}
		\frac{\Lm_{\sg(\om)}\lt(\cL_{\sg^{-n}(\om)}^{n+1} f_{\sg^{-n}(\om)}\rt)}{\Lm_{\om}\lt(\cL_{\sg^{-n}(\om)}^{n} f_{\sg^{-n}(\om)}\rt)}
		\in [\rho_\om,\norm{\cL_\om\ind_\om}_\infty].
	\end{align}
	Thus there must exist a sequence $(n_k)_{k\in\NN}$ along which this ratio converges to some value $\lm_{\om,f}$, that is
	\begin{align*}
		\lm_{\om,f}:=\lim_{k\to\infty} \frac{\Lm_{\sg(\om)}\lt(\cL_{\sg^{-n_k}(\om)}^{n_k+1} f\rt)}{\Lm_{\om}\lt(\cL_{\sg^{-n_k}(\om)}^{n_k} f\rt)}.
	\end{align*}
	Hence we have 
	\begin{align}\label{eq: Lq=lm q depends on f*}
		\cL_\om q_{\om,f}
		&=
		\lim_{k\to\infty} 
		\frac{\cL_{\sg^{-n_k}(\om)}^{n_k+1} f}{\Lm_\om\lt(\cL_{\sg^{-n_k}(\om)}^{n_k} f\rt)}
		=
		\lim_{k\to\infty} 
		\frac{\cL_{\sg^{-n_k}(\om)}^{n_k+1} f}{\Lm_{\sg(\om)}\lt(\cL_{\sg^{-n_k}(\om)}^{n_k+1} f\rt)}
		\cdot
		\frac{\Lm_{\sg(\om)}\lt(\cL_{\sg^{-n_k}(\om)}^{n_k+1} f\rt)}{\Lm_{\om}\lt(\cL_{\sg^{-n_k}(\om)}^{n_k} f\rt)}
		=
		\lm_{\om,f}q_{\sg(\om),f}.		
	\end{align}
	From \eqref{eq: Lq=lm q depends on f*} it follows that $\lm_{\om,f}$ does not depend on the sequence $(n_k)_{k\in\NN}$, and in fact we have
	\begin{align*}
		\lm_{\om,f}=\lim_{n\to\infty} \frac{\Lm_{\sg(\om)}\lt(\cL_{\sg^{-n}(\om)}^{n+1} f\rt)}{\Lm_{\om}\lt(\cL_{\sg^{-n}(\om)}^{n} f\rt)},
	\end{align*}
	and thus,
	\begin{align}\label{eq: Lq=lm q depends on f}
		\cL_\om q_{\om,f}=\lm_{\om,f}q_{\sg(\om),f}.
	\end{align}
	To see that $q_{\om,f}$ and $\lm_{\om,f}$ do not depend on $f$, we apply Lemma~\ref{lem: exp conv in C+ cone} (with $p=-n$,  $l=0$, and $V(\om)=\max\set{V_f(\om)\var(f_\om)/\Lm_\om(f_\om), V_h(\om) \var(h_\om)/\Lm_\om(h_\om)}$) to functions $f,h\in\cD_+$ to get that 
	\begin{align}
		\Ta_{\om,+}\lt(q_{\om,f}, q_{\om,h}\rt)
		\leq 
		\Ta_{\om,+}\lt(q_{\om,f}, f_{\om,n}\rt)
		+
		\Ta_{\om,+}\lt(f_{\om,n}, h_{\om,n}\rt)
		+
		\Ta_{\om,+}\lt(q_{\om,h}, h_{\om,n}\rt)
		\leq 
		3\Dl\vta^n
		\label{eq: Ta estimate for q_om,f}
	\end{align}
	for each $n\geq N_3(\om)$. Thus, inserting \eqref{eq: Ta estimate for q_om,f} into Lemma~\ref{lem: birkhoff cone contraction} yields
	\begin{align*}
		\norm{q_{\om,f}-q_{\om,h}}_\infty
		&\leq 
		\norm{q_{\om,f}}_\infty \lt(e^{\lt(\Ta_{\om,+}(q_{\om,f}, q_{\om,h})\rt)}-1\rt)
		\leq \norm{q_{\om,f}}_\infty \lt(e^{3\Dl\vta^n}-1\rt),
	\end{align*}
	which converges to zero exponentially fast as $n$ tends towards infinity. Thus we must in fact have that $q_{\om,f}=q_{\om,h}$ for all $f,h$. Moreover, in light of \eqref{eq: Lq=lm q depends on f}, this implies that $\lm_{\om,f}=\lm_{\om,h}$. We denote the common values by $q_\om$ and $\lm_\om$ respectively. It follows from \eqref{eq: rho log int} and \eqref{eq: lm geq rho} that 
	\begin{align}\label{eq: rho leq lm leq a rho}
		0<\rho_\om\leq\lm_\om\leq \norm{\cL_\om\ind_\om}_\infty.
	\end{align} 
	Measurability of the map $\om\mapsto\lm_\om$ 
	follows from the measurability of the sequence 
	$$\lt(
	\frac{\Lm_{\sg(\om)}\lt(\cL_{\sg^{-n}(\om)}^{n+1}\ind_{\sg^{-n}(\om)}\rt)}
	{\Lm_{\om}\lt(\cL_{\sg^{-n}(\om)}^{n}\ind_{\sg^{-n}(\om)}\rt)}
	\rt)_{n\in\NN}.
	$$ 
	The $\log$-integrability of $\lm_\om$ follows from the $\log$-integrability of $\rho_\om$ and \eqref{eq: rho leq lm leq a rho}.
	Finally, measurability of the maps $\om\mapsto\inf q_\om$ and $\om\mapsto\norm{q_\om}_\infty$ follows from the fact that we have 
	\begin{align*}
		q_\om=\lim_{n\to\infty}\frac{\cL_{\sg^{-n}(\om)}^n\ind_{\sg^{-n}(\om)}}{\Lm_\om(\cL_{\sg^{-n}(\om)}^n\ind_{\sg^{-n}(\om)})},
	\end{align*}
	which is a limit of measurable functions, and thus finishes the proof. 

\end{proof}
\begin{remark}\label{rem: def of lm^k}
	For each $k\in\NN$, inducting on \eqref{eq: Lq=lm q depends on f} for any $f\in\cD_+$ yields 
	\begin{align}
		\cL_\om^k(q_\om)
		&=
		\lim_{n\to\infty}
		\frac{\cL_{\sg^{-n}(\om)}^{n+k} f_{\sg^{-n}(\om)}}
		{\Lm_\om(\cL_{\sg^{-n}(\om)}^n f_{\sg^{-n}(\om)})}
		\nonumber\\
		&=
		\lim_{n\to\infty}
		\frac{\cL_{\sg^{-n}(\om)}^{n+k}f_{\sg^{-n}(\om)}}
		{\Lm_{\sg^k(\om)}(\cL_{\sg^{-n}(\om)}^{n+k} f_{\sg^{-n}(\om)})}
		\cdot 
		\frac{\Lm_{\sg^k(\om)}(\cL_{\sg^{-n}(\om)}^{n+k} f_{\sg^{-n}(\om)})}
		{\Lm_\om(\cL_{\sg^{-n}(\om)}^n f_{\sg^{-n}(\om)})}
		\nonumber\\
		&=
		q_{\sg^k(\om)}\cdot 
		\lim_{n\to\infty} 
		\frac{\Lm_{\sg^k(\om)}(\cL_{\sg^{-n}(\om)}^{n+k} f_{\sg^{-n}(\om)})}
		{\Lm_\om(\cL_{\sg^{-n}(\om)}^n f_{\sg^{-n}(\om)})}.
		\label{eq: def of lm_f^k limits}
	\end{align}
	The final limit in \eqref{eq: def of lm_f^k limits} telescopes to give us 
	\begin{align*}
		\lim_{n\to\infty} 
		\frac{\Lm_{\sg^k(\om)}(\cL_{\sg^{-n}(\om)}^{n+k} f_{\sg^{-n}(\om)})}
		{\Lm_\om(\cL_{\sg^{-n}(\om)}^n f_{\sg^{-n}(\om)})}
		&=
		\lim_{n\to\infty} 
		\frac{\Lm_{\sg(\om)}(\cL_{\sg^{-n}(\om)}^{n+1} f_{\sg^{-n}(\om)})}
		{\Lm_\om(\cL_{\sg^{-n}(\om)}^n f_{\sg^{-n}(\om)})}
		\cdots
		\frac{\Lm_{\sg^k(\om)}(\cL_{\sg^{-n}(\om)}^{n+k} f_{\sg^{-n}(\om)})}
		{\Lm_{\sg^{k-1}(\om)}(\cL_{\sg^{-n}(\om)}^{n+k-1} f_{\sg^{-n}(\om)})}
		\nonumber\\
		&=
		\lm_\om\lm_{\sg(\om)}\cdots\lm_{\sg^{k-1}(\om)},
	\end{align*}
	For each $k\geq 1$ we denote 
	\begin{align}\label{eq: def of lm_f^k}
		\lm_\om^k:=\lm_\om\lm_{\sg(\om)}\cdots\lm_{\sg^{k-1}(\om)}.
	\end{align}
	Rewriting \eqref{eq: def of lm_f^k limits} gives
	\begin{align*}
		\cL_\om^kq_\om=\lm_\om^k q_{\sg^k(\om)}.
	\end{align*}
\end{remark}
The following proposition shows that the density $q_\om$ coming from Corollary~\ref{cor: exist of unique dens q} is in fact supported on the set $D_{\om,\infty}$.
\begin{proposition}\label{prop: lower bound for density}
	For $m$-a.e. $\om\in\Om$ we have that
	\begin{align*}
		\inf_{D_{\om,\infty}} q_\om >0. 
	\end{align*} 
\end{proposition}
\begin{proof}
	
	First we note that since $\Lm_\om(q_\om)=1>0$ for $m$-a.e. $\om\in\Om$, using the definition of $\Lm_\om$ \eqref{eq: def of Lm}, we must in fact have that 
	$$
		\inf_{D_{\sg^n(\om),n}} \cL_\om^n(q_\om)>0
	$$ 
	for $n\in\NN$ sufficiently large, which, in turn implies that 
	\begin{align}\label{eq: q pos on X_n}
		\inf_{X_{\om,n-1}}q_\om>0
	\end{align} 
	for all $n\in\NN$ sufficiently large.
	Next, for $m$-a.e. $\om\in\Om$ and all $n\in\NN$ we use \eqref{eq: q invariant density} to see that 
	\begin{align}
		\inf_{D_{\om,\infty}}q_\om
		&=
		\lt(\lm_{\sg^{-n}(\om)}^n\rt)^{-1}\inf_{D_{\om,\infty}}\cL_{\sg^{-n}(\om)}^nq_{\sg^{-n}(\om)}
		\nonumber\\
		&\geq 
		\inf_{X_{\sg^{-n}(\om),n-1}} q_{\sg^{-n}(\om)} \lt(\lm_{\sg^{-n}(\om)}^n\rt)^{-1}\inf_{D_{\om,\infty}}\cL_{\sg^{-n}(\om)}^n\ind_{\sg^{-n}(\om)}.
	\end{align}	
	As the right hand side is strictly positive for all $n\in\NN$ sufficiently large by \eqref{eq: q pos on X_n}, \eqref{eq: rho leq lm leq a rho}, and \eqref{eq: L1 pos on D infty}; thus we are finished.
\end{proof}
\begin{lemma}\label{lem: Lm is a linear functional}
	For each $\om\in\Om$ the functional $\Lm_\om$ is linear, positive, and enjoys the property that 
	\begin{align}\label{eq: Lam equivariance}
		\Lm_{\sg(\om)}(\cL_\om f)=\lm_\om \Lm_\om(f)
	\end{align}
	for each $f\in\BV(I)$. Furthermore, for each $\om\in\Om$ we have that 
	\begin{align}\label{eq: lam functional integral def}
		\lm_\om=\rho_\om=\Lm_{\sg(\om)}(\cL_\om\ind_\om).
	\end{align}
\end{lemma}
\begin{proof}
	Positivity of $\Lm_\om$ follows from the initial properties of $\Lm_\om$ shown in Observation~\ref{obs: properties of Lm}. To prove the remaining claims we first prove a more robust limit characterization of $\Lm_\om$ than the one given by its definition, \eqref{eq: def of Lm}. 
	Now, for any two sequences of points $(x_n)_{n\geq 0}$ and $(y_n)_{n\geq 0}$ with $x_n,y_n\in D_{\sg^n(\om), n}$ we have 
	\begin{align}
		\lim_{n\to\infty}
		\absval{
			\frac{\cL_\om^n f}{\cL_\om^n \ind_\om}(x_n)
			-
			\frac{\cL_\om^n f}{\cL_\om^n \ind_\om}(y_n)
		}
		&=
		\lim_{n\to\infty}
		\absval{\frac{\cL_\om^n f}{\cL_\om^n \ind_\om}(y_n)}
		\cdot
		\absval{
			\frac
			{\cL_\om^n f(x_n) }
			{\cL_\om^n \ind_\om(x_n)}
			\cdot 
			\frac
			{\cL_\om^n \ind_\om(y_n)}
			{\cL_\om^n f(y_n)}
			-
			1
		}
		\nonumber\\
		&\leq 
		\norm{f}_\infty
		\limsup_{n\to\infty}
		\absval{\exp\lt(\Ta_{\sg^n(\om),+}\lt(\~\cL_\om^n f, \~\cL_\om^n\ind_\om\rt)\rt)-1}
		= 0.\label{eq: calc to show Lm does not need inf}
	\end{align} 
	Thus, we have shown that we may remove the infimum from \eqref{eq: def of Lm}, which defines the functional $\Lm_\om$, that is now we may write   
	\begin{align}\label{eq: lim characterization of Lm}
		\Lm_\om(f)= \lim_{n\to\infty}\frac{\cL_\om^n f}{\cL_\om^n\ind_\om}(x_n)
	\end{align}
	for all $f\in\sC_{\om,+}$ and all $x_n \in D_{\sg^n(\om),n}$. Moreover, this identity also shows that the functional $\Lm_\om$ is linear. 
	To extend \eqref{eq: lim characterization of Lm} to all of $\BV(I)$, we simply write $f=f_+ -f_-$ so that $f_+, f_-\in\sC_{\om,+}$ for each $f\in\BV(I)$ so that we have 
	\begin{align}\label{eq: lim char for all BV}
		\Lm_\om(f)
		=
		\Lm_\om(f_+)-\Lm_\om(f_-)
		=
		\lim_{n\to\infty}\frac{\cL_\om^n f_+}{\cL_\om^n\ind_\om} - \lim_{n\to\infty}\frac{\cL_\om^n f_-}{\cL_\om^n\ind_\om}
		=
		\lim_{n\to\infty}\frac{\cL_\om^n f}{\cL_\om^n\ind_\om}.
	\end{align} 
	To prove \eqref{eq: Lam equivariance} and \eqref{eq: lam functional integral def} we use \eqref{eq: lim char for all BV} to note that 
	\begin{align}\label{eq: Lam L f limit equality}
		\Lm_{\sg(\om)}(\cL_\om f) 
		&=\lim_{n\to \infty}\frac{\cL_\om^{n+1} f}{\cL_{\sg(\om)}^{n}\ind_{\sg(\om)}}(x_{n+1})
		\nonumber\\
		&=\lim_{n\to \infty}\frac{\cL_\om^{n+1} f}{\cL_\om^{n+1}\ind_\om}(x_{n+1})
		\cdot \frac{\cL_\om^{n+1} \ind_\om}{\cL_{\sg(\om)}^{n}\ind_{\sg(\om)}}(x_{n+1})
		\nonumber\\
		&=\Lm_\om(f)\cdot\Lm_{\sg(\om)}(\cL_\om\ind_\om).		
	\end{align}
	Considering the case where $f=q_\om$ in \eqref{eq: Lam L f limit equality} in conjunction with the fact that $\Lm_\om(q_\om)=1$ and $\cL_\om q_\om=\lm_\om q_{\sg(\om)}$ gives
	\begin{align}\label{eq: lam functional integral def in proof}
		\rho_\om:=
		\Lm_{\sg(\om)}(\cL_\om \ind_\om)
		=
		\Lm_\om(q_\om)\Lm_{\sg(\om)}(\cL_\om \ind_\om)
		=
		\Lm_{\sg(\om)}(\cL_\om q_\om)
		=
		\Lm_{\sg(\om)}(\lm_\om q_{\sg(\om)})
		=
		\lm_\om,
	\end{align}
	which finishes the proof.
\end{proof}
\begin{remark}
	In light of the fact that $\log\rho_\om\in L^1(m)$ by \eqref{eq: rho log int}, Lemma~\ref{lem: Lm is a linear functional} implies that 
	\begin{align}\label{eq: lm log int}
		\log\lm_\om\in L^1(m).
	\end{align}
\end{remark}
In the next lemma we are finally able to show that the functional $\Lm_\om$ can be thought of as Borel probability measure for the random open system. 
\begin{lemma}\label{lem: Lm is conf meas}
	There exists a non-atomic Borel probability measure $\nu_\om$ on $I_\om$ such that 
	$$
	\Lm_\om(f)=\int_{I_\om}f \, d\nu_\om
	$$
	for all $f\in\BV(I)$. Consequently, we have that  
	\begin{align}\label{eq: conformal measure property}
		\nu_{\sg(\om)}(\cL_\om f)=\lm_\om\nu_\om(f)
	\end{align}
	for all $f\in\BV(I)$.
	Furthermore, we have that $\supp(\nu_\om)\sub X_{\om,\infty}$.
\end{lemma}
\begin{proof}
	The proof that the functional $\Lm_\om$ can be equated to a non-atomic Borel probability measure $\nu_\om$ goes exactly like the proof of Lemma~4.3 in \cite{LMD}. Thus, we have only to prove that $\supp(\nu_\om)\sub X_{\om,\infty}$.
	To that end, suppose $f\in L^1(\nu_{\om,c})$ with $f\equiv0$ on $X_{\om,n-1}$. Then
	\begin{align*}
		\int_{I}f \,d\nu_{\om}
		&=\lt(\lm_\om^n\rt)^{-1}\int_{I}\cL_{\om}^n(f)\,d\nu_{\sg^n(\om)}
		=\lt(\lm_\om^n\rt)^{-1}\int_{I}\cL_{\om,c}^n(\hat X_{\om,n-1}\cdot f)\,d\nu_{\sg^n(\om)}
		=0.
	\end{align*}
	As $0<\lm_\om^n<\infty$ for each $n\in\NN$, we must have that $\supp(\nu_{\om})\sub X_{\om,\infty}$.
\end{proof}
\begin{remark}\label{rem: conformal meas prop}
	We can immediately see, cf. \cite{denker_existence_1991, atnip_critically_2020}, that the conformality of the family $(\nu_\om)_{\om\in\Om}$ produced in Lemma~\ref{lem: Lm is conf meas} enjoys the property that for each $n\geq 1$ and each set $A$ on which $T_\om^n\rvert_A$ is one-to-one we have 
	\begin{align*}
		\nu_{\sg^n(\om)}(T_\om^n(A))=\lm_\om^n\int_A e^{-S_{n,T}(\phi_\om)} \, d\nu_\om.
	\end{align*}
	In particular, this gives that for each $n\geq 1$ and each $Z\in\cZ_\om^{(n)}$ we have 
	\begin{align*}
		\nu_{\sg^n(\om)}(T_\om^n(Z))=\lm_\om^n\int_Z e^{-S_{n,T}(\phi_\om)} \, d\nu_\om.
	\end{align*} 	
\end{remark}

\begin{remark}\label{rem: props of norm op}
	In light of Lemmas~\ref{lem: Lm is a linear functional} and \ref{lem: Lm is conf meas}, the normalized operator $\~\cL_\om$ is given by $\~\cL_\om(\spot):=\rho_\om^{-1}\cL_\om(\spot)=\lm_\om^{-1}\cL_\om(\spot)$. Furthermore, $\~\cL_\om$ enjoys the properties 
	\begin{align*}
		\~\cL_\om q_\om = q_{\sg(\om)}
		\qquad\text{ and }\qquad
		\nu_{\sg(\om)}\lt(\~\cL_\om(f)\rt)=\nu_\om(f)
	\end{align*}
	for all $f\in\BV(I)$. 
\end{remark}

For each $\om\in\Om$ we may now define the measure $\mu_\om\in\cP(I)$ by 
\begin{align}\label{eq: def of mu_om}
	\mu_\om(f):=\int_{X_{\om,\infty}} fq_\om \,d\nu_\om,  \qquad f\in L^1(\nu_\om).
\end{align}
Lemma~\ref{lem: Lm is conf meas} and Proposition~\ref{prop: lower bound for density} together show that, for $m$-a.e. $\om\in\Om$, $\mu_\om$ is a non-atomic Borel probability measure with $\supp(\mu_\om)\sub X_{\om,\infty}$, which is absolutely continuous with respect to $\nu_\om$. Furthermore, in view of Proposition~\ref{prop: lower bound for density}, for $m$-a.e. $\om\in\Om$, we may now define the fully normalized transfer operator $\hcL_\om:\BV(I)\to\BV(I)$ by 
\begin{align}\label{eq: def fully norm tr op}
	\hcL_\om f:= \frac{1}{q_{\sg(\om)}}\~\cL_\om(fq_\om) = \frac{1}{\lm_\om q_{\sg(\om)}}\cL_\om (fq_\om), 
	\qquad f\in \BV(I).
\end{align}
As an immediate consequence of Remark~\ref{rem: props of norm op} and \eqref{eq: def fully norm tr op}, we get that 
\begin{align}\label{eq: fully norm op fix ind}
	\hcL_\om\ind_\om =\ind_{\sg(\om)}.
\end{align}
We end this section with the following proposition which shows that the family $(\mu_\om)_{\om\in\Om}$ of measures is $T$-invariant.
\begin{proposition}\label{prop: mu_om T invar}
	The family $(\mu_\om)_{\om\in\Om}$ defined by \eqref{eq: def of mu_om} is $T$-invariant in the sense that 
	\begin{align}\label{eq: mu_om T invar}
		\int_{X_{\om,\infty}} f\circ T_\om \, d\mu_\om 
		=
		\int_{X_{\sg(\om),\infty}} f \, d\mu_{\sg(\om)}
	\end{align}
	for $f\in L^1(\mu_{\sg(\om)})=L^1(\nu_{\sg(\om)})$.	
\end{proposition} 
The proof of Proposition~\ref{prop: mu_om T invar} goes just like the proof of Proposition~8.11 of \cite{AFGTV21}, and has thus been omitted. 

\section{Decay of Correlations}\label{sec: dec of cor}
We are now ready to show that images under the normalized transfer operator $\~\cL_\om$ converge exponentially fast to the invariant density as well as the fact that the invariant measure $\mu_\om$ established in Section~\ref{sec: conf and inv meas} satisfies an exponential decay of correlations. Furthermore, we show that the families $(\nu_\om)_{\om\in\Om}$ and $(\mu_\om)_{\om\in\Om}$ are in fact random measures as defined in Section~\ref{sec: rand meas} and then introduce the RACCIM $\eta$ supported on $\cI$. 

To begin this section we state a lemma which shows that the $\BV$ norm of the invariant density $q_\om$ does not grow too much along a $\sg$-orbit of fibers by providing a measurable upper bound. In fact, we show that the $\BV$ norm of $q_\om$ is tempered. As the proof of the following lemma is the same as the proof of Lemma 8.5 in the closed dynamical setting of \cite{AFGTV21}, its proof is omitted. 
\begin{lemma}\label{lem: BV norm q om growth bounds}
	For all $\dl>0$ there exists a measurable random constant $C(\om,\dl)>0$ such that for all $k\in\ZZ$ and $m$-a.e. $\om\in\Om$ we have 
	\begin{align*}
		\norm{q_{\sg^k(\om)}}_{\BV}
		=
		\norm{q_{\sg^k(\om)}}_{\infty}+\var(q_{\sg^k(\om)})
		\leq
		C(\om,\dl)e^{\dl|k|}.
	\end{align*} 
	Consequently, we have that $q\in\cD$.
\end{lemma}
We are now able to prove the following theorem which completes the proof of Theorem~\ref{main thm: quasicompact}. 
\begin{theorem}\label{thm: exp convergence of tr op}
	There exists a measurable, $m$-a.e. finite function $D:\Om\to\RR$ and $\kp<1$ such that for each $f\in\cD$, each $n\in\NN$, and each $|p|\leq n$ we have
	\begin{align}\label{eq: exp conv norm op}
		\norm{\~\cL_{\sg^p(\om)}^n f_{\sg^p(\om)} - \nu_{\sg^p(\om)}(f_{\sg^p(\om)})q_{\sg^{p+n}(\om)}}_\infty
		\leq 
		D(\om)\norm{f_{\sg^p(\om)}}_\infty\kp^n
	\end{align}
	and 
	\begin{align}\label{eq: exp conv full norm op}
		\norm{\hcL_{\sg^p(\om)}^n f_{\sg^p(\om)} - \mu_{\sg^p(\om)}(f_{\sg^p(\om)})\ind_{\sg^{p+n}(\om)}}_\infty
		\leq 
		D(\om)\norm{f_{\sg^p(\om)}}_\infty\kp^n.
	\end{align}
\end{theorem}
\begin{proof}
	We first note that for $m$-a.e. $\om\in\Om$, all $n\geq N_3(\om)$, all $|p|\leq n$, and all  $f\in\cD_+$ we may use Lemma~\ref{lem: exp conv in C+ cone} to get that 
	\begin{align*}
		&\Ta_{\sg^{p+n}(\om),+}\lt(\~\cL_{\sg^p(\om)}^nf_{\sg^p(\om)}, \nu_{\sg^p(\om)}(f_{\sg^p(\om)})q_{\sg^{p+n}(\om)}\rt)
		\\
		&\qquad\qquad
		=
		\Ta_{\sg^{p+n}(\om),+}\lt(\~\cL_{\sg^p(\om)}^nf_{\sg^p(\om)},\nu_{\sg^p(\om)}(f_{\sg^p(\om)})\~\cL_{\sg^p(\om)}^nq_{\sg^p(\om)}\rt)
		\leq \Dl\vta^n.
	\end{align*}
	Applying Lemma~\ref{lem: birkhoff cone contraction} with $\vrho=\nu_{\sg^{p+n}(\om)}$ and $\norm{\spot}=\norm{\spot}_\infty$, together with Lemma~\ref{lem: exp conv in C+ cone} then gives 
	\begin{align}
		&\norm{\~\cL_{\sg^p(\om)}^nf_{\sg^p(\om)}-\nu_{\sg^p(\om)}(f_{\sg^p(\om)})q_{\sg^{p+n}(\om)}}_\infty
		\nonumber\\
		&\quad
		\leq 
		\lt(\exp\lt(\Ta_{\sg^{p+n}(\om),+}\lt(\~\cL_{\sg^{p}(\om)}^nf_{\sg^{p}(\om)}, \nu_{\sg^p(\om)}(f_{\sg^p(\om)})q_{\sg^{p+n}(\om)}\rt)\rt)-1\rt)
		\nu_{\sg^p(\om)}(f_{\sg^p(\om)})\norm{q_{\sg^{p+n}(\om)}}_\infty 
		\nonumber\\
		&\quad
		\leq 
		\norm{f_{\sg^p(\om)}}_\infty\norm{q_{\sg^{p+n}(\om)}}_\infty\lt(e^{\Dl\vta^n}-1\rt)
		\nonumber\\
		&\quad
		\leq 
		\norm{f_{\sg^p(\om)}}_\infty\norm{q_{\sg^{p+n}(\om)}}_\infty\~\kp^n
		\label{eq: exp conv eq 2}
	\end{align}
	for some $\~\kp\in(0,1)$\footnote{Any $\~\kp>\vta$ will work for $n$ sufficiently large.}. 	
	Since  $\norm{q_\om}_\infty$ is tempered, as a consequence of Lemma~\ref{lem: BV norm q om growth bounds}, for each $n\in\NN$ and each $\dl>0$ we can find a tempered function $A_{n,\dl}:\Om\to\RR$ such that 
	\begin{align}
		\norm{q_{\sg^{p+n}(\om)}}_\infty
		&\leq 
		A_{n,\dl}(\om)e^{(p+n)\dl}\norm{q_\om}_\infty
		\leq 
		A_{n,\dl}(\om)e^{2n\dl}\norm{q_\om}_\infty,
		\label{eq: exp conv eq 3}
	\end{align}
where we have used the fact that $|p|\leq n$.
	For each $\dl>0$ we let
	\begin{align}\label{eq: exp conv eq 4}
		B_\dl(\om)
		:=
		\max_{1\leq n\leq N_3(\om)}\set{A_{n,\dl}(\om)}\~\kp^{-N_3(\om)}.
	\end{align}
	Combining \eqref{eq: exp conv eq 2} - \eqref{eq: exp conv eq 4}, for any $n\in\NN$ we see that 
	\begin{align}
		\norm{\~\cL_{\sg^p(\om)}^nf_{\sg^p(\om)}-\nu_{\sg^p(\om)}(f_{\sg^p(\om)})q_{\sg^{p+n}(\om)}}_\infty
		&\leq 
		2B_\dl(\om)e^{2n\dl}\norm{f_{\sg^p(\om)}}_\infty\norm{q_\om}_\infty\~\kp^n
		\nonumber\\
		&\leq 
		2B(\om)\norm{f_{\sg^p(\om)}}_\infty\norm{q_\om}_\infty\kp^n,
		\label{eq: Ca exp conv eq}
	\end{align}
	where here we have fixed $\dl>0$ sufficiently small such that 
	\begin{align*}
		e^{2\dl}\~\kp=:\kp<1,
	\end{align*}
	and we have set $B(\om)=B_\dl(\om)$
	
	Now, to extend \eqref{eq: Ca exp conv eq} to all of $\cD$ we write a function $f\in\cD_+$ as $f=f_+-f_-$, where $f_+,f_-\in\cD_+$. Applying the triangle inequality and using \eqref{eq: Ca exp conv eq} twice gives
	\begin{align*}
		&\norm{\~\cL_{\sg^p(\om)}^n f_{\sg^p(\om)} - \nu_{\sg^p(\om)}(f_{\sg^p(\om)})q_{\sg^{p+n}(\om)}}_\infty
		\leq
		4B(\om)\norm{f_{\sg^p(\om)}}_\infty\norm{q_\om}_\infty\kp^n.
	\end{align*}
	Setting $D(\om):=4B(\om)\norm{q_\om}_\infty$ finishes the proof of \eqref{eq: exp conv norm op}. To prove the second claim follows easily from the first claim in a similar fashion as in Theorem 10.4 of \cite{AFGTV21}.
\end{proof}

From the previous result we easily deduce that the invariant measure $\mu$ satisfies an exponential decay of correlations. The following theorem, whose proof is exactly the same as Theorem~11.1 of \cite{AFGTV21}, completes the proof of Theorem~\ref{main thm: exp dec of corr}.
\begin{theorem}\label{thm: dec of cor}
	For $m$-a.e. every $\om\in\Om$, every $n\in\NN$, every $|p|\leq n$, every $f\in L^1(\mu)$, and every $h\in \cD$ we have 
	\begin{align*}
		\absval{
			\mu_{\tau}
			\lt(\lt(f_{\sg^{n}(\tau)}\circ T_{\tau}^n\rt)h_{\tau} \rt)
			-
			\mu_{\sg^{n}(\tau)}(f_{\sg^{n}(\tau)})\mu_{\tau}(h_{\tau})
		}
		\leq 
		D(\om)\norm{f_{\sg^n(\tau)}}_{L^1(\mu_{\sg^n(\tau)})}\norm{h_\tau}_\infty \kp^n,
	\end{align*} 
	where $\tau=\sg^p(\om)$. 
\end{theorem}

\begin{remark}\label{rem: strong lim char of nu}
	Note that Theorem~\ref{thm: exp convergence of tr op} implies a stronger limit characterization of the measure $\nu_\om$ than what is concluded in \eqref{eq: lim characterization of Lm}. Indeed, Theorem~\ref{thm: exp convergence of tr op} implies that 
	\begin{align*}
		\nu_\om(f)=\lim_{n\to\infty}\frac{\cL_\om^n f(x_n)}{\cL_\om^n\ind_\om (y_n)}
	\end{align*}
	for any pair of sequences $(x_n)_{n\in\NN}$ and $(y_n)_{n\in\NN}$ with $x_n,y_n\in D_{\sg^n(\om),\infty}$, which further implies that 
	\begin{align}\label{eq: improved limit characterization}
		\nu_\om(f)
		=\lim_{n\to\infty}\frac{\norm{\cL_\om^n f}_\infty}{\inf\cL_\om^n\ind_\om}
		=\lim_{n\to\infty}\frac{\inf\cL_\om^n f}{\norm{\cL_\om^n\ind_\om}_\infty}.
	\end{align}
	Furthermore, the same holds for $\nu_{\om,c}$; see Lemma 9.2 of \cite{AFGTV21}.
\end{remark}
We now address the uniqueness of the families of measures $\nu=(\nu_\om)_{\om\in\Om}$ and $\mu=(\mu_\om)_{\om\in\Om}$ as well as the invariant density $q$.
\begin{proposition}\label{prop: uniqueness of nu and mu}
	\,\newline
	\begin{enumerate}
		\item 
		The family $\nu=(\nu_\om)_{\om\in\Om}$ is a random probability measure which is uniquely determined by \eqref{eq: conformal measure property}. 
		
		\item The global invariant density $q\in\cD$ produced in Corollary~\ref{cor: exist of unique dens q} is the unique element of $L^1(\nu)$ (modulo $\nu$) such that 
		\begin{align*}
			\~\cL_\om q_\om=q_{\sg(\om)}.
		\end{align*} 	
		\item The family $\mu=(\mu_\om)_{\om\in\Om}$ is a unique random $T$-invariant probability measure which is absolutely continuous with respect to $\nu$.
	\end{enumerate}
\end{proposition}
\begin{proof}
	The fact that the family $(\nu_\om)_{\om\in\Om}$ is a random measure as in Definition~\ref{def: random prob measures} follows from the limit characterization given in \eqref{eq: improved limit characterization}, as we have that $\nu_\om$ is a limit of measurable functions. 
	Indeed, for every interval $J\subset I$, the measurability of the function $\om \mapsto \nu_\om(J)$ follows from the fact that it is given by the limit of measurable functions by \eqref{eq: improved limit characterization} applied to the characteristic function $f_\om=\ind_J$. 
	Since $\sB$ is generated by intervals, $\om \mapsto \nu_\om(B)$ is measurable for every $B\in \sB$. Furthermore, $\nu_\om$ is a Borel probability measure for $m$-a.e. $\om \in \Om$ from Proposition~\ref{lem: Lm is conf meas}.
	
	The remainder of the proof Proposition~\ref{prop: uniqueness of nu and mu} follows along exactly like the proofs of Propositions~9.4 and 10.5 of \cite{AFGTV21}, and is therefore left to the reader.
\end{proof}
The proof of the  following proposition is the same as the proof of Proposition 4.7 of \cite{mayer_distance_2011}, and so it is omitted.
\begin{proposition}
	The random $T$-invariant probability measure $\mu$ defined in \eqref{eq: def of mu_om} is ergodic.
\end{proposition}
In the following lemma we establish the existence of the unique random absolutely continuous conditionally invariant probability measure $\eta$.
\begin{lemma}
	The random measure $\eta\in\cP_\Om(I)$, whose disintegrations are given by 
	\begin{align*}
		\eta_\om(f):=\frac{\nu_{\om,c}(f\cdot\ind_\om\cdot q_\om)}{\nu_{\om,c}(\ind_\om\cdot q_\om)},
	\end{align*} 
	is a unique random absolutely continuous conditionally invariant probability measure with respect to $\nu_c$ and is supported on $\cI$. 
\end{lemma}
\begin{proof}
	The fact that $\eta$ is an RACCIM follows from Lemma~\ref{LMD lem 1.1.1} and uniqueness follows from Proposition~\ref{prop: uniqueness of nu and mu}. 
\end{proof}
As a corollary of Theorem~\ref{thm: exp convergence of tr op}, the following results gives the exponential convergence of the closed conformal measure $\nu_{\om,c}$ conditioned on the survivor set to the RACCIM $\eta_\om$.   
\begin{corollary}\label{cor: exp conv of eta}
	For $m$-a.e. every $\om\in\Om$, every $n\in\NN$, every $|p|\leq n$, and every $A,B\in\sB$ we have 
	\begin{align*}
		\absval{
			\nu_{\sg^p(\om),c}\lt(T_{\sg^p(\om)}^{-n}(A) \, \rvert \, X_{\sg^p(\om),n}\rt)
			- 
			\eta_{\sg^{p+n}(\om)}(A)
		}\leq \frac{D(\om)}{\nu_{\sg^{p+n}(\om),c}\lt(\ind_{\sg^{p+n}(\om)}q_{\sg^{p+n}(\om)}\rt)}\kp^n
	\end{align*}
	and
	\begin{align*}
		\absval{
			\eta_{\sg^p(\om)}\lt(A \, \rvert \, X_{\sg^p(\om),n}\rt)
			- 
			\mu_{\sg^p(\om)}(A)
		}\leq D(\om)\kp^n.
	\end{align*}
\end{corollary}
\begin{proof}
	For $A\in\sB$, Lemma~\ref{lem: useful identity} allows us to write 
	\begin{align*}
		\nu_{\sg^{p}(\om),c}\lt(T_{\sg^p(\om)}^{-n}(A)\cap X_{\sg^p(\om), n}\rt)
		&=
		\int_{X_{\sg^p(\om),n}}\ind_{\sg^p(\om)}\ind_A\circ T_{\sg^p(\om)}^n \, d_{\sg^p(\om),c}
		\\
		&=\lt(\lm_{\sg^p(\om),c}^n\rt)^{-1}\int_{I_{\sg^{p+n}(\om)}}\ind_A\cL_{\sg^p(\om)}^n\ind_{\sg^p(\om)}\, d\nu_{\sg^{p+n}(\om),c}
		\\
		&=
		\nu_{\sg^{p+n}(\om),c}\lt(\ind_A\ind_{\sg^{p+n}(\om)} \~\cL_{\sg^p(\om)}^n(\ind_{\sg^p(\om)})\rt).
	\end{align*}
	So, if $A=I$, then we have 
	\begin{align*}
		\nu_{\sg^{p}(\om),c}\lt(X_{\sg^p(\om),n}\rt)
		=
		\nu_{\sg^{p+n}(\om),c}\lt(\ind_{\sg^{p+n}(\om)}\~\cL_{\sg^p(\om)}^n(\ind_{\sg^p(\om)})\rt).
	\end{align*}
	Thus, we apply \eqref{eq: exp conv norm op} of Theorem~\ref{thm: exp convergence of tr op} together with elementary calculation to get that 
	\begin{align*}
		&\absval{
		\nu_{\sg^{p}(\om),c}\lt(T_{\sg^p(\om)}^{-n}(A)\, \rvert\, X_{\sg^p(\om),n}\rt)
		-
		\eta_{\sg^{p+n}(\om)}(A)	
		}
		\\
		&\qquad
		=
		\absval{
		\frac{\nu_{\sg^{p+n}(\om),c}\lt(\ind_A\ind_{\sg^{p+n}(\om)}\~\cL_{\sg^p(\om)}^n(\ind_{\sg^p(\om)})\rt)}
		{\nu_{\sg^{p+n}(\om),c}\lt(\ind_{\sg^{p+n}(\om)}\~\cL_{\sg^p(\om)}^n(\ind_{\sg^p(\om)})\rt)}
		-
		\frac{\nu_{\sg^{p+n}(\om),c}\lt(\ind_A\ind_{\sg^{p+n}(\om)}q_{\sg^{p+n}(\om)}\rt)}{\nu_{\sg^{p+n}(\om),c}\lt(\ind_{\sg^{p+n}(\om)}q_{\sg^{p+n}(\om)}\rt)}
		}
		\\
		&\qquad
		\leq \frac{D(\om)}{\nu_{\sg^{p+n}(\om),c}\lt(\ind_{\sg^{p+n}(\om)}q_{\sg^{p+n}(\om)}\rt)}\kp^n.
	\end{align*}
	To see the second claim we note that for $A\in\sB$
	\begin{align*}
		\eta_{\sg^p(\om)}\lt(A\, \rvert\, X_{\sg^p(\om),n}\rt)
		&=
		\frac
		{\nu_{\sg^p(\om),c}\lt(\ind_A \hat X_{\sg^p(\om),n}\ind_{\sg^p(\om)}q_{\sg^p(\om)}\rt)}
		{\nu_{\sg^p(\om),c}\lt(\hat X_{\sg^p(\om),n}\ind_{\sg^p(\om)}q_{\sg^p(\om)}\rt)}
		\\
		&=
		\frac
		{\nu_{\sg^{p+n}(\om),c}\lt(\ind_{\sg^{p+n}(\om)}\~\cL_{\sg^p(\om)}^n(\ind_A q_{\sg^p(\om)})\rt)}
		{\nu_{\sg^{p+n}(\om),c}\lt(\ind_{\sg^{p+n}(\om)}\~\cL_{\sg^p(\om)}^n(q_{\sg^p(\om)})\rt)}
		\\
		&=
		\frac
		{\nu_{\sg^{p+n}(\om),c}\lt(\ind_{\sg^{p+n}(\om)}\hcL_{\sg^p(\om)}^n(\ind_A) q_{\sg^{p+n}(\om)}\rt)}
		{\nu_{\sg^{p+n}(\om),c}\lt(\ind_{\sg^{p+n}(\om)}q_{\sg^{p+n}(\om)}\rt)}
		\\
		&=
		\eta_{\sg^{p+n}(\om)}\lt(\hcL_{\sg^p(\om)}^n(\ind_A)\rt).
	\end{align*} 
	Thus, applying \eqref{eq: exp conv full norm op} of Theorem~\ref{thm: exp convergence of tr op}, we have 
	\begin{align*}
		\absval{
			\eta_{\sg^p(\om)}\lt(A \, \rvert \, X_{\sg^p(\om),n}\rt)
			- 
			\mu_{\sg^p(\om)}(A)
		}
		=
		\absval{
		\eta_{\sg^{p+n}(\om)}\lt(\hcL_{\sg^p(\om)}^n(\ind_A)\rt)
		-
		\mu_{\sg^p(\om)}(A)
		}
		\leq 
		D(\om)\kp^n,
	\end{align*}
	which finishes the proof.
\end{proof}

\section{Expected Pressures and Escape Rates}\label{sec: exp press}
We now establish the rate at which mass escapes through the hole with respect to the closed conformal measure $\nu_{\om,c}$ and the RACCIM $\eta_\om$ in terms of the open and closed expected pressures.
Given a (closed) potential $\phi_c$, which generates the (open) potential $\phi$, in light of \eqref{eq: lm log int}, we recall the definition of the expected pressures $\cEP(\phi_c),\cEP(\phi)\in\RR$, of $\phi_c$ and $\phi$ respectively:
\begin{align}
	\cEP(\phi_c):=\int_\Om\log\lm_{\om,c}\, dm(\om)
	\quad\text{ and }\quad
	\cEP(\phi):=\int_\Om\log\lm_\om\, dm(\om). 
\end{align}
Since $\log\lm_\om, \log\lm_{\om,c}\in L^1(m)$, Birkhoff's Ergodic Theorem gives that 
\begin{align}\label{eq: BET exp pres hole}
	\cEP(\phi)
	=\lim_{n\to\infty}\frac{1}{n}\log\lm_\om^n
	=\lim_{n\to\infty}\frac{1}{n}\log\lm_{\sg^{-n}(\om)}^n
\end{align} 
and
\begin{align}\label{eq: BET exp pres closed}
	\cEP(\phi_c)
	=\lim_{n\to\infty}\frac{1}{n}\log\lm_{\om,c}^n
	=\lim_{n\to\infty}\frac{1}{n}\log\lm_{\sg^{-n}(\om),c}^n.
\end{align}
The following lemma, which is the open analogue of Lemma 10.1 of \cite{AFGTV21}, gives an alternate method for calculating the expected pressure.
\begin{lemma}\label{lem: conv of pressure limits}
	For $m$-a.e. $\om\in\Om$ we have that 
	\begin{align}\label{eq: backward pressure limit}
		\lim_{n\to\infty}\norm{\frac{1}{n}\log\cL_{\sg^{-n}(\om)}^n\ind_{\sg^{-n}(\om)} -\frac{1}{n}\log\lm_{\sg^{-n}(\om)}^n}_\infty=0
	\end{align}
	and 
	\begin{align}\label{eq: forward pressure limit}
		\lim_{n\to\infty}\norm{\frac{1}{n}\log\cL_{\om}^n\ind_{\om} -\frac{1}{n}\log\lm_{\om}^n}_\infty=0.
	\end{align}	
\end{lemma}
As the proof of the previous lemma is exactly the same as the proof of Lemma 10.1 of \cite{AFGTV21}, the proof is left to the reader. Now, in view of the fact that 
\begin{align*}
	\cL_\om\ind_\om \leq \cL_{\om,c}\ind,
\end{align*}
Lemma~\ref{lem: conv of pressure limits} and Lemma 10.1 of \cite{AFGTV21}, together with \eqref{eq: BET exp pres hole} and \eqref{eq: BET exp pres closed} imply that 
\begin{align}\label{eq: open pres leq closed pres}
	\cEP(\phi)\leq \cEP(\phi_c).
\end{align}
We now prove the following corollary of Lemma~\ref{lem: conv of pressure limits}.
\begin{lemma}\label{leq: inf q tempered}
	For $m$-a.e. $\om\in\Om$ we have that 
	\begin{align*}
		\lim_{n\to\infty}\frac{1}{n}\log\inf_{D_{\sg^n(\om),\infty}} q_{\sg^n(\om)}=0.
	\end{align*}
\end{lemma}
\begin{proof}
	We first note that Lemma~\ref{lem: conv of pressure limits} immediately implies that 
\begin{align}\label{eq: norm tr op forw is temp}
	\lim_{n\to\infty}\frac{1}{n}\log\inf_{D_{\sg^n(\om),\infty}}\~\cL_\om^n\ind_\om
	=
	\lim_{n\to\infty}\frac{1}{n}\log\norm{\~\cL_\om^n\ind_\om}_\infty
	=0
\end{align}
and 
\begin{align}\label{eq: norm tr op back is temp}
	\lim_{n\to\infty}\frac{1}{n}\log\inf_{D_{\om,\infty}}\~\cL_{\sg^{-n}(\om)}^n\ind_{\sg^{-n}(\om)}
	=
	\lim_{n\to\infty}\frac{1}{n}\log\norm{\~\cL_{\sg^{-n}(\om)}^n\ind_{\sg^{-n}(\om)}}_\infty
	=0.
\end{align}
It follows from Theorem~\ref{thm: exp convergence of tr op} that for $m$-a.e. $\om\in\Om$ and all $n\in\NN$ sufficiently large 
\begin{align*}
	\inf_{D_{\sg^n(\om),\infty}} q_{\sg^n(\om)}
	\geq 
	\inf_{D_{\sg^n(\om),\infty}}\~\cL_\om^n\ind_\om -D(\om)\kp^n,
\end{align*}
and thus  
\begin{align*}
	\frac{1}{n}\log\inf_{D_{\sg^n(\om),\infty}} q_{\sg^n(\om)}
	\geq 
	\frac{1}{n}\log\lt(\inf_{D_{\sg^n(\om),\infty}}\~\cL_\om^n\ind_\om -D(\om)\kp^n\rt).
\end{align*}
In light of \eqref{eq: norm tr op forw is temp}, we have that the right hand side goes to zero as $n\to\infty$, and thus the proof is complete. 
\end{proof}
Recall that the escape rate of a random probability measure $\vrho$ is given by 
\begin{align*}
	R(\vrho):=-\lim_{n\to\infty}\frac{1}{n}\log \vrho_\om(X_{\om,n})
\end{align*}
The previous results allow us to calculate the following escape rates, thus proving Theorem~\ref{main thm: escape rate}.
\begin{proposition}\label{prop: escape rate}
	For $m$-a.e. $\om\in\Om$ we have that 
	\begin{align*}
		R(\nu_{\om,c})=R(\eta_\om)=\cEP(\phi_c)-\cEP(\phi).
	\end{align*}
\end{proposition}
\begin{proof}
	We begin by noting that 
	\begin{align*}
		\nu_{\om,c}(X_{\om,n-1})
		&=
		\lt(\lm_{\om,c}^n\rt)^{-1}\nu_{\sg^n(\om),c}\lt(\cL_{\om,c}^n(\hat X_{\om,n-1})\rt)
		=\frac{\lm_\om^n}{\lm_{\om,c}^n}\nu_{\sg^n(\om),c}\lt(\~\cL_\om^n\ind_\om\rt)
		\\
		&=\frac{\lm_\om^n}{\lm_{\om,c}^n}\lt(\nu_{\sg^n(\om),c}(q_{\sg^n(\om)})-\nu_{\sg^n(\om),c}\lt(\~\cL_\om^n\ind_\om -q_{\sg^n(\om)}\rt)\rt).
	\end{align*}
	Using Theorem~\ref{thm: exp convergence of tr op} gives that
	\begin{align*}
		-R(\nu_{\om,c})
		&=
		\lim_{n\to\infty}\frac{1}{n}\log\frac{\lm_\om^n}{\lm_{\om,c}^n}
		+
		\lim_{n\to\infty}\frac{1}{n}\log\nu_{\sg^n(\om),c}(q_{\sg^n(\om)}).
	\end{align*}
	Thus, the temperedness of $\inf_{D_{\om,\infty}} q_\om$ and $\norm{q_\om}_\infty$, coming from Lemmas~\ref{lem: BV norm q om growth bounds} and \ref{leq: inf q tempered} respectively, imply that 
	\begin{align*}
		0=
		\lim_{n\to\infty}\frac{1}{n}\log\inf_{D_{\sg^n(\om),\infty}} q_{\sg^n(\om)}
		\leq 
		\lim_{n\to\infty}\frac{1}{n}\log\nu_{\sg^n(\om),c}(q_{\sg^n(\om)})
		\leq 
		\lim_{n\to\infty}\frac{1}{n}\log\norm{q_{\sg^n(\om)}}_\infty
		=0,
	\end{align*}
	which, when combined with \eqref{eq: BET exp pres hole} and \eqref{eq: BET exp pres closed}, completes the proof of the first claim. As the second equality follows similarly to the first, we are done. 	
\end{proof}
	\begin{remark}
		If there exists a $T$-invariant measure $\mu_c$ on the closed system which is absolutely continuous with respect to $\nu_c$ then the proof of Proposition~\ref{prop: escape rate}, with minor adjustments, also shows that for $m$-a.e. $\om\in\Om$ we have that 
		\begin{align*}
			R(\nu_{\om,c})=R(\mu_{\om,c})=\cEP(\phi_c)-\cEP(\phi).
		\end{align*}
	\end{remark}

\section{Bowen's Formula}\label{sec: bowen}
This section is devoted to proving a formula for the Hausdorff dimension of the survivor set in terms of the expected pressure function, which was first proven by Bowen in \cite{bowen_hausdorff_1979} in the setting quasi-Fuchsian groups. 
In this section we will consider geometric potentials of the form $\phi_{c,t}(\om,x)=-t\log |T_\om'|(x)$ for $t\in[0,1]$. We denote the expected pressure of $\phi_{c,t}$ by $\cEP_c(t)$ and the expected pressure of the open potential $\phi_t$ by $\cEP(t)$. In the case that $t=1$, the fiberwise closed conformal measures $\nu_{\om,c,1}$ are equal to Lebesgue measure and $\lm_{\om,c,1}=1$. Furthermore, we note that for any $t\geq0$ we have that 
\begin{align}\label{eq: g_t in terms of g_1}
	g_{\om,c,t}^{(n)}=\lt(g_{\om,c,1}^{(n)}\rt)^t=\frac{1}{|(T_\om^n)'|^t}. 
\end{align} 

\begin{definition}\label{def: bdd dist}
	We will say that the weight function $g_{\om,c}$ has the \textit{Bounded Distortion Property} if for $m$-a.e. $\om\in\Om$ there exists $K_\om\geq 1$ such that for all $n\in\NN$, all $Z\in\cZ_{\om}^{(n)}$, and all $x,y\in Z$ we have that 
	\begin{align*}
		\frac{g_{\om,c}^{(n)}(x)}{g_{\om,c}^{(n)}(y)}\leq K_\om.
	\end{align*}
\end{definition}
We now adapt the following definitions from \cite{LMD} to the random setting.
\begin{definition}\label{def: large images}
	We will say that the map $T$ has \textit{large images} if for $m$-a.e. $\om\in\Om$ we have 
	\begin{align*}
		\inf_{n\in\NN}\inf_{Z\in\cZ_\om^{(n)}}\nu_{\sg^n(\om),c}\lt(T_\om^n(Z)\rt)>0.
	\end{align*}
	$T$ is said to have \textit{large images with respect to $H$} if for $m$-a.e. $\om\in\Om$, each $n\in\NN$, and each $Z\in\cZ_\om^{(n)}$ with $Z\cap X_{\om,\infty}\neq\emptyset$ we have 
	\begin{align*}
		T_\om^n(Z\cap X_{\om,n-1})\bus X_{\sg^n(\om),\infty}.
	\end{align*}
\end{definition}
\begin{remark}\label{rem: full support}
	If $T$ has large images with respect to $H$ then it follows from Remark~\ref{rem: conformal meas prop} that $\supp(\nu_{\om,t})=X_{\om,\infty}$ for any $t\in[0,1]$.
\end{remark}
We now prove a formula for the Hausdorff dimension of the surviving set, \`a la Bowen, proving Theorem~\ref{main thm: Bowens formula}.
\begin{theorem}\label{thm: Bowen's Formula}
	There exists a unique $h\in [0,1]$ such that $\cEP(t)>0$ for all $0\leq t<h$ and $\cEP(t)<0$ for all $h<t\leq 1$.
	Furthermore, if $T$ has large images and large images with respect to $H$, then for $m$-a.e. $\om\in\Om$
	\begin{align*}
		\HD(X_{\om,\infty})=h.		
	\end{align*}
\end{theorem}
\begin{proof}
	We will prove this theorem in a series of lemmas.
	\begin{lemma}\label{lem: bf lem 1}
		The function $\cEP(t)$ is strictly decreasing and there exists $h\in[0,1]$ such that $\cEP(t)>0$ for all $0\leq t<h$ and $\cEP(t)<0$ for all $h<t\leq 1$.
	\end{lemma}
	\begin{proof}
		We first note that, using \eqref{eq: g_t in terms of g_1}, for any $n\in\NN$ and $s<t\in[0,1]$ we can write  
		\begin{align*}
			\cL_{\om,t}^n\ind_\om
			\leq 
			\norm{g_{\om,1}^{(n)}}_\infty^{t-s}\cL_{\om,s}^n\ind_\om.  
		\end{align*}
		This immediately implies that 
		\begin{align*}
			\cEP(t)< \cEP(s)
		\end{align*}
		since 
		\begin{align*}
			\lim_{n\to \infty}\frac{1}{n}\log\norm{g_{\om,1}^{(n)}}_\infty<0.
		\end{align*}
		Now since $\cEP(0)\geq 0$ and $\cEP(1)\leq \cEP_c(1)=0$, there must exist some $h\in[0,1]$ such that for all $s<h<t$ we have 
		\begin{align*}
			\cEP(t)<0<\cEP(s).
		\end{align*} 
	\end{proof}
	To prove the remaining claim of Theorem~\ref{thm: Bowen's Formula}, we now suppose that $T$ has large images and large images with respect to $H$.
	\begin{lemma}\label{lem: bf lem 2}
		If $T$ has large images and large images with respect to $H$, and $\nu_{\om,t}(Z)>0$
		for all $t\in[0,1]$, all $n\in\NN$, and all $Z\in\cZ_\om^{(n)}$, then for all $x\in Z$ we have 
		\begin{align*}
			K_\om^{-1}\leq\frac{\lt(g_{\om,c,1}^{(n)}\rt)^t(x)}{\lm_{\om,t}^n\nu_{\om,t}(Z)} \leq K_\om
			\quad \text{ and } \quad
			K_\om^{-1}\leq\frac{g_{\om,c,1}^{(n)}(x)}{\Leb(Z)} \leq K_\om.
		\end{align*}
	\end{lemma}
	\begin{proof}
		In light of Remark~\ref{rem: full support}, $\supp(\nu_{\om,t})=X_{\om,\infty}$, and thus for any $Z\in\cZ_\om^{(n)}$ for any $n\geq 1$, $\nu_{\om,t}(Z)>0$ if and only if $Z\cap X_{\om,\infty}\neq\emptyset$. 
		Furthermore, since $T$ has large images with respect to $H$, we have that 
		\begin{align}\label{eq: measur cons of LI wrt H}
			\nu_{\sg^n(\om),t}(T_\om^n(Z))=1
		\end{align}
		for any $Z\in\cZ_\om^{(n)}$ with $Z\cap X_{\om,\infty}\neq\emptyset$.
		Thus, we may write 
		\begin{align}
			\nu_{\om,t}(Z)
			&=
			\int_{X_{\om,\infty}} \ind_Z \,d\nu_{\om,t}
			=
			\lt(\lm_{\om,t}^n\rt)^{-1}\int_{X_{\sg^n(\om),\infty}}\cL_{\om,t}^n\ind_Z\, d\nu_{\sg^n(\om),t}
			\nonumber\\
			&=
			\lt(\lm_{\om,t}^n\rt)^{-1}\int_{X_{\sg^n(\om),\infty}}\cL_{\om,c,t}^n\left(\ind_Z \hat X_{\om,n-1}\right)\, d\nu_{\sg^n(\om),t}
			\nonumber\\
			&=
			\lt(\lm_{\om,t}^n\rt)^{-1}\int_{T_\om^n(Z)} \lt(\lt(g_{\om,c,1}^{(n)}\rt)^t \hat X_{\om,n-1}\rt)\circ T_{\om,Z}^{-n}\, d\nu_{\sg^n(\om),t}
			\nonumber\\
			&=
			\lt(\lm_{\om,t}^n\rt)^{-1}\int_{T_\om^n(Z)} \lt(g_{\om,c,1}^{(n)}\rt)^t \circ T_{\om,Z}^{-n}\, d\nu_{\sg^n(\om),t}.
			\label{eq: LMD 5.3}
		\end{align}
		The Bounded Distortion Property implies that for $x\in Z$ we have 
		\begin{align*}
			K_\om^{-1}\nu_{\sg^n(\om),t}(T_\om^n(Z))\lt(g_{\om,c,1}^{(n)}\rt)^t(x)
			&\leq 
			\int_{T_\om^n(Z)} \lt(g_{\om,c,1}^{(n)}\rt)^t \circ T_{\om,Z}^{-n}\, d\nu_{\sg^n(\om),t}
			\\
			&\leq 
			K_\om\nu_{\sg^n(\om),t}(T_\om^n(Z))\lt(g_{\om,c,1}^{(n)}\rt)^t(x).
		\end{align*} 
		Thus 
		\begin{align*}
			K_\om^{-1}\frac{\lt(g_{\om,c,1}^{(n)}\rt)^t(x)}{\lm_{\om,t}^n\nu_{\om,t}(Z)}
			\leq
			\frac{1}{\nu_{\sg^n(\om),t}(T_\om^n(Z))}
			\leq
			K_\om\frac{\lt(g_{\om,c,1}^{(n)}\rt)^t(x)}{\lm_{\om,t}^n\nu_{\om,t}(Z)},
		\end{align*}
		which then implies that 
		\begin{align*}
			K_\om^{-1}\frac{1}{\nu_{\sg^n(\om),t}(T_\om^n(Z))}
			\leq 
			\frac{\lt(g_{\om,c,1}^{(n)}\rt)^t(x)}{\lm_{\om,t}^n\nu_{\om,t}(Z)}
			\leq
			K_\om\frac{1}{\nu_{\sg^n(\om),t}(T_\om^n(Z))}.
		\end{align*}
		The first claim follows from \eqref{eq: measur cons of LI wrt H}.
		The proof of the second claim involving the Lebesgue measure follows similarly noting that  $\nu_{\om,c,1}=\Leb$ and $\lm_{\om,c,1}=1$.
	\end{proof}
	
	Let $\ep>0$ and $n\in\NN$ such that $\diam(Z)<\ep$ for all $Z\in\cZ_\om^{(n)}$. Denote 
	\begin{align*}
		\cF_\om:=\set{Z\in\cZ_\om^{(n)}: Z\cap X_{\om,\infty}\neq\emptyset},
	\end{align*}
	which is a cover of $X_{\om,\infty}$ by sets of diameter less than $\ep$. Then, letting $x_Z$ be any element of $Z$ and using Lemma~\ref{lem: bf lem 2} twice (first with respect to $\Leb$ and then with respect to $\nu_{\om,t}$), we have 
	\begin{align}
		\sum_{Z\in\cF_\om}\diam^t(Z)
		&\leq 
		\sum_{Z\in\cF_\om}\Leb^t(Z)
		\leq 
		K_\om^t \sum_{Z\in\cF_\om}\lt(g_{\om,c,1}^{(n)}\rt)^t(x_Z)
		\nonumber\\
		&\leq 
		K_\om^{2t}\sum_{Z\in\cF_\om} \lm_{\om,t}^n\nu_{\om,t}(Z)
		=
		K_\om^{2t} \lm_{\om,t}^n\nu_{\om,t}(X_{\om,\infty})
		=
		K_\om^{2t} \lm_{\om,t}^n.
		\label{eq: HD up bd 1}	
	\end{align}
	Now, if $t>h$ we have 
	\begin{align*}
		\lim_{n\to\infty}\frac{1}{n}\log\lm_{\om,t}^n
		=
		\cEP(t)<0,
	\end{align*}
	and thus, for $\dl>0$ sufficiently small and all $n\in\NN$ sufficiently large, 
	\begin{align*}
		\lm_{\om,t}^n\leq e^{-n\dl}.
	\end{align*}
	Consequently, we see that the right-hand side of \eqref{eq: HD up bd 1} must go to zero, and thus we must have that $\HD(X_{\om,\infty})\leq h$.
	
	For the lower bound we turn to the following result of Young. 
	\begin{proposition}[Proposition, \cite{young_dimension_1982}]\label{prop: Young prop}
		Let $X$ be a metric space and $Z\sub X$. Assume there exists a probability measure $\mu$ such that $\mu(Z)>0$. For any $x\in Z$ we define 
		\begin{align*}
			\ul{d_\mu}(x):=\liminf_{\ep\to 0}\frac{\log \mu(B(x,\ep))}{\log\ep}.
		\end{align*}
		If $\ul{d_\mu}(x)\geq d$ for each $x\in Z$, then $\HD(Z)\geq d$.
	\end{proposition}
	We will use this result to prove a lower bound for the dimension, thus completing the proof of Theorem~\ref{thm: Bowen's Formula}. 
	Let $x\in X_{\om,\infty}$, let $\ep>0$, and in light of Lemma~\ref{lem: bf lem 2}, let $n_{\om,0}+1$ be the least positive integer such that there exists $y_0\in B(x,\ep)$ such that 
	\begin{align*}
		g_{\om,c}^{(n_{\om,0}+1)}(y_0)\leq 2\ep K_{\om}.
	\end{align*}
	Note that as $\ep\to 0$ we must have that $n_{\om,0}\to\infty$.
	So we must have
	\begin{align}\label{eq: bf lb eq 1}
		g_{\om,c}^{(n_{\om,0})}(y_0)g_{\sg^{n_{\om,0}}(\om),c}(T_\om^{n_{\om,0}}(y_0))=g_{\om,c}^{(n_{\om,0}+1)}(y_0)
		\leq
		2\ep K_\om.
	\end{align}
	Thus, using \eqref{eq: bf lb eq 1} and the definition of $n_{\om,0}$ we have that 
	\begin{align*}
		2\ep K_\om
		<
		g_{\om,c}^{(n_{\om,0})}(y_0)
		\leq 
		\frac{2\ep K_\om}{\inf g_{\sg^{n_{\om,0}}(\om),c}}.
	\end{align*}
	Now let $Z_{0}\in\cZ_\om^{(n_{\om,0})}$ be the partition element containing $y_0$. Then Lemma~\ref{lem: bf lem 2} gives that
	\begin{align}\label{eq: bf lb eq 2}
		\diam(Z_0)
		\leq
		K_\om g_{\om,c}^{(n_{\om,0})}(y_0)
		\leq 
		\frac{2\ep K_\om^2}{\inf g_{\sg^{n_{\om,0}}(\om),c}},	
	\end{align}
	and 
	\begin{align}\label{eq: bf lb eq 3}
		\diam(Z_0)
		\geq 
		K_\om^{-1} g_{\om,c}^{(n_{\om,0})}(y_0)
		>2\ep.
	\end{align}
	Combining \eqref{eq: bf lb eq 2} and \eqref{eq: bf lb eq 3} gives 
	\begin{align}\label{eq: bf lb eq 4}
		2\ep < \diam(Z_0) \leq \frac{2\ep K_\om^2}{\inf g_{\sg^{n_{\om,0}}(\om),c}}. 
	\end{align}
	Now, we define 
	\begin{align*}
		B_{\om,1}:=B(x,\ep)\bs Z_0,
	\end{align*}
	which may be empty. If $B_{\om,1}\neq\emptyset$, then we let $n_{\om,1}+1$ be the least positive integer such that there exists $y_1\in B_{\om,1}$ such that 
	\begin{align*}
		g_{\om,c}^{(n_{\om,1}+1)}(y_1)\leq 2\ep K_{\om}.
	\end{align*}
	Following the same line of reasoning to derive \eqref{eq: bf lb eq 4}, we see that  
	\begin{align}\label{eq: bf lb eq 5}
		2\ep
		<
		\diam(Z_1)
		\leq
		\frac{2\ep K_\om^2}{\inf g_{\sg^{n_{\om,1}}(\om),c}},
	\end{align}  
	where $Z_1\in\cZ_\om^{(n_{\om,1})}$ is the partition element containing $y_1$. Note that by definition we have that $n_{\om,1}\geq n_{\om,0}$ and $Z_0\cap Z_1=\emptyset$. This immediately implies that 
	$$
		B(x,\ep)\sub Z_0\cup Z_1,
	$$
	as otherwise using the same construction we could find some $y_2\in B_{\om,1}\bs Z_1$, some $n_{\om,2}\geq n_{\om,1}$ and a partition element $Z_2\in\cZ_\om^{(n_{\om,2})}$ containing $y_2$ with diameter greater than $2\ep$. But this would produce three disjoint intervals each with diameter greater than $2\ep$ all of which intersect $B(x,\ep)$, which would obviously be a contradiction.  
	
	Now, using \eqref{eq: LMD 5.3} and Lemma~\ref{lem: bf lem 2} gives that 
	\begin{align*}
		\nu_{\om,t}(Z_j)
		=
		\lt(\lm_{\om,t}^{n_{\om,j}}\rt)^{-1}\int_{T_\om^{n_{\om,j}}(Z)} \lt(g_{\om,c,1}^{(n_{\om,j})}\rt)^t \circ T_{\om,Z}^{-n_{\om,j}}\, d\nu_{\sg^{n_{\om,j}}(\om),t}
		\leq 
		K_\om^t\lt(\lm_{\om,t}^{n_{\om,j}}\rt)^{-1}\diam^t(Z_j)
	\end{align*}
	for $j\in\set{0,1}$.
	Using this we see that 
	\begin{align}
		&\log\nu_{\om,t}(B(x,\ep))
		\leq 
		\log\lt(\nu_{\om,t}(Z_0)+\nu_{\om,t}(Z_1)\rt)
		\nonumber\\
		&\,\,
		\leq 
		t\log K_\om
		+
		\log\lt(\lt(\lm_{\om,t}^{n_{\om,0}}\rt)^{-1}\diam^t(Z_0)
		+
		\lt(\lm_{\om,t}^{n_{\om,1}}\rt)^{-1}\diam^t(Z_1)\rt)
		\nonumber\\
		&\,\,
		\leq
		t\log K_\om
		+
		\log\lt(\lt(\lm_{\om,t}^{n_{\om,0}}\rt)^{-1}\left(\frac{2\ep K_\om^2}{\inf g_{\sg^{n_{\om,0}}(\om),c}}\right)^t
		+
		\lt(\lm_{\om,t}^{n_{\om,1}}\rt)^{-1}\left(\frac{2\ep K_\om^2}{\inf g_{\sg^{n_{\om,1}}(\om),c}}\right)^t\rt)
		\nonumber\\
		&\,\,
		=
		t\log K_\om
		+
		t\log 2\ep K_\om^2
		+
		\log\lt(\lt(\lm_{\om,t}^{n_{\om,0}}\rt)^{-1}\left(\inf g_{\sg^{n_{\om,0}}(\om),c}\right)^{-t}
		+
		\lt(\lm_{\om,t}^{n_{\om,1}}\rt)^{-1}\left(\inf g_{\sg^{n_{\om,1}}(\om),c}\right)^{-t}\rt).		
		\label{eq: bowen lb eq 0}
	\end{align}
	Now since $\log\inf g_{\om,c}\in L^1(m)$, $\inf g_{\om,c}$ is tempered and thus for each $\dl>0$ and all $n\in\NN$ sufficiently large we have that 
	\begin{align}\label{eq: bowen lb eq 1}
		e^{-nt\dl}\leq \inf \lt(g_{\sg^n(\om),c}\rt)^t.
	\end{align}
	From \eqref{eq: BET exp pres hole} we get that there for all $n\in\NN$ sufficiently large 
	\begin{align}\label{eq: bowen lb eq 2}
		\lm_{\om,t}^n\geq e^{n(\cEP(t)-\dl)}.
	\end{align}
	Thus combining \eqref{eq: bowen lb eq 1} and \eqref{eq: bowen lb eq 2} with \eqref{eq: bowen lb eq 0} gives 
	\begin{align}
		\log\nu_{\om,t}(B(x,\ep))
		&\leq 
		t\log K_\om
		+
		t\log 2\ep K_\om^2
		+
		\log\lt(e^{n_{\om,0}\dl(t+1)-n_{\om,0}\cEP(t)}+e^{n_{\om,1}\dl(t+1)-n_{\om,1}\cEP(t)}\rt)
		\nonumber\\
		&\leq 
		t\log K_\om
		+
		t\log 2\ep K_\om^2
		+
		\log\lt(e^{2n_{\om,0}\dl-n_{\om,0}\cEP(t)}+e^{2n_{\om,1}\dl-n_{\om,1}\cEP(t)}\rt),
		\label{eq: bowen lb eq 3}
	\end{align}
	where we have used the fact that $t\in[0,1]$.
	Then for $\dl>0$ sufficiently small and $n_{\om,0}$ and $n_{\om,1}$ sufficiently large (which requires $\ep>0$ to be sufficiently small) we have that 
	\begin{align}\label{eq: bowen lb eq 3*}
		\log\lt(e^{2n_{\om,0}\dl-n_{\om,0}\cEP(t)}+e^{2n_{\om,1}\dl-n_{\om,1}\cEP(t)}\rt)<0,
	\end{align} 
	since for all $t<h$ we have that $\cEP(t)>0$.
	Dividing both sides of \eqref{eq: bowen lb eq 3} by $\log\ep<0$ and using \eqref{eq: bowen lb eq 3*} yields
	\begin{align}
		\frac{\log\nu_{\om,t}(B(x,\ep))}{\log\ep}
		&\geq 
		t\frac{\log K_\om}{\log\ep}
		+
		t\frac{\log 2\ep K_\om^2}{\log\ep}
		+
		\frac{\log\lt(e^{2n_{\om,0}\dl-n_{\om,0}\cEP(t)}+e^{2n_{\om,1}\dl-n_{\om,1}\cEP(t)}\rt)}{\log\ep}
		\nonumber\\
		&\geq
		t\frac{\log K_\om}{\log\ep}
		+
		t\frac{\log 2 K_\om^2}{\log\ep}
		+t.	
		\label{eq: bowen lb eq 4}
	\end{align}
	Taking a liminf of \eqref{eq: bowen lb eq 4} as $\ep$ goes to $0$ gives that $\ul{d_{\nu_{\om,t}}}(x)\geq t$ for all $x\in X_{\om,\infty}$. As this holds for all $t<h$, we must in fact have that $\ul{d_{\nu_{\om,t}}}(x)\geq h$. In light of Proposition~\ref{prop: Young prop}, we have proven Theorem~\ref{thm: Bowen's Formula}.
\end{proof}

\section{Examples}\label{sec: examples}
In this final section we present several examples of our general theory of open random systems.
In particular, we show that our results hold for a large class of random $\bt$-transformations with random holes which have uniform covering times as well as a large class of random Lasota-Yorke maps with random holes. However, we note that in principle any of the finitely branched classes of maps treated in \cite{AFGTV21} will satisfy our assumptions given a suitable choice of hole. This includes random systems where we allow non-uniformly expanding maps, or even maps with contracting branches to appear with positive probability. We also note that the examples we present allow for both random maps and random hole, which, to the authors' knowledge, has not appeared in literature until now. 
Before presenting our examples, we first give alternate hypotheses (to our assumptions \eqref{cond Q1}-\eqref{cond Q3}) that are more restrictive but simpler to check. 

We begin by recalling the definitions of the various partitions constructed in Section~\ref{sec: tr op and Lm} which are used in producing our main Lasota-Yorke inequality (Lemma~\ref{ly ineq}) and are implicitly apart of our main assumptions \eqref{cond Q1}-\eqref{cond Q3}.
Recall that $\cZ_\om^{(n)}$ denotes the partition of monotonicity of $T_\om^n$, and $\sA_\om^{(n)}$ denotes the collection of all finite partitions of $I$ such that
\begin{align}\label{eq: def A partition 2}
	\var_{A_i}(g_{\om}^{(n)})\leq 2\norm{g_{\om}^{(n)}}_{\infty}
\end{align}
for each $\cA=\set{A_i}\in\sA_{\om}^{(n)}$.
Given $\cA\in\sA_\om^{(n)}$, $\widehat\cZ_\om^{(n)}(\cA)$ denotes the coarsest partition amongst all those finer than $\cA$ and $\cZ_\om^{(n)}$ such that all elements of $\widehat\cZ_\om^{(n)}(\cA)$ are either disjoint from $X_{\om,n-1}$ or contained in $X_{\om,n-1}$. From $\widehat\cZ_\om^{(n)}$ we recall the subcollections $\cZ_{\om,*}^{(n)}$, $\cZ_{\om,b}^{(n)}$, and $\cZ_{\om,g}^{(n)}$ defined in \eqref{eq: Z_*}-\eqref{eq: Z_g}.

For the purposes of showing that examples easily satisfy our conditions, we take the more general approach to partitions found in Section~2.2 of \cite{AFGTV21}, and instead now set, for $\hat\al\geq 0$,
$\ol\sA_\om^{(n)}(\hat\al)$ to be the collection of all finite partitions of $I$ such that
\begin{align}\label{eq: def A partition 3}
	\var_{A_i}(g_{\om}^{(n)})\leq \hat\al\norm{g_{\om}^{(n)}}_{\infty}
\end{align}
for each $\cA=\set{A_i}\in\ol\sA_{\om}^{(n)}(\hat\al)$. Note that for some $\hat\al\leq 1$ the collection $\ol\sA_\om^{(n)}(\hat\al)$ may be empty, but such partitions always exist for any $\hat\al>1$, and may exist even with $\hat\al=0$ if the weight function $g_\om$ is constant; see \cite{rychlik_bounded_1983} Lemma~6. 
We now suppose that we can find $\hat\al\geq 0$ sufficiently large such that 
\begin{enumerate}
	\item[\mylabel{Z}{cond Z}] $\cZ_\om^{(n)}\in\ol\sA_\om^{(n)}(\hat\al)$ for each $n\in\NN$ and each $\om\in\Om$.  
\end{enumerate}
Now we set $\ol\cZ_\om^{(n)}$ be the coarsest partition such that all elements of $\cZ_\om^{(n)}$ are either disjoint from $X_{\om,n-1}$ or contained in $X_{\om,n-1}$. Note that $\ol\cZ_\om^{(n)}=\widehat\cZ_\om^{(n)}(\cZ_\om^{(n)})$. 
Now, define the following subcollections:
\begin{align*}
	\ol\cZ_{\om,*}^{(n)}&:=\set{Z\in \ol\cZ_\om^{(n)}: Z\sub X_{\om,n-1} },
	\\
	\ol\cZ_{\om,b}^{(n)}&:=\set{Z\in \ol\cZ_\om^{(n)}: Z\sub X_{\om,n-1} \text{ and } \Lm_{\om}(\ind_Z)=0 },
	\\
	\ol\cZ_{\om,g}^{(n)}&:=\set{Z\in \ol\cZ_\om^{(n)}: Z\sub X_{\om,n-1} \text{ and } \Lm_{\om}(\ind_Z)>0}.
\end{align*}

Consider the collection $\ol\cZ_{\om,F}^{(n)}\sub\ol\cZ_{\om,*}^{(n)}$ such that for $Z\in\ol\cZ_{\om,F}^{(n)}$ we have $T_\om^n(Z)=I$. We will elements of $\ol\cZ_{\om,F}^{(n)}$ ``full intervals''. We let $\ol\cZ_{\om,U}^{(n)}:=\ol\cZ_{\om,*}^{(n)}\bs\ol\cZ_{\om,F}^{(n)}$. Since for any $Z\in\ol\cZ_{\om,F}^{(n)}$ we have that $T_\om^n(Z)=I$, and hence 
\begin{align}\label{eq: check tilde Q3}
	\Lm_\om(\ind_Z)
	\geq 
	\frac{\inf_{D_{\sg^n(\om),n}}\cL_\om^n\ind_Z}{\sup_{D_{\sg^n(\om),n}}\cL_\om^n\ind_\om}	
	\geq 
	\frac{\inf g_{\om,c}^{(n)}}{\norm{\cL_{\om,c}^n\ind}_\infty}>0.
\end{align}
Consequently, we have that $\ol\cZ_{\om,F}^{(n)}\sub\ol\cZ_{\om,g}^{(n)}$, and thus 
\begin{align}\label{eq: bad interv cont in non-full interv}
	\ol\cZ_{\om,b}^{(n)}\sub\ol\cZ_{\om,U}^{(n)}.
\end{align} 
We let $\zt_\om^{(n)}\geq 0$ denote the maximum number of contiguous non-full intervals for $T_\om^n$ in $\ol\cZ_{\om,*}^{(n)}$ for each $\om\in\Om$ and $n\in\NN$. Note that $\zt_\om^{(1)}$ may be equal to $0$, but $\zt_\om^{(n)}\geq 1$ for all $n\geq 2$, and so it follows from \eqref{eq: bad interv cont in non-full interv} that 
\begin{align}\label{eq: xi leq zt}
	0\leq \log \xi_\om^{(n)}\leq \log \zt_\om^{(n)}
\end{align}
for all $n\geq 2$.
In the interest of having assumptions that are easier to check than \eqref{cond Q1}-\eqref{cond Q3} we introduce the following simpler assumptions which use the collections $\ol\cZ_{\om,F}^{(n)}$ and $\ol\cZ_{\om,U}^{(n)}$ rather than $\cZ_{\om,g}^{(n)}$ and $\cZ_{\om,b}^{(n)}$. We assume the following:
	\begin{enumerate}
	\item[\mylabel{\^Q0}{cond tilde Q0}] $\ol\cZ_{\om,F}^{(1)}\neq\emptyset$ for $m$-a.e. $\om\in\Om$, 
	\item[\mylabel{\^Q1}{cond tilde Q1}] 
	\begin{align*}
		\lim_{n\to\infty}\frac{1}{n}\log\norm{g_{\om}^{(n)}}_\infty
		+
		\limsup_{n\to\infty}\frac{1}{n}\log\zt_{\om}^{(n)}
		<
		\lim_{n\to\infty}\frac{1}{n}\log\rho_{\om}^n
		=
		\int_\Om \log\rho_{\om}\, dm(\om),
	\end{align*}
	
	\item[\mylabel{\^Q2}{cond tilde Q2}] for each $n\in\NN$ we have $\log^+\zt_\om^{(n)}\in L^1(m)$,
	
	\item[\mylabel{\^Q3}{cond tilde Q3}] for each $n\in\NN$, $\log\hat\dl_{\om,n}\in L^1(m)$, where
	\begin{align}\label{eq: def of tilde dl_om,n}
		\hat\dl_{\om,n}:=\min_{Z\in\ol\cZ_{\om,F}^{(n)}}\Lm_\om(\ind_Z).
	\end{align}
\end{enumerate}
Our assumptions \eqref{cond Q1}-\eqref{cond Q3} are used exclusively in Section~\ref{sec: LY ineq}, and primarily in Lemma~\ref{ly ineq}. In the proof of Lemma~\ref{ly ineq}, the good and bad interval collections $\cZ_{\om,g}^{(n)}$ and $\cZ_{\om,b}^{(n)}$ are used only to estimate the variation of a function and can easily be replaced by the collections $\ol\cZ_{\om,F}^{(n)}$ and $\ol\cZ_{\om,U}^{(n)}$. Therefore, we can easily replace the assumptions \eqref{cond Q1}-\eqref{cond Q3} with \eqref{cond tilde Q0}-\eqref{cond tilde Q3} without any changes. In particular, we are still able to construct the number $N_*$ which is defined in \eqref{eq: def of N}. Note that by replacing the $2$ in \eqref{eq: def A partition 2} with the $\hat\al\geq 0$ that appears in \eqref{eq: def A partition 3}, the constant coefficients which appear in the definitions of $A_\om^{(n)}$ and $B_\om^{(n)}$ in \eqref{eq: def of A and B in ly ineq} at the end of Lemma~\ref{ly ineq} may be different, consequently changing the value of $N_*$. This ultimately does not affect our general theory as we only care that such a value exists, though if one wishes to check the simplified assumptions \eqref{cond Q2'} and \eqref{cond Q3'} a value $\hat\al<2$, and thus a smaller value for $N_*$, is helpful. 
\begin{remark}
	As in Remark~\ref{rem: alternate hypoth}, we again note that checking \eqref{cond Z}, \eqref{cond tilde Q2}, and \eqref{cond tilde Q3} for all $n\in\NN$ could be difficult and that it suffices to instead check these conditions only for $n=N_*$. Thus we may replace \eqref{cond Z}, \eqref{cond tilde Q2}, and \eqref{cond tilde Q3} with the following: 
	\begin{enumerate}
		\item[\mylabel{Z'}{cond Z'}] there exists $\hat\al\geq 0$ such that $\cZ_\om^{(N_*)}\in\ol\sA_\om^{(N_*)}(\hat\al)$ for each $\om\in\Om$,
		\item[\mylabel{\^ Q2'}{cond tilde Q2'}] $\log^+\zt_\om^{(N_*)}\in L^1(m)$,
		
		\item[\mylabel{\^ Q3'}{cond tilde Q3'}] $\log\hat\dl_{\om,N_*}\in L^1(m)$.
	\end{enumerate}
\end{remark}
The following proposition gives that assumption \eqref{cond tilde Q3} is always satisfied, and thus that we really only need to assume \eqref{cond tilde Q1} and \eqref{cond tilde Q2}.
\begin{proposition}\label{prop: tilde Q3 trivial}
	Assumption \eqref{cond tilde Q3} is trivially satisfied.
\end{proposition}
\begin{proof}
	As the right hand side of \eqref{eq: check tilde Q3} is $\log$-integrable by \eqref{eq: log sup inf g integ} and \eqref{eq: log sup inf L integ}, we must also have $\log\hat\dl_{\om,n}\in L^1(m)$.
\end{proof}
Recall that two elements $W, Z\in\cZ_{\om,*}^{(n)}$ are said to be contiguous if either $W$ and $Z$ are contiguous in the usual sense, i.e. they share a boundary point, or if they are separated by a connected component of $\cup_{j=0}^{n-1} T_\om^{-j}(H_{\sg^j(\om)})$.
The following proposition gives an upper bound for the exponential growth of the number $\zt_\om^{(n)}$ which will be useful in checking our assumption \eqref{cond tilde Q1},  which implies \eqref{cond Q1}.
\begin{proposition} 
	The following inequality holds for $\zt_{\om}^{(n)}$, the largest number of contiguous
	non-full intervals for $T_\om^n$. 
	\begin{align}\label{eq: zt submultish}
	\zt_{\om}^{(n)}
	\leq n \prod_{j=0}^{n-1} (\zt^{(1)}_{\sg^j(\om)}+2).
	\end{align}
	Consequently, using \eqref{eq: xi leq zt} and the ergodic theorem, we have that 
	\begin{align}\label{eq: zt bet}
		\lim_{n\to\infty}\frac{1}{n}\log\xi_\om^{(n)}
		\leq
		\lim_{n\to\infty}\frac{1}{n}\log\zt_\om^{(n)}
		\leq 
		\int_\Om\log(\zt_\om^{(1)}+2)\, dm(\om).
	\end{align}
\end{proposition}
\begin{proof}
	This is  a random version of \cite[Lemma 6.3]{LMD}. We sketch the argument here.
	To upper bound $\zt_{\om}^{(n+1)}$, we observe that the largest number of contiguous non-full intervals for $T_{\om}^{n+1}$ is given by 
	\begin{equation}\label{eq:nonfullrec}
		\zt_{\om}^{(n+1)}\leq \zt_{\om}^{(1)} (\zt^{(n)}_{\sg(\om)}+2)+ 2 \zt^{(n)}_{\sg(\om)}.
	\end{equation}
	Indeed, the first term on the right hand side accounts for the (worst case) scenario that all non-full branches of $T_{\sg(\om)}^n$ are pulled back inside contiguous non-full intervals for $T_\om$. 
	For each non-full interval of $T_\om$, there at most $\zt^{(n)}_{\sg(\om)}+2$ contiguous non-full intervals for $T_\om^{n+1}$, as in addition to the $\zt^{(n)}_{\sg(\om)}$ non-full intervals pulled back from $T_{\sg(\om)}^n$, there may be full branches of $T_{\sg(\om)}^n$ to the left and right of these which are only partially pulled back inside the corresponding branch of  $T_\om$.
	The second term in \eqref{eq:nonfullrec} accounts for an extra (at most) $\zt^{(n)}_{\sg(\om)}$ non full branches of $T_{\sg(\om)}^n$ pulled back inside the full branches of $T_\om$ neighboring the cluster of $\zt_{\om}^{(1)}$ non-full branches.
	
	Rearranging \eqref{eq:nonfullrec} yields $\zt_{\om}^{(n+1)}\leq  \zt^{(n)}_{\sg(\om)}(\zt_{\om}^{(1)}+2)+ 2 \zt^{(1)}_{\om}.$
	The claim follows directly by induction.
\end{proof}
Let $\pzh_\om$ denote the number of connected components of $H_\om$. The following lemma shows that the conditions 
\begin{lemma}\label{lem: check Q2' with zt}
	If assumption \eqref{cond Z'} holds as well as 
	\begin{align}
		\tag{CCH}\label{cond CCH}
		\log \pzh_\om\in L^1(m),
	\end{align}
	then $\log\zt_\om^{(1)}\in L^1(m)$. Consequently, we have that \eqref{cond Q2'} and \eqref{cond tilde Q2} hold.
\end{lemma}
\begin{proof}
	Since condition \eqref{cond Z'} holds, we see that \eqref{cond LIP} in conjunction with \eqref{cond CCH} gives that $\log\#\ol\cZ_{\om,*}^{(1)}\in L^1(m)$, and thus we must have that $\log\zt_\om^{(1)}\in L^1(m)$. 	
	In light of \eqref{eq: zt submultish} we see that \eqref{cond Q2'} and \eqref{cond tilde Q2} hold if 
	$$
	\log(\zt_\om^{(1)}+2)\in L^1(m),
	$$
	and thus we are done.
\end{proof}
For each $n\in\NN$ define 
\begin{align}\label{eq: def of F_om^n}
	F_\om^{(n)}:=\min_{y\in [0,1]}\#\{T^{-n}_\omega (y)\}. 
\end{align}
Since the sequences $\{\om \mapsto \|g_\om^{(n)}\|_\infty\}_{n\in\NN}$ and $\set{\om \mapsto \inf \cL_\om^n \ind_\om}_{n\in\NN}$  are submultiplicative and supermultiplicative, respectively, the subadditive ergodic theorem implies that the assumption that 
\begin{align*}
	\lim_{n\to\infty}\frac{1}{n}\log\norm{g_\om^{(n)}}_\infty 
	<
	\lim_{n\to\infty}\frac{1}{n}\log\inf_{D_{\sg^n(\om),n}}\cL_\om^n\ind_\om
\end{align*}
is equivalent to the assumption that there exist $N_1,N_2 \in \NN$ such that
\begin{align}\label{eq: on avg CPN1N2}
	\frac{1}{N_1}\int_\Om\log \norm{g_\om^{(N_1)}}_\infty\, dm(\om) 
	<  
	\frac{1}{N_2}\int_\Om\log \inf_{D_{\sg^{N_2}(\om),N_2}} \cL_\om^{N_2}\ind_\om\, dm(\om).
\end{align}
A useful lower bound for the right hand side is the following:
\begin{align}\label{eq: example calc inf up bound}
	\inf_{D_{\sg^{N_2}(\om),N_2}} \cL_\om^{N_2}\ind_\om
	\geq
	\inf_{X_{\om,N_2-1}}g_\om^{(N_2)}F_\om^{(N_2)}
	\geq 
	\inf g_{\om,c}^{(N_2)} F_\om^{(N_2)}.
\end{align}
The next lemma, which offers a sufficient condition to check assumptions \eqref{cond Q1} and \eqref{cond tilde Q1}, follows from \eqref{eq: zt bet}, \eqref{eq: on avg CPN1N2}, \eqref{eq: example calc inf up bound}, and the calculations from the proof of Lemma~13.16 in \cite{AFGTV21}. 
\begin{lemma}\label{lem: LY example Q1 Lemma}
	If there exists $N_1\geq N_2\in\NN$ such that 
	\begin{align*}
		\frac{1}{N_1}\int_\Om \sup S_{N_1, T}(\phi_{\om,c})-\inf S_{N_1, T}(\phi_{\om,c})+ N_1\log(\zt_\om^{(1)}+2)\, dm(\om)
		<
		\frac{1}{N_2}\int_\Om \log F_\om^{(N_2)}\, dm(\om),
	\end{align*}	 
	then \eqref{cond tilde Q1} (and thus \eqref{cond Q1}) holds.
\end{lemma}

The following definition will be useful in checking our measurability assumptions for examples. 
\begin{definition}
	We say that a function $f: \Omega \to \RR_+$ is $m$-continuous function if there is a partition of $\Om \pmod{m}$  into at most countably many Borel sets $\Om_1, \Om_2, \dots$ such that $f$ is constant on each $\Om_j$, say $f|_{\Om_j}=f_j$.
\end{definition}

We now give specific classes of random maps with holes which meet our assumptions. In principle, any of the classes of finitely branched maps discussed in Section~13 of \cite{AFGTV21} (including random non-uniformly expanding maps) will fit our current assumptions given a suitable hole $H$. 
\subsection{Random $\bt$-Transformations With Holes}

For this first example we consider the class of maps described in Section~13.2 in \cite{AFGTV21}. These are $\bt$-transformations for which the last (non-full) branch is not too small so that each branch in the random closed system has a uniform covering time. In particular we assume there is some $\dl>0$ such that for $m$-a.e. $\om\in\Om$ we have 
$$	
\bt_\om\in 
\bigcup_{k=1}^\infty [k+\delta,k+1].
$$
Further suppose that the map $\om\mapsto \bt_\om$ is $m$-continuous. 
We consider the random $\bt$-transformation
$T_\om : [0,1] \to [0,1]$ given by
\[
T_\om(x) = \beta_\omega x \pmod{1}.
\]
and the potential 
$$
	\phi_{\om,c}=-t\log|T'_\om|=-t\log\bt_\om
$$ 
for $t\geq 0$. 
In addition, we assume that 
\begin{equation}
	\label{log9}
	\int_\Om \log\lfloor\beta_\omega\rfloor\ dm(\omega)>\log 3
\end{equation}
and
\begin{equation}
	\label{blb}
	\int_\Om\log\lceil\beta_\omega\rceil\ dm(\omega)<\infty.
\end{equation}
Note that we allow $\bt_\om$ arbitrarily large. It follows from Lemma~13.6 of \cite{AFGTV21} that our assumptions \eqref{cond T1}-\eqref{cond T3}, \eqref{cond LIP}, \eqref{cond GP}, \eqref{cond A1}-\eqref{cond A2}, \eqref{cond M1}, \eqref{cond C1}, and \eqref{cond Z} are satisfied. 

To check the remainder of our assumptions we must now describe the choice of hole $H_\om$. For our holes $H_\om$ we will consider intervals of length at most $1/\bt_\om$ so that $H_\om$ may not intersect more than two monotonicity partition elements. To ensure that \eqref{cond tilde Q0} is satisfied we assume there is a  full-branched element $Z\in\cZ_{\om,F}^{(1)}$ such that $Z\cap H_\om=\emptyset$ for each $\om\in\Om$, and thus, in light of Remark~\ref{rem: check cond D}, we also have that assumption \eqref{cond D} is satisfied with $D_{\om,\infty}=I$ for each $\om\in\Om$. 

Now, we note that since \eqref{blb} implies our assumption \eqref{cond LIP}, Lemma~\ref{lem: check Q2' with zt} implies that assumption \eqref{cond tilde Q2} is satisfied. Thus, we have only to check the condition \eqref{cond tilde Q1}. 
Depending on $H_\om$ we may have that 
\begin{align*}
	\inf\cL_\om\ind_\om=\frac{\floor{\bt_\om}-1}{\bt_\om^t},
\end{align*}
for example if $H_\om$ is the last full branch.
To ensure that \eqref{cond tilde Q1} holds, note that \eqref{eq: on avg CPN1N2} holds with $N=1$, and thus it suffices to have 
\begin{align*}
	\int_\Om \log(\floor{\bt_\om}-1)\, dm(\om)
	> 
	\int_\Om \log(\zt_\om^{(1)}+2)\, dm(\om),
\end{align*}
since 
\begin{align*}
	\int_{\Om}\log\inf\cL_\om\ind_\om 
	-
	\log\norm{g_\om}_\infty\, dm(\om)
	\geq 
	\int_\Om \log\frac{\floor{\bt_\om}-1}{\bt_\om^t}+\log\bt_\om^t \, dm(\om).
\end{align*}
Depending on the placement of $H_\om$ we may have $\zt_\om^{(1)}=i$ for any $i\in\set{0,1,2,3}$.
Thus, we obtain the following lemma assuming the worst case scenario, i.e. assuming $\zt_\om^{(1)}=3$ for $m$-a.e. $\om\in\Om$. 
\begin{lemma}
	If $H_\om\sub I$ is such that $\zt_\om^{(1)}\leq 3$ for $m$-a.e. $\om\in\Om$, then Theorems~\ref{main thm: existence}-\ref{main thm: escape rate} if 
	\begin{align*}
		\int_\Om \log(\floor{\bt_\om}-1)\, dm(\om)
		> 
		\log 5.
	\end{align*}
\end{lemma}
On the other hand, if we have that $H_\om$ is equal to the monotonicity partition element which contains $1$, then $\zt_\om^{(1)}=0$ and 
\begin{align*}
	\inf\cL_\om\ind_\om=\frac{\floor{\bt_\om}}{\bt_\om^t}.
\end{align*}
Furthermore, the additional hypotheses necessary for Theorem~\ref{main thm: Bowens formula} are satisfied. In particular, the fact that $T$ has large images follows from the fact that these maps have a uniform covering time; see Lemma~13.5 of \cite{AFGTV21}.
Thus, we thus have the following lemma. 
\begin{lemma}
	If $H_\om=Z_1$, where $1\in Z_1\in\cZ_\om$, for $m$-a.e. $\om\in\Om$ then Theorems~\ref{main thm: existence}-\ref{main thm: Bowens formula} hold.
\end{lemma}

More generally, we can consider general potentials, non-linear maps, and holes which are unions of finitely many intervals so that condition \eqref{cond CCH} holds.

\subsection{Random Open Lasota-Yorke Maps}
We now present an example of a large class of random Lasota-Yorke maps with holes. The following lemma summarizes the closed setting for this particular class of random maps was treated in Section 13.6 of \cite{AFGTV21}.

\begin{lemma}\label{lem: closed LY example}
	Let $\Phi_c:\Om \to BV(I)$ be an $m$-continuous function, and let
	$\phi_c:\Om \times I \to \RR$ be given by $\phi_{\om,c}:=-t\log|T_\om'|=\Phi_c(\om)$ for $t\geq 0$. Then 
	$g_{\om,c} = e^{\phi_{\om,c}}=1/|T_\om'|^t\in BV(I)$ for \maeom. We further suppose the system satisfies the following: 
	\begin{enumerate} 
		\item[\mylabel{1}{hyp 1}] $\log\#\cZ_\om\in L^1(m)$,
		\item[\mylabel{2}{hyp 2}] there exists $M(n)\in\NN$ such that for any $\om\in\Om$ and any $Z\in\cZ_\om^{(n)}$ we have that $T_\om^{M(n)}(Z)=I$,	
		\item[\mylabel{3}{hyp 3}] for each $\om\in\Om$, $Z\in\cZ_\om$, and $x\in Z$ 
		\begin{enumerate}
			\item $T_\om\rvert_Z\in C^2$, 
			\item there exists $K\geq 1$ such that
			\begin{align*}
				\frac{|T_\om''(x)|}{|T_\om'(x)|}\leq K,
			\end{align*}
		\end{enumerate}	
		\item[\mylabel{4}{hyp 4}] there exist $1< \lm\leq \Lm<\infty$ and $n_1\geq n_2\in\NN$ such that 
		\begin{enumerate}
			\item[\mylabel{a}{hyp 4a}] $|T_\om'|\leq \Lm$ for $m$-a.e. $\om\in\Om$,
			\item[\mylabel{b}{hyp 4b}] $|(T_\om^{n_1})'|\geq \lm^{n_1}$ for $m$-a.e. $\om\in\Om$
			\item[\mylabel{c}{hyp 4c}] $\frac{1}{n_2}\int_\Om \log F_\om^{(n_2)} \,dm(\om)>t\log\frac{\Lm}{\lm}$,
		\end{enumerate}
		\item[\mylabel{5}{hyp 5}] for each $n\in\NN$ there exists 
		$$
		\ep_n:=\inf_{\om\in\Om}\min_{Z\in\cZ_\om^{(n)}}\diam(Z) >0.
		$$
	\end{enumerate}
	Then Theorems~ 2.19-2.23  of \cite{AFGTV21} hold, and in particular, our assumptions \eqref{cond T1}-\eqref{cond T3}, \eqref{cond LIP}, \eqref{cond GP}, \eqref{cond A1}-\eqref{cond A2}, \eqref{cond M1}, and \eqref{cond C1} hold. 
\end{lemma}

The following lemma gives a large class of random Lasota-Yorke maps with holes for which our results apply. In particular, we allow our hole to be composed of finitely many intervals which may change depending on the fiber $\om$, provided the number of connected components of the hole is $\log$-integrable over $\Om$ \eqref{cond CCH}. 

\begin{lemma}\label{lem: open LY example}
	Let $\phi_{\om,c}=-t\log|T_\om'|$ and suppose the hypotheses of Lemma~\ref{lem: closed LY example} hold. Additionally we suppose that $H\sub\Om\times I$ such that \eqref{cond CCH} holds as well as the following:
	\begin{enumerate}
		\item[\mylabel{6}{open 1}] for $m$-a.e. $\om\in\Om$ there exists $Z\in\cZ_\om$ with $Z\cap H_\om=\emptyset$ such that $T_\om(Z)=I$,	
		\item[\mylabel{7}{open 2}] $\frac{1}{n_2}\int_\Om \log F_\om^{(n_2)} \,dm(\om)>t\log\frac{\Lm}{\lm}+\int_\Om \log(\zt_\om^{(1)}+2)\, dm(\om)$.
	\end{enumerate}
	Then the hypotheses of Theorems~\ref{main thm: existence}-\ref{main thm: escape rate} hold. If in addition we have that 
	\begin{enumerate}
		\item[\mylabel{8}{open 3}] there exists $M:\NN\to\NN$ such that $T_\om^{M(n)}(Z)=I$ for $m$-a.e. $\om\in\Om$  and each $Z\in\cZ_\om^{(n)}$, i.e. there is a uniform covering time,
		\item[\mylabel{9}{open 4}] for $m$-a.e. $\om\in\Om$ there exists $Z_1,\dots,Z_{k}\in\cZ_\om$ such that $H_\om=\cup_{j=1}^{k}Z_j$ and $T_\om(Z)=I$ for all $Z\in\cZ_\om$ such that $Z\cap H_\om=\emptyset$,
\end{enumerate}
	then the hypotheses of Theorem~\ref{main thm: Bowens formula} also hold.
\end{lemma}
\begin{proof}
	The conclusion of Lemma~\ref{lem: closed LY example} leaves only to check assumptions \eqref{cond D} and  \eqref{cond tilde Q0}-\eqref{cond tilde Q3}. But in light of Proposition~\ref{prop: tilde Q3 trivial} we see that \eqref{cond tilde Q3} holds, and hypothesis \eqref{open 1} implies \eqref{cond D} (by Remark~\ref{rem: check cond D}) and \eqref{cond tilde Q0} hold.
	
	To check our remaining hypotheses on the open system we first show that \eqref{cond Z'} holds. To see this we note that equation (13.20) of \cite{AFGTV21}, together with the fact that hypothesis \eqref{hyp 2} of Lemma~\ref{lem: closed LY example} implies that
	$\Lm^{-kn_1t}\leq g_{\om,c}^{(kn_1)}\leq \lm^{-kn_1t}<1$,  gives that for any $\om\in\Om$ and $Z\in\cZ_\om^{(kn_1)}$ we have
	\begin{align*}
		\var_Z(g_{\om,c}^{(kn_1})
		&\leq 
		2\norm{g_{\om,c}^{(kn_1)}}_\infty + \frac{tK}{\Lm-1}\cdot\lt(\frac{\Lm}{\lm^t}\rt)^{kn_1}
		\\
		&\leq 
		2\norm{g_{\om,c}^{(kn_1)}}_\infty + \frac{tK}{\Lm-1}\cdot\lt(\frac{\Lm}{\lm^t}\rt)^{kn_1}\cdot \Lm^{kn_1t}\norm{g_{\om,c}^{kn_1}}_\infty
		\\
		&\leq
		\hat\al_k\norm{g_{\om,c}^{(kn_1)}}_\infty,
	\end{align*}
	where
	\begin{align*}
		\hat\al_k:=2+\frac{tK}{\Lm-1}\cdot\lt(\frac{\Lm^{2t}}{\lm^t}\rt)^{kn_1}.
	\end{align*}
	Taking $k_*$ so large that
	\begin{align*}
		\int_\Om\log Q_{\om}^{(k_*n_1)}\, dm(\om)<0,
	\end{align*}
	where $Q_\om^{(k_*n_1)}$ is defined in \eqref{eq: def of Q and K},
	and setting $N_*=k_*n_1$, we then see that \eqref{cond Z'} holds, that is we have that $\cZ_\om^{(N_*)}\in\ol\sA_\om^{(N_*)}(\hat\al_{k_*})$ for each $\om\in\Om$. Thus, Lemma~\ref{lem: check Q2' with zt} together with \eqref{cond CCH} ensures that \eqref{cond tilde Q2} holds.  
	Now taking \eqref{open 2} in conjunction with Lemma~\ref{lem: LY example Q1 Lemma} implies assumption \eqref{cond tilde Q1}.
	
	The second claim holds since the assumptions \eqref{open 3} and \eqref{open 4} together imply that $T$ has large images and large images with respect to $H$, and assumptions \eqref{hyp 3} and \eqref{hyp 4}\eqref{hyp 4b} gives the bounded distortion condition for $g_{\om,c}$. 
	
\end{proof}

\begin{remark}
	If one wishes to work with general potentials rather than the geometric potentials in Lemmas~\ref{lem: closed LY example} and \ref{lem: open LY example} then one could replace \eqref{hyp 4} of Lemma~\ref{lem: closed LY example} with \eqref{eq: on avg CPN1N2} and \eqref{open 2} of Lemma~\ref{lem: open LY example} with Lemma~\ref{lem: LY example Q1 Lemma}.
\end{remark}

\section*{Acknowledgments}
J.A., G.F., and C.G.-T. thank the Centro de Giorgi in Pisa and CIRM in Luminy for their support and hospitality.\\
J.A. is supported by an ARC Discovery project and thanks the School of Mathematics and Physics at the University of Queensland  for their hospitality.\\
G.F., C.G.-T., and S.V. are partially supported by an ARC Discovery Project.\\
S.V. thanks the Laboratoire International Associé LIA LYSM, the INdAM (Italy), the UMI-CNRS 3483, Laboratoire Fibonacci (Pisa) where this work has been done  under a CNRS delegation and the Centro di Ricerca Matematica Ennio de Giorgi and UniCredit Bank R\&D group for financial support through the “Dynamics and Information Theory Institute” at the Scuola Normale Superiore.
\bibliographystyle{plain}
\bibliography{mybib}

\end{document}